\documentclass[10pt]{amsart}
\usepackage{euscript}
\usepackage{amssymb}
\usepackage{amsmath}
\usepackage{amsthm}
\usepackage{epic}
\usepackage{graphicx}
\usepackage{epsfig}
\usepackage{color}
\usepackage{amscd}
\usepackage{xspace}
\usepackage{stmaryrd}
\usepackage{jrsymb}

\usepackage[matrix,arrow,curve]{xy}
\usepackage{tikz}
\usetikzlibrary{calc,intersections,through}
\newlength{\lw}
\numberwithin{equation}{section}

\newtheoremstyle{my}{1.5em}{0.5em}{\em}{}{\sc}{.}{0.5em}{}
\theoremstyle{my}

\theoremstyle{my}
\newtheorem{thm}{Theorem}[section]
\newtheorem{Theorem}[thm]{Theorem}
\newtheorem*{Theorem*}{Theorem}
\newtheorem{Corollary}[thm]{Corollary}

\newtheorem{cor}[thm]{Corollary}
\newtheorem*{corollary*}{Corollary}
\newtheorem{Lemma}[thm]{Lemma}
\newtheorem{lem}[thm]{Lemma}

\newtheorem{Addendum}[thm]{Addendum}
\newtheorem{prop}[thm]{Proposition}
\newtheorem{Proposition}[thm]{Proposition}
\newtheorem{Conjecture}[thm]{Conjecture}
\newtheorem*{conjecture*}{Conjecture}

\newtheorem*{question*}{Question}
\newtheorem{defn}[thm]{Definition}
\newtheorem{Definition}[thm]{Definition}

\newtheorem{Hypothesis}[thm]{Hypothesis}
\newtheorem*{definitions*}{Definitions}
\newtheorem{rem}[thm]{Remark}
\newtheorem*{rem*}{Remark}
\newtheorem{Remark}[thm]{Remark}

\newtheorem*{remark*}{Remark}

\newtheorem*{remarks*}{Remarks}
\newtheorem*{example*}{Example}
\newtheorem{Example}[thm]{Example}
\newtheorem*{examples*}{Examples}

\newtheorem*{convention*}{Convention}
\newtheorem*{conventions*}{Conventions}
\newtheorem*{Note*}{Note}
\newtheorem*{exercise*}{Exercise}
\newtheorem*{bibliographical-note*}{Bibliographical note}

\newcommand{\Acknowledgements}{{\em Acknowledgements.} }

\def\co{\colon\thinspace}
\newcommand{\nc}{{\bf nc}}
\newcommand{\bR}{\mathbb{R}}
\newcommand{\bZ}{\mathbb{Z}}
\newcommand{\bQ}{\mathbb{Q}}
\newcommand{\bC}{\mathbb{C}}

\newcommand{\bP}{\mathbb{P}}
\newcommand{\bk}{\mathbf{k}}

\newcommand{\iso}{\cong}           %isomorphism sign
\newcommand{\cdbar}{\mathrm{\overline{\partial}}}
\newcommand{\Sym}{\mathrm{Sym}}
\newcommand{\Symp}{\mathrm{Symp}}

\newcommand{\id}{\mathrm{id}}

\newcommand{\im}{\mathrm{im}}
\renewcommand{\ker}{\mathrm{ker}}

\newcommand{\Slice}{\mathcal{S}}
\newcommand{\Hom}{\mathrm{Hom}}

\newcommand{\End}{\mathrm{End}}
\newcommand{\Kh}{\mathrm{Kh}}
\newcommand{\Br}{\mathrm{Br}}
\newcommand{\Ext}{\mathrm{Ext}}
\newcommand{\Tw}{\operatorname{Tw}}
\newcommand{\Id}{\mathrm{Id}}
\newcommand{\tr}{\mathrm{tr}}

\newcommand{\Ob}{\mathrm{Ob} \,}

\newcommand{\Hilb}{\mathrm{Hilb}}
\newcommand{\Conf}{\textrm{Conf}}

\newcommand{\reg}{\scriptscriptstyle{\mathrm{reg}}}

\newcommand{\sub}{\scriptscriptstyle{\mathrm{sub}}}

\newcommand{\scrA}{\EuScript{A}}
\newcommand{\scrB}{\EuScript{B}}
\newcommand{\scrC}{\EuScript{C}}
\newcommand{\scrD}{\EuScript{D}}
\newcommand{\scrY}{\EuScript{Y}}
\newcommand{\scrP}{\EuScript{P}}
\newcommand{\scrQ}{\EuScript{Q}}

\newcommand{\scrR}{\EuScript{R}}
\newcommand{\scrG}{\EuScript{G}}
\newcommand{\scrX}{\EuScript{X}}
\newcommand{\Spin}{\mathrm{Spin}}
\newcommand{\scrO}{\EuScript{O}}
\newcommand{\scrL}{\EuScript{L}}
\newcommand{\scrU}{\varcurlyvee}
\newcommand{\scrCap}{ \varcurlywedge }
\newcommand{\scrV}{\EuScript{V}}
\newcommand{\scrW}{\EuScript{W}}
\newcommand{\scrT}{\EuScript{T}}

\newcommand{\OO}{\EuScript{O}}
\newcommand{\FF}{\EuScript{F}}

\newcommand{\scrH}{\EuScript{H}}

\newcommand{\scrM}{\EuScript{M}}
\newcommand{\scrE}{\EuScript{E}}
\newcommand{\scrF}{\EuScript{F}}

\newcommand{\bF}{\mathbf{F}}

\newcommand{\Mbar}{\bar{M}}
\newcommand{\Mdbar}{\bar{\bar{M}}}
\newcommand{\Mod}[2]{\scrR^{#2}_{#1}}
\newcommand{\Modbar}[2]{\bar{\scrR}^{#2}_{#1}}
\newcommand{\Fuk}{\scrF}
\newcommand{\Chord}{{\EuScript X}}
\newcommand{\ro}{{\mathrm o}}
\newcommand{\CO}{{\mathcal CO}}
\newcommand{\sCO}{co}

\newcommand{\ev}{\mathrm{ev}}

\newcommand{\Nbar}{\bar{N}}
\newcommand{\Ndbar}{\bar{\bar{N}}}
\newcommand{\scrK}{\EuScript{K}}

\newcommand{\tens}[2][]{%
  \mathbin{\mathop{\otimes}\limits_{#2}^{#1}}%
}

\newcommand{\wpb}{\wp_{\mathrm{plait}}}
\newcommand{\wpcirc}{\wp_{\circ}}

\numberwithin{equation}{section}
\newcommand{\superscript}[1]{\ensuremath{^{\textrm{#1}}} }

\renewcommand{\th}[0]{\superscript{th}}
\newcommand{\st}[0]{\superscript{st}}

\newcommand{\comment}[1]{}

\title{Khovanov homology from Floer cohomology}
\author{Mohammed Abouzaid}
\address{Mohammed Abouzaid, Department of Mathematics,  
Columbia University, 2990 Broadway, New York, NY 10027, USA}
\author{Ivan Smith}
\address{Ivan Smith, Centre for Mathematical Sciences, University of Cambridge, Wilberforce Road, CB3 0WB, England.}

\date{v1: April 2015. v2: December 2017. M.A. was partially supported by NSF grants DMS-1308179,  DMS-1609148, and DMS-1564172, and by the Simons Foundation through its ``Homological Mirror Symmetry'' Collaboration grant. I.S. is partially supported by a Fellowship from EPSRC}

\begin{document}

\begin{abstract} This paper realises the Khovanov homology of a link in $S^3$ as a Lagrangian Floer cohomology group, establishing a conjecture of Seidel and the second author.	The starting point is the previously established formality theorem for the symplectic arc algebra over a field $\bk$ of characteristic zero.  Here we prove the symplectic cup and cap bimodules which relate different symplectic arc algebras are themselves formal over $\bk$, and construct a long exact triangle for symplectic Khovanov cohomology. We then prove the symplectic and combinatorial arc algebras are isomorphic over $\bZ$ in a manner compatible with the cup bimodules.  It follows that Khovanov cohomology and symplectic Khovanov cohomology co-incide in characteristic zero.
\end{abstract}

\maketitle

%%%%%%%%%%%%%%%%%%%%%%%%%%%%%%%%%%%%%%%

\section{Introduction}\label{Sec:Introduction}

%This paper realises the Khovanov homology \cite{Khovanov} of a link in $S^3$ as a Lagrangian Floer cohomology group, establishing a conjecture from \cite{SS}. 

 Let $\scrY_n = \chi|_{\Slice}^{-1}(t)$ be a smooth fibre of the restriction of the adjoint quotient  $\chi: \frak{sl}_{2n}(\bC) \rightarrow \bC^{2n-1}$ to a transverse slice $\Slice \subset  \frak{sl}_{2n}(\bC)$ at a nilpotent matrix with two equal Jordan blocks.  Let $\scrF(\scrY_n)$ denote the Fukaya category of closed exact Lagrangian branes in $\scrY_n$, as constructed in \cite{FCPLT}. The paper \cite{SS} defines a symplectic structure $\omega$ on $\scrY_n$ which is exact and has contact type at infinity, and an action of the braid group $Br_{2n}$ (by parallel transport varying $t$) on objects of $\scrF(\scrY_n)$.  Let $\kappa$ be a link in $S^3$ realised as the closure of a braid $\beta_{\kappa} \times \id \in Br_{2n}$, with $\beta_{\kappa} \in \Br_n$.  There is a  distinguished Lagrangian submanifold $L_{\wp_{\circ}} \subset \scrY_n$, and a relatively $\bZ$-graded Floer cohomology group 
 \[
 \Kh_{symp}(\kappa) = HF^*(L_{\wp_{\circ}}, (\beta_{\kappa} \times \id)(L_{\wp_{\circ}}))
 \]  
 called the \emph{symplectic Khovanov cohomology} of $\kappa$. The main theorem of \cite{SS} proved that this is indeed a link invariant (independent of the choice of $\beta_{\kappa}$, and in particular of $n$, up to shifts; the relative grading can be refined to an absolute grading if one orients $\kappa$), and conjectured that it co-incided with a singly graded version of Khovanov's combinatorial / representation-theoretic invariant $\Kh(\kappa)$ from \cite{Khovanov}.  This paper proves that conjecture in characteristic zero.
  
 The categories $\scrF(\scrY_n)$ and $\scrF(\scrY_{n+1})$ are related by various canonical bimodules $\cup_i$ and  $\cap_i$, for $1\leq i \leq 2n+1$, defined by symplectic analogues of the cup and cap bimodules of \cite{Khovanov:functor}, cf. Section \ref{Sec:MeetCup}.  Such bimodules play an implicit role in the construction of the link invariant $\Kh_{symp}(\kappa)$, and were further considered in the work of Rezazadegan  \cite{Reza}.  Here we prove that  the Floer cohomology algebra generated by the Lagrangian iterated vanishing cycles associated to upper half-plane crossingless matchings (the ``symplectic arc algebra")  is isomorphic over $\bZ$ to Khovanov's arc algebra \cite{Khovanov:functor}, and we prove that the bimodules $\cup_i$ and $\cap_i$ are formal over any field $\bk$ of characteristic zero.  We also construct a long exact triangle for symplectic Khovanov cohomology, analogous to the skein triangle obeyed (tautologically) by its combinatorial sibling.  It follows that combinatorial and symplectic Khovanov cohomologies are isomorphic in characteristic zero, in particular have the same total rank over $\bQ$.

\noindent \textbf{Outline of the paper.}  The paper broadly divides into three Parts.

\begin{enumerate}
\item In \cite{AbSm} we introduced an abstract formality criterion, due to Paul Seidel, for formality of an $A_{\infty}$-algebra.  In the first Part (Sections 2-4), we provide a complementary formality criterion for $A_{\infty}$-bimodules. We explain how one can sometimes implement this criterion for bimodules over Fukaya categories of Stein manifolds  via appropriate counts of holomorphic discs in a partial compactification, and as an application prove that  the cup bimodules are formal.

\item In the second Part (Sections 5-6), we prove that the symplectic arc algebra is isomorphic over $\bZ$ to Khovanov's combinatorial arc algebra $H_n$.  The essential difficulty is one of signs: over $\bZ$, Floer complexes are not intrinsically based by intersection points, but by orientations of abstract lines associated to $\cdbar$-operators.  We construct bases of the symplectic arc algebra on $2n$ strands  in which products of positive generators are positive linear combinations of positive generators by an argument inductive on $n$. The induction relies on reduction to special cases, on plumbing models, and on the existence of a natural map from the cohomology of a symplectic manifold to the center of its Fukaya category.

\item In the third part (Section 7 and the Appendix)  we construct the  exact triangle for symplectic Khovanov cohomology, and conclude the proof.   Our discussion of the exact triangle is essentially self-contained, but adapts currently unpublished more general results due to the first author and Ganatra \cite{Abouzaid-Ganatra}, partly based on forthcoming work of the first author and Seidel \cite{AS:Lef-Wrap}.  The argument uses a Fukaya category of a Morse-Bott Lefschetz fibration, similar in nature to the categories studied by Seidel in the Lefschetz case, and various canonical functors thereon. This formalism is reminiscent of the cobordism techniques developed by Biran and Cornea \cite{BC:Cob,BC:Lef}, rather than the quilted Floer groups of Wehrheim and Woodward \cite{WW:triangle}, and avoids technical issues\footnote{Even if one is primarily interested in closed Lagrangian submanifolds and their monodromy images, \emph{a priori} the quilt theory construction of the exact triangle requires compactness for spaces of curves whose only Lagrangian boundary condition is the correspondence itself.} stemming from non-compactness of the correspondences  in our setting. 

 \end{enumerate}

The last section of the paper draws the various pieces together to complete the proof. The basic architecture of the argument can be summarised, somewhat schematically, as follows.  Combinatorial Khovanov homology is essentially determined by the arc algebra $H_n$ and the collection of cup functors (bimodules) $\cup_i$ which relate $H_n$ and $H_{n+1}$, together with a particular $H_n$-module $P_{\wp_0}$.  More precisely, Khovanov constructs a braid group action on the derived category $D=D(mod-H_n)$, in which the braid group generators are obtained formally as cones over co-units of adjunctions $\cup_i \circ \cap_i \to \id$, and defines his invariant (for the knot closure of a braid $\beta$) to be $\Ext^*_D(P_{\wp_0}, (\beta\times\id)(P_{\wp_0}))$.  Symplectic Khovanov cohomology is obtained from a geometric braid group action on the Fukaya category $\scrF(\scrY_n)$, and a choice of distinguished Lagrangian submanifold (hence module) $L_{\wp_0}$, as a Floer group $HF^*(L_{\wp_0}, (\beta\times\id)(L_{\wp_0}))$. In fact, the symplectic theory really only uses the part of $\scrF(\scrY_n)$ which can be understood as the derived category of $A_{\infty}$-modules over the ``symplectic arc algebra" $\scrH_n^{symp}$ of \cite{AbSm}, defined by a particular finite collection of Lagrangians.  The essential content of the equivalence is that 
\begin{itemize}
\item the arc algebra and symplectic arc algebra agree at the level of cohomology, in a way which entwines  the distinguished modules $P_{\wp_0}$ and $L_{\wp_0}$, and the cohomological actions of the cup bimodules; (this paper, Sections 5-6);
\item the symplectic arc algebra is formal, hence quasi-isomorphic to its cohomology, as are the geometrically defined cup bimodules; (\cite{AbSm} and this paper, Sections 2-4);
\item the generators of the geometric braid group action of \cite{SS}, which are fibred Dehn twist monodromies, can be identified with the co-units over the adjunctions between cup and cap bimodules. This largely amounts to establishing a long exact triangle for the fibred Dehn twists; (this paper, Section 7 and the Appendix).
\end{itemize}
A technical but important point is that the formality criteria for a given  bimodule can be established  starting from any fixed quasi-isomorphism of the underlying $A_{\infty}$-algebra and its cohomological algebra, which enables one to establish formality of all the cup bimodules simultaneously. 
 The existence of the isomorphism $H_n \cong \scrH_n^{symp}$ is surprisingly involved, and we give an overview of the strategy at the start of Section 5. The proof furthermore shows that  the isomorphism is compatible with cup functors, and hence (using formality) with their adjoint cap functors.  Via the exact triangle, we infer that the Fukaya category $D\scrF(\scrY_n)$ generated by the Lagrangians associated to upper half-plane crossingless matchings is braid-group equivariantly equivalent to  $D(mod-H_n)$, which implies the main theorem. More precisely, it follows that for any oriented link $\kappa$,  and for the absolute grading on $\Kh_{symp}$ mentioned above, one has isomorphisms
\begin{equation} \label{eq:main}
\Kh_{symp}^k(\kappa)\  \cong\  \bigoplus_{i-j=k} \Kh^{i,j} (\kappa) \qquad \forall \, k\in\bZ
\end{equation}
which establishes \cite[Conjecture 2]{SS}.  Any Floer group $HF^*(L_{\wp_{\circ}}, (\beta\times \id)(L_{\wp_{\circ}}))$ also inherits a generalised eigenspace decomposition from the endomorphism induced by the \nc-vector field on $\scrF(\scrY_n)$ constructed in  \cite{AbSm}.  One can view this as an additional ``weight" grading, \emph{a priori} by elements in the algebraic closure $\overline{\bk}$, and not \emph{a priori} invariant under  Markov moves. 
\begin{Conjecture}
The relative weight grading on $HF^*(L_{\wp_{\circ}}, (\beta\times \id)(L_{\wp_{\circ}})) $ is Markov invariant, and recovers the relative second grading on  $\Kh_{symp}(\kappa)$.
\end{Conjecture}

\Acknowledgements We are grateful to Sabin Cautis, Sheel Ganatra, Robert Lipshitz, Ciprian Manolescu, Jake Rasmussen and Paul Seidel, for their insights over the years on various aspects of this project.  We would also like to thank Paul Biran and Octav Cornea for sharing a preliminary version of their manuscript \cite{BC:Lef} with us, and to thank the referee for suggesting many corrections and improvements to the exposition. 

%M.A. was partially supported by NSF grant DMS-1308179.
%%%%%%%%%%%%%%%%%%%%%%%%%%%%%%%%%%%%%%%%

\section{Algebra} \label{Sec:Formality}

In this section we work over a coefficient field $\bk$, specialising where necessary to the case  in which $\bk$ has characteristic zero.

\subsection{Background}
Let $\scrA$ be a $\bZ$-graded cohomologically unital $A_{\infty}$-algebra over $\bk$, equipped with $A_{\infty}$-products
\begin{equation}
  \mu_{\scrA}^{d} \co \scrA^{\otimes d} \to \scrA, \, \, 1 \leq d
\end{equation}
of degree $2-d$. The first two operations satisfy the Leibniz equation\footnote{Our sign conventions follow those of Seidel in \cite{FCPLT}: elements of $\scrA$ are equipped with the reduced degree $||a|| = |a| -1$, and operators act on the right.}:
\begin{equation} \label{eq:Leibnitz}
  \mu_{\scrA}^{1}(  \mu_{\scrA}^{2}(a_2,a_1)) +   \mu_{\scrA}^{2}( a_2,\mu_{\scrA}^{1}( a_1)) + (-1)^{|a_1|-1} \mu_{\scrA}^{2}( \mu_{\scrA}^{1}(a_2), a_1)  = 0.
\end{equation}

The cohomology groups with respect to $\mu_{\scrA}^{1}$, denoted $A = H(\scrA)$, naturally form an $A_{\infty}$ algebra for which all operations vanish except the product, which is induced by $\mu_{\scrA}^{2} $. 
\begin{defn}
$\scrA$ is formal if it is quasi-isomorphic to $A$.
\end{defn}

The prequel  formulated and proved a necessary and sufficient condition for the formality of an $A_{\infty}$-algebra, due to Paul Seidel, in terms of the existence of a particular kind of degree one Hochschild cohomology class.   We recall that the Hochschild cochain complex $CC^*(\scrA,\scrA)$ has chain groups
\begin{equation} \label{eqn:CC_chain_complex}
CC^d(\scrA,\scrA)  = \prod_{s\geq 0} \Hom_d(\scrA[1]^{\otimes s},\scrA)
\end{equation}
where $[1]$ denotes downward shift by $1$ and $\Hom_d$ denotes $\bk$-linear maps of degree $d$, equipped with a differential
\begin{align}
\delta: CC^{d-1}(\scrA,\scrA) & \rightarrow CC^{d}(\scrA,\scrA) \\
(\delta \sigma)^d(a_d,\dots,a_1) & = \sum_{i,j} (-1)^{(|\sigma|-1)\dagger_{i}}
\mu^{d-j+1}_\scrA(a_d,\dots,\sigma^j(a_{i+j},\dots,a_{i+1}),\dots,a_1) \\ \notag
& + \sum_{i,j} (-1)^{|\sigma|+ \dagger_{i}} \sigma^{d-j+1}(a_d,\dots,\mu_\scrA^j(a_{i+j},\dots,a_{i+1}),\dots,a_1).
\end{align}
Here $\dagger_i = \sum_{k=1}^i (|a_k|-1)$.

\begin{Definition} An \nc-vector field  is a cocycle $b\in CC^1(\scrA,\scrA)$.
\end{Definition}

 In the definition, \nc \   stands for non-commutative. On a graded algebra, we have a canonical \nc-vector field called the \emph{Euler vector field}, which multiplies the graded piece $A^i \subset A$ of $A$ by $i$:
\begin{equation}
\label{Eqn:EulerField}
e: A^{i} \rightarrow A^{i}, \quad a \mapsto i \cdot a.
\end{equation}
More precisely, the fact that multiplication preserves the grading
\begin{equation}
|a_2 a_1| = |a_2| + |a_1|
\end{equation}
implies that there is a cocycle $e \in CC^1(A,A)$  which has no constant or higher order terms, i.e. only a term with $s=1$ in \eqref{eqn:CC_chain_complex}; this cocycle then defines a class in $HH^1(A,A)$.   Note that there is a natural projection of chain complexes
\begin{align} \label{eq:image_b_0}
 CC^*(\scrA,\scrA) & \rightarrow \scrA \\
b & \mapsto b^{0}
\end{align}
induced by taking the order-$0$ part of a Hochschild cochain.  Given an element of the kernel of this map, the first order part 
\begin{equation}
  b^{1} \co \scrA \to \scrA
\end{equation}
is a chain map, and hence defines an endomorphism of $A$.

\begin{Definition} \label{def:purity_algebra}
An \nc-vector field $b \in CC^1(\scrA,\scrA)$ is \emph{pure} if $b^0 = 0$, and the induced endomorphism of $A$ agrees with the Euler vector field. \end{Definition}

 If $\scrA$ admits a pure vector field, in a minor abuse of notation we say that $\scrA$ itself is pure.   The prequel paper proved:

\begin{Theorem}[Seidel] \label{Thm:Pure} 
Suppose $\bk$ has characteristic zero. If $b$ is a pure vector field on $\scrA$, then $\scrA$ is formal. Indeed, there is an equivalence
\begin{equation}
  \psi \co A \to  \scrA
\end{equation} 
acting trivially on cohomology, so $[\psi^1] = \id: A \to H(\scrA) = A$, with left inverse $\psi^{-1}$,  so that the $\psi^{-1}$-pullback of $b$ to $CC^*(A,A)$ agrees with the Euler vector field.  \qed
\end{Theorem}

%%%%%%%%%%%%%%%%%%%%%%%%%%%%%%%%%%%%%%%
\subsection{Formality for bimodules} \label{Sec:Bimodules}

To shorten formulae, in this section we use the abbreviation
\[
\scrC(X_0,X_{1},\ldots,X_r)=\hom_{\scrC}(X_{r-1},X_r)\otimes \hom_{\scrC}(X_{r-2},X_{r-1}) \otimes \cdots \otimes \hom_{\scrC}(X_0,X_1)
\]
for any $A_{\infty}$-category $\scrC$, where $r\geq 1$ (for $r=0$ define the corresponding expression to equal the ground field).

Take two $A_{\infty}$-categories $\scrA, \scrB$.  Let $P$ be an $(\scrA,\scrB)$-bimodule.   This comprises  a collection of graded vector spaces $P(Y,X)$ for any objects $X \in \scrA$, $Y\in\scrB$, together with multi-linear maps
\[
\mu_{P}^{r|1|s}: \scrA(X_0,\ldots,X_r)\otimes P(Y_0,X_0) \otimes \scrB(Y_s,\ldots,Y_0) \to P(Y_s, X_r)
\]
for $r,s\geq 0$, and any objects $X_i,Y_j$ of $\scrA$ respectively $\scrB$.  (Our ordering convention is such that hom$_{\scrA}(X,X') = \scrA(X,X')$ and if $P$ is an $(\scrA,\scrB)$-bimodule there are maps
\[
\scrA(X,X') \otimes P(Y,X) \rightarrow P(Y,X')
\]
which fits with the ordering conventions for Floer theory multiplication from \cite{FCPLT}.) Assuming that the generators of the domain of $\mu_{P}^{r|1|s}$ are written $x_r\otimes \cdots \otimes x_1 \otimes \underline{m} \otimes y_1 \otimes \cdots \otimes y_s$, where we underline the bimodule element, we write the $A_\infty$ relations as:
\begin{equation} \label{eq:relation_bimodule}
\begin{array}{ll} 
%\scriptstyle
0= \sum (-1)^{\star} \mu_{P}^{r-R+S|1|s}(x_r, \ldots, x_{R+1}, {\scriptstyle\mu_{\scrA}^{R-S+1}(x_R , \ldots , x_{S})} , x_{S-1} , \ldots , x_1, \underline{m} , y_1 , \ldots , y_s)+ \\
\;\;\,\, +
%\scriptstyle 
\sum (-1)^{\diamond} \mu_{P}^{r-R|1|s-S}(x_r , \ldots , x_{R+1} , {\scriptstyle\underline{\mu_{P}^{R|1|S}(x_{R},\ldots,x_1,\underline{m},y_1,\ldots,y_S)}} , y_{S+1} , \ldots , y_s) + \\
\;\;\,\, +
%\scriptstyle 
\sum (-1)^{\diamond} \mu_{P}^{r|1|s-R+S}(x_r , \ldots , x_1 , \underline{m} , y_1 , \ldots , y_{R-1} , {\scriptstyle\mu_{\scrB}^{R-S+1}(y_R , \ldots , y_{S})} , y_{S+1} , \ldots , y_s)
\end{array}
\end{equation}
summing over all $R,S$, and where the signs are
\[
\diamond = \sigma(y)_{S+1}^{s} \qquad \qquad \star = \sigma(y)_1^s + \mathrm{deg}(\underline{m}) + \sigma(x)_1^{S-1}
\]
with $\sigma_i^j = \sigma(x)_i^j = \sum_{\ell=i}^j (|x_{\ell}|-1)$.   The simplest equation in particular implies that the map $\mu_P^{0|1|0}$ squares to zero, hence defines the structure of a chain complex on each group $P(Y,X)$. Hence $P$ descends to a bimodule $HP$ of the corresponding cohomological categories, and we require that multiplication by the homological units of $H\scrA$ and $H \scrB$ agree with the identity map of $HP$. 

We will view a (differential) graded category as an $A_{\infty}$-category with vanishing higher products, so the usual associative composition $x\cdot y = (-1)^{|y|} \mu^2(x,y)$ on the $dg$-category is obtained by twisting the $A_{\infty}$-product by a degree-dependent sign. In particular, the cohomology of an $A_\infty$ algebra, which is a graded algebra equipped with the product induced by $(-1)^{|y|} \mu^2(x,y)$, will be considered as a (formal) $A_\infty$ algebra with the product induced by $\mu^2(x,y)$.

With this convention, bimodules form a differential graded category, with the morphisms $\scrH_{\scrA-\scrB}(P,Q)$ from $P$ to $Q$ (sometimes called pre-morphisms) given by a collection of multi-linear maps
\[
f^{r|1|s}: \scrA(X_0,\ldots,X_r)\otimes P(Y_0,X_0) \otimes \scrB(Y_s,\ldots,Y_0) \to Q(Y_s, X_r)
\]
for $r,s\geq 0$. The differential assigns to such a sequence $f$ the linear maps
\[
\begin{array}{ll}
%& %\scriptstyle 
(x_r, \ldots, x_1, \underline{m}, y_1, \ldots, y_s) \mapsto \sum (-1)^{\star} f^{r-R+S|1|s}(x_r, \ldots , x_{R+1} ,  {\scriptstyle \mu_{\scrA}^{R-S+1}(x_R, \ldots, x_S)}, x_{S-1} , \ldots x_1 , \underline{m} , y_1 , \ldots , y_s) +\\
%\scriptstyle + & 
%\scriptstyle 
\;\;\,\, +\sum (-1)^{\diamond} f^{r-R|1|s-S}(x_r, \ldots , x_{R+1} ,  {\scriptstyle\underline{\mu_{P}^{R|1|S}(x_R,\ldots,x_1,\underline{m},y_1,\ldots,y_S)}} , y_{S+1} , \ldots , y_s)+\\
%\scriptstyle + & 
%\scriptstyle 
\;\;\,\, +\sum (-1)^{\diamond} f^{r|1|s-S+R}(x_r, \ldots , x_1 , \underline{m} , y_1 , \ldots , y_{R-1} ,  {\scriptstyle\mu_{\scrB}^{S-R+1}(y_R, \ldots, y_S)}, y_{S+1} , \ldots , y_s) +\\
%\scriptstyle + & 
%\scriptstyle 
\;\;\,\, +\sum (-1)^{\diamond\mathrm{deg}(f)} \mu_{Q}^{r-R|1|s-S}(x_r, \ldots, x_{R+1},  {\scriptstyle\underline{f^{R|1|S}(x_R,\ldots,x_1,\underline{m},y_1,\ldots,y_S)}},y_{S+1},\ldots,y_s)
\end{array}
\]
summing over $R,S$, and where the signs are as before.  The product is given by
\begin{multline}
  \mu^2( f, g) (x_r, \ldots, x_1, \underline{m}, y_1, \ldots, y_s) = \\
\sum \sum (-1)^{|g| \cdot \diamond} f^{r-R|1|s-S}(x_r , \ldots , x_{R+1} , {\scriptstyle\underline{g^{R|1|S}(x_{R},\ldots,x_1,\underline{m},y_1,\ldots,y_S)}} , y_{S+1} , \ldots , y_s),
\end{multline}
and all higher order operations vanish.

In particular, the bimodule endomorphisms of a fixed bimodule $P$ form a cochain complex we denote  $(\mathcal{E}(P), \partial)$ with cohomology $\End_{\scrA,\scrB}(P)$. There are maps 
\begin{equation}
\begin{aligned}
\nu_{\scrA}: HH^*(\scrA,\scrA) \rightarrow \End_{\scrA,\scrB}(P) \\
\nu_{\scrB}: HH^*(\scrB,\scrB) \rightarrow \End_{\scrA,\scrB}(P) 
\end{aligned} 
\end{equation}
induced by a chain map which we formally denote by $\sigma \mapsto \mu^{r|1|s} \circ \sigma$, under which $\sigma$ eats some collection of the $\scrA$-inputs respectively the $\scrB$-inputs to $P$. 

\begin{Definition}
Let $b_{\scrA}$ and $b_{\scrB}$ be \nc-vector fields on $\scrA$ and $\scrB$.   A bimodule $P$ is \emph{equivariant} if the induced cohomology classes agree:
 \[
\nu_{\scrA}(b_{\scrA}) = \nu_{\scrB}(b_{\scrB}) \in \End^1_{\scrA,\scrB}(P).
 \]
 A choice of equivariant structure is an endomorphism $c_P \in \mathcal{E}^0(P)$ for which 
 \begin{equation} \label{Eqn:EquivtBimodule}
 dc_P = \nu_{\scrA}(b_{\scrA}) - \nu_{\scrB}(b_{\scrB}).
 \end{equation}
\end{Definition}

Let $b$ be an \nc-vector field on a category $\scrA$. \cite{AbSm} introduced a weaker, homotopy-theoretic version of purity of $b$ (in the sense of Definition \ref{def:purity_algebra}) which is also useful.  For an object $X \in Ob\,\scrA$, a $b$-equivariant structure on $X$ is a choice of cochain $c \in \hom^0(X,X)$ with $dc = b|_X$.  We say that $b$ defines a \emph{pure} structure on $\scrA$ if every object $X$ admits and is equipped with a choice of such an equivariant structure $c_X$, in such a way that the endomorphisms $a \mapsto b^1(a) - \mu^2(c_X, a) + \mu^2(a,c_{X'})$ of $\hom_{\scrA}(X,X')$ agree with the Euler field.  The stronger formality theorem proved in \cite[Corollary 2.13]{AbSm} is that a pure category $\scrA$ is formal. 

Now suppose that $b_{\scrA}$ and $b_{\scrB}$ induce pure structures on $\scrA$ and $\scrB$.  This means that all objects $X \in Ob\,\scrA$ and $Y \in Ob\,\scrB$  come with chosen equivariant structures $c_X, c_Y$.   The elements $c_P, c_X, c_Y$ induce an endomorphism of the cohomology $H^*(P(X,Y))$ (taking cohomology with respect to $\mu_P^{0|1|0})$. More precisely 
\begin{equation}\label{eq:equivariant_endo}
C^*(P(Y,X)) \ni \phi \mapsto c_P^1(\phi) - c_X \circ \phi + \phi \circ c_Y  \in C^*(P(Y,X))
\end{equation}
is a chain map, hence induces an endomorphism of $H^*(P(Y,X))$.  We say that $(P,c_P)$ is pure if this endomorphism agrees with the Euler field, i.e. acts in degree $i$ by multiplication by $i$.

\begin{Lemma} \label{Lem:Pure3}
If $c_P$ is pure, then $P$ is formal; more precisely, given $A_{\infty}$-equivalences $\psi_{\scrA} \co H(\scrA) \to \scrA$ and $ \psi_{\scrB} \co H(\scrB) \to \scrB$ acting trivially on cohomology and pulling back the \nc-vector fields on $\scrA$ and $\scrB$ to the corresponding Euler vector field, there are inverse equivalences of $H(\scrA)$-$H(\scrB)$ bimodules  $\psi \co H(P) \to P  $ and  $\phi \co P \to H(P)$  so that
\begin{equation}
  \psi \circ c_P \circ \phi = \mathrm{Euler}_{ H(P) }  \in \End_{H(\scrA),H(\scrB)}(H P) 
\end{equation}
\end{Lemma}
% \begin{Lemma} \label{Lem:Pure3}
% If $c_P$ is pure, then $P$ is formal; more precisely, there are $A_{\infty}$-equivalences $H(\scrA) \to \scrA$ and $ H(\scrB) \to \scrB$ acting trivially on cohomology, and inverse equivalences of $H(\scrA)$-$H(\scrB)$ bimodules  $\psi \co H(P) \to P  $ and  $\phi \co P \to H(P)$  and so that
% \begin{equation}
%   \psi \circ c_P \circ \phi = \mathrm{Euler}_{ H(P) }  \in \End_{H(\scrA),H(\scrB)}(H P) 
% \end{equation}
% \end{Lemma}
\begin{proof}
  We begin by using $\psi_{\scrA}$ and $\psi_{\scrB}$ to pull back $P$ to an $H(\scrA)$-$H(\scrB)$ bimodule. Equipping $H(\scrA)$ and $H(\scrB)$ with the Euler vector fields, the equivariant structure on $P$ pulls back to an equivariant structure in the category of $H(\scrA)$-$H(\scrB)$ bimodules, for which we use the same notation.

  Consider the category $H(\scrA) \coprod_{P} H(\scrB) $, whose set of objects is the union of the objects of $H(\scrA)$ and $H(\scrB)$, such that morphisms from objects of $H(\scrA)$ to those of $H(\scrB)$ are given by $P[1]$, and morphisms in the other direction vanish. The shift in degree implies that the products on $H(\scrA)$ and $H(\scrB)$ defined from $\mu^2_{\scrA}$ and $\mu^2_{\scrB}$ by passing to cohomology, together with operations  $\mu^{s|1|r}_{P}$, define an $A_\infty$ structure.

The Hochschild complex of this category is the direct sum of the Hochschild complexes of $H(\scrA)$ and $H(\scrB)$ with the endomorphism complex of $P$, so the data $(b_{H(\scrA)}, c_P, b_{H(\scrB)})$ induces an \nc-vector field on $H(\scrA) \coprod_{P} H(\scrB)  $. By assumption, this vector field is pure, so we conclude from Theorem \ref{Thm:Pure} that there are inverse equivalences
\begin{equation}
\xymatrix{ H(\scrA) \coprod_{P} H(\scrB) \ar[r]^-{\phi} & \ar[l]^-{\psi}  H(\scrA) \coprod_{H(P)}H(\scrB)}
\end{equation}
which map the \nc-vector field to the Euler vector field. Since the proof of Theorem \ref{Thm:Pure} proceeds by induction on the order of non-vanishing of the $A_\infty$-structure, we can insure that the induced endofunctors of $H(\scrA)$ and $H(\scrB)$ are the identity. Translating back from the  category $H(\scrA) \coprod_{P} H(\scrB) $ to the bimodule, we obtain the desired result.
\end{proof}

\subsection{Functors and bimodules}
\label{sec:functors-bimodules}
Let $\scrA$ and $\scrB$ be $A_{\infty}$ categories as before, and recall (e.g. from \cite[Section (1b)]{FCPLT}) that an $A_{\infty}$ functor $\scrF \co  \scrA \to \scrB$ assigns an object to $\scrB$ to each object of $\scrA$, together with maps
\begin{equation}
  \scrF^{d} \co \scrA(X_0,\ldots,X_d) \to \scrB( \scrF(X_0), \scrF(X_d))
\end{equation}
for each sequence $(X_0, \ldots, X_d)  $ of objects of $\scrA$. These maps are required to satisfy the $A_{\infty}$-equation for functors given, e.g. as Equation (1.6) of \cite{FCPLT}. The first such equation implies that $\scrF^1$  is a functor of cohomological categories, which we require to be unital.

The collection of functors from $\scrA$ to $\scrB$ forms an $A_{\infty}$-category, with morphisms from $\scrF$ to $\scrG$ given by pre-natural transformations, i.e. collections of maps
\begin{equation}
  T^{d} \co \scrA(X_0,\ldots,X_d) \to \scrB(\scrF(X_0),\scrG(X_d)).
\end{equation}
The $A_{\infty}$ operations are induced by the $A_{\infty}$ operations on the categories $\scrA$ and $\scrB$, together with the structure maps of the functors $\scrF$ and $\scrG$ (see e.g. \cite[Section (1d)]{FCPLT}).

Given a $\scrC-\scrD$ bimodule $P$, and functors $\scrF \co \scrA \to \scrC$ and $\scrG \co \scrB \to \scrD $, we obtain a \emph{$2$-sided pullback} $_{\scrF} P _{\scrG} $ which is an $\scrA-\scrB$-bimodule which assigns $P( \scrG Y, \scrF X)$  to a pair of objects $X$ and $Y$ (see \cite[Section 2.8]{Ganatra}). The structure maps are 
\begin{equation}
 \mu^{r|1|s}_{ _{\scrF} P _{\scrG}  }  = \sum \mu^{k|1|l}_{P} \circ \left( \scrF^{i_k} \otimes \cdots \otimes  \scrF^{i_1} \otimes \Id_{P} \otimes  \scrG^{j_k} \otimes \cdots \otimes  \scrG^{j_1} \right).
\end{equation}
When either $\scrF$ or $\scrG$ are the identity, we obtain \emph{$1$-sided pullbacks} which we denote  $  P _{\scrG} $ and $ _{\scrF} P $.
% These are functorial in the sense that the assignment $  \scrF \to _{\scrF} P $ extends to an $A_{\infty}$ functor from the category of functors from $\scrA$ to $\scrC$ to that of $\scrA-\scrB$-bimodules. On morphisms, this functor is given by
%  There is a natural functor from the category of functors from $\scrA$ to $\scrB$ to the category of $\scrA-\scrB$ bimodules defined as follows:

We shall be particularly interested in the case of the diagonal: recall that the diagonal bimodule of $\scrB$ (which we denote $\Delta^{\scrB}$)  assigns to a pair of objects $(Y,Y')$ the morphism space $\scrB(Y,Y')$, with structure maps $\mu^{r|1|s}= ( -1 )^{\diamond+1} \mu^{r+1+s} $. In this case, a functor $\scrF \co \scrA \to \scrB$ gives rise to $\scrA-\scrB$ bimodules and $\scrB-\scrA$-bimodules whose values on objects are:
\begin{align}
  {}_{\scrF}\Delta^{\scrB}(Y,X) & \equiv \scrB(Y, \scrF(X)) \\
  \Delta^{\scrB}_{\scrF}(X,Y) & \equiv \scrB(\scrF(X),Y). 
\end{align}
%Recall the definition of (non-unital) functor and natural transformations (Seidel's book?)
%, and denote by $_\scrF\scrB$, the $\scrA-\scrB$ bimodule which is the pullback of the diagonal bimodule of $\scrB$ on the left by $\scrF$, and by $\scrB_\scrF$, the $\scrB-\scrA$ bimodule obtained by pulling back the diagonal on the right. 
The following standard result is proved e.g. in \cite[Proposition 5.22]{LO:bimod}:
\begin{lem}[Yoneda Lemma] \label{lem:Yoneda-bimodules}
The assignment $\scrF \mapsto  {}_{\scrF}\Delta^{\scrB}$  (resp. $ \scrF \mapsto \Delta^{\scrB} _{\scrF} $) extends to a cohomologically fully faithful embedding from the category of functors to that of $\scrA-\scrB$ bimodules (resp. $\scrB-\scrA$-bimodules). \qed
\end{lem}

The essential image of the embedding of functors into bimodules can be determined as follows: we say that an $\scrA-\scrB$ bimodule $P$ is \emph{representable as a $\scrB$-module} if, for every object $X$ of $\scrA$, there is an object $\scrF(X)$ of $\scrB$ so that the $\scrB$-module $P( \_, X)$ is quasi-isomorphic to the (Yoneda) module $\scrB( \_, \scrF(X))$. The following result is proved, e.g. in \cite[Proposition 5.23]{LO:bimod}:
\begin{lem} \label{lem:representable_bimod_functor}
  If $P$ is  representable as a $\scrB$-module, then there is a functor $\scrF \co \scrA \to \scrB$, such that
  \begin{equation}
    P \cong {}_{\scrF}\Delta^{\scrB}.
  \end{equation} \qed
\end{lem}
By the Yoneda Lemma, the representing functor is in fact unique up to quasi-isomorphism. We shall need an additional result comparing composition of functors (see \cite[Sections (1b) and (1e)]{FCPLT}) to the tensor product of bimodules.  Recall the definition of the tensor product of $\scrA-\scrB$ and $\scrB-\scrC$ bimodules $P$ and $Q$. This is the $\scrA-\scrC$ bimodule, denoted $P \tens{\scrB} Q$, which is given by
\begin{equation}
  (P \tens{\scrB} Q)(Z,X) = \bigoplus_{\{Y_0,\ldots,Y_d\} \in \Ob \scrB} P(Y_d,X) \otimes \scrB(Y_0, \cdots, Y_d) \otimes Q(Z,Y_0),
\end{equation}
with differential induced by the structure maps of $P$ and $Q$ as $\scrB$ modules (together with the differential $\mu^1_\scrB$), and higher operations induced by the remaining structure maps of $P$ and $Q$ as bimodules (see \cite[Equation (2.14)]{Seidel:algebra_and_nat}).
 
%, $\scrF \circ \scrG$ the composition  and $ _\scrF\scrB \tens{\scrB} {}_\scrG\scrC$ the tensor product of the corresponding bimodules.

\begin{lem} \label{lem:tensor_product_is_composition} 
Let $\scrF \co \scrA \to \scrB$ and $\scrG \co \scrB \to \scrC$ be $A_{\infty}$ functors.  There is a natural quasi-isomorphism of bimodules
  \begin{equation}
    _\scrF\Delta^{\scrB} \tens{\scrB} {}_\scrG\Delta^\scrC \longrightarrow  {}_{\scrG \circ \scrF} \Delta^\scrC.
  \end{equation}
\end{lem}
\begin{proof}[Sketch of proof:]
The map from
\[
\scrB(X_0, \ldots, X_r) \otimes \scrB(Y_d,\scrF X_0) \otimes \scrB(Y_0, \cdots, Y_d) \otimes \scrC(Z_s, \scrG Y_0) \otimes \scrC(Z_s, \ldots, Z_0)
\]
to $\scrC(Z_0,(\scrG\circ\scrF) X_r)$ is: 
\begin{multline}
(x_r, \ldots, x_1, \underline{p},y_d, \ldots, y_1, \underline{q},y_1, \ldots, y_s)  \mapsto  
 \sum (-1)^{\diamond} \mu_{\scrC}^{k+e+s}\big(\scrG^{j_{e}}(x_{r} , \ldots, x_{r-j_{e}+1}) , \ldots , \\ \ldots,  \scrG^{i_k + j_1}(x_{j_1} , \ldots, x_1, \underline{p}, y_d, \ldots, y_{d-i_k+2}),  \ldots ,  \scrG^{i_1}(y_{i_1}, \cdots, y_1) , \underline{q},z_1, \ldots, z_s\big).
\end{multline}
At the linear level, we may fix an object $Z$ of $\scrC$. Letting $ {}_\scrG \scrY_{Z} $ denote the pullback under $\scrG$ of the left Yoneda module of $Z$, the above map is exactly the linear term in the natural map of left $\scrB$-modules 
\begin{equation}
  \Delta^{\scrB} \tens{\scrB} {}_\scrG \scrY_{Z}  \longrightarrow   {}_\scrG \scrY_{Z} 
\end{equation}
The result therefore follows from the fact that tensoring with the diagonal bimodule induces a quasi-isomorphism of bimodules by the acyclicity of the bar complex (see for instance \cite[Proposition 2.2]{Ganatra}).
\end{proof}

\subsection{Formality for functors}

If $\scrA$ and $\scrB$ are equipped with \nc-vector fields $b_{\scrA}$ and $b_{\scrB}$, we say that  $\scrF$ is pure if either $\Delta^\scrB_\scrF $ or $_\scrF\Delta^\scrB$ is pure. This implicitly means that $\scrA$ and $\scrB$ are equipped with pure equivariant structures and that the pullback of the diagonal bimodule is equipped with a pure equivariant structure; we write $\psi_{\scrA}$ and $\psi_{\scrB}$ for the induced equivalences from $H \scrA$ to $\scrA$ and $H \scrB$ to $\scrB$. In Corollary \ref{cor:purity_left_implies_right} below, we prove that purity as a left or right bimodule are equivalent conditions.

By Theorem \ref{Thm:Pure}, the pure equivariant structures on $\scrA$ and $\scrB$ induce quasi-equivalences of these categories with their cohomological categories. We say that $\scrF$ is formal if there is an equivalence of functors between $H \scrF$ and the composition
\begin{equation}\label{eq:descend_functor_to_cohomology}
\xymatrix{H \scrA \ar[r]^{\psi_{\scrA}} & \scrA  \ar[r]^{\scrF} & \scrB \ar[r]^{\psi^{-1}_{\scrB}} & H \scrB.}
\end{equation}
\begin{lem}
  The functor $\scrF$ is formal if and only if $ \Delta^\scrB_\scrF $ (or $_\scrF\Delta^\scrB $) is formal as a bimodule.
\end{lem}
\begin{proof}
The $H\scrB$-$H\scrA$ bimodule $ \Delta^{H\scrB}_{H\scrF} $ which represents $H \scrF $ is canonically isomorphic to $H\left(  \Delta^\scrB_\scrF   \right)  $. If $\scrF$ is formal, we conclude that the pullback of $\Delta^\scrB_\scrF $ to an $H\scrB$-$H\scrA$ bimodule is quasi-isomorphic to its cohomological module, hence is formal. In the other direction,  if $\Delta^\scrB_\scrF  $ is formal, we conclude that the bimodules representing   $H \scrF $ and the composition of $\scrF$ with the equivalences from $\scrA $ and $\scrB$ to their cohomologies are quasi-isomorphic, hence the corresponding functors are equivalent by the Yoneda Lemma \ref{lem:Yoneda-bimodules}.  
\end{proof}
%An equivariant structure on an $A_\infty$ functor $\scrF \co \scrA \to \scrB$ is an equivariant structure on $_\scrF\scrB$, the $\scrA-\scrB$ bimodule which is the pullback of the diagonal bimodule of $\scrB$ on the left by $\scrF$. We say that $\scrF$ is pure if either $_\scrF\scrB $ or  is pure; this implicitly means that $\scrA$ and $\scrB$ are equipped with pure equivariant structures; we write $\psi_{\scrA}$ and $\psi_{\scrB}$ for the induced equivalences from $H \scrA$ to $\scrA$ and $H \scrB$ to $\scrB$.

\begin{prop} \label{prop:purity_implies_formal_functor}
$\scrF $ is pure if and only if it is formal.
\end{prop}
\begin{proof}
We consider the case in which  $_\scrF\Delta^\scrB  $ is pure: by Lemma \ref{Lem:Pure3}, the pullback of $_\scrF\Delta^\scrB  $ to an $ H \scrA- H \scrB $-bimodule is formal, hence equivalent to $ _{H\scrF} \Delta^{H\scrB} $. The result follows from the fact that the left-sided pullback defines a (cohomologically) fully faithful embedding from the category of functors to the category of bimodules.
\end{proof}
One can formulate the above Lemma more precisely, by noting that a quasi-isomorphism from $H \scrF $ to the composition in Equation \eqref{eq:descend_functor_to_cohomology} equips $\scrF$ with a pure structure, and that a pure structure induces a quasi-isomorphism between these two functors.

\begin{cor} \label{cor:purity_left_implies_right}
If $_\scrF\Delta^\scrB $ is pure, so are  $\Delta^\scrB_\scrF$ and $_\scrF \Delta^\scrB_\scrF$, and vice-versa replacing left by right sided pullbacks.
\end{cor}
\begin{proof}
Purity implies formality of $\scrF$, which implies formality of $\Delta^\scrB_\scrF$ and $_\scrF \Delta^\scrB_\scrF$; hence the existence of a pure equivariant structures on these bimodules.
\end{proof}

%The adjoint of a formal $A_{\infty}$-functor is in fact automatically formal; whilst we do not need this, we include a proof in the Appendix for completeness. 

\begin{rem}
By considering only the pure case, we avoided having to appeal to the following result whose proof uses duality of bimodules: an equivariant structure on $_\scrF\Delta^\scrB $ induces an equivariant structure on $\Delta^\scrB_\scrF$, and hence on $_\scrF \Delta^\scrB_\scrF $  via the natural equivalence
\begin{equation}
_\scrF \Delta^\scrB \tens{\scrB} \Delta^\scrB_\scrF \to  {}_\scrF\Delta^\scrB_\scrF.
\end{equation}
Moreover, equipping the two-sided pullback with this equivariant structure, the natural map
\begin{equation}
\Delta^\scrA \to  {}_\scrF\Delta^\scrB_\scrF
\end{equation}
of $\scrA-\scrA$-bimodules is equivariant. If $\scrF$ has an adjoint, cf. Section \ref{sec:adjunctions}, this map gives rise to the unit of the adjunction.
\end{rem}

\subsection{Adjunctions}\label{sec:adjunctions} 
Let $\scrF \co \scrA \to \scrB$ and $\scrG \co \scrB \to \scrA$ be $A_{\infty}$ functors. As before, denote by $_\scrF\Delta^\scrB$ the $\scrA-\scrB$ bimodule which is the pullback of the diagonal bimodule of $\scrB$ on the left by $\scrF$, and by $\Delta^\scrA_\scrG$ the   $\scrA-\scrB$ bimodule which is the pullback of the diagonal bimodule of $\scrA$ on the right by $\scrG$.
\begin{defn}
An \emph{adjunction} between $\scrF$ and $\scrG$ is an equivalence between $_\scrF\Delta^\scrB$ and $\Delta^\scrA_\scrG$ as $\scrA-\scrB$ bimodules.
\end{defn}

The equivalence between $_\scrF\Delta^\scrB$ and $\Delta^\scrA_\scrG$ induces equivalences of $A_{\infty}$-bimodules
\begin{align}
 _\scrF \Delta^\scrB_\scrF \cong _\scrF\Delta^\scrB \tens{\scrB} \Delta^\scrB_\scrF \cong  \Delta^\scrA_\scrG \tens{\scrB} \Delta^\scrB_\scrF  &  \cong   \Delta^ \scrA_{\scrG \circ \scrF } \\
 \Delta^\scrB_\scrF   \tens{\scrA}  {}_\scrF \Delta^\scrB  \cong \Delta^\scrB_\scrF   \tens{\scrA} \Delta^\scrA_\scrG &  \cong \Delta^ \scrB_{\scrF \circ \scrG }  .
\end{align}
In particular, the natural maps
\begin{align}
\Delta^\scrA & \to {}_\scrF\Delta^\scrB_\scrF \\
 \Delta^\scrB_\scrF   \tens{\scrA}  {}_\scrF\Delta^\scrB & \to \Delta^\scrB
\end{align}
induce maps called the unit and counit of the adjunction:
\begin{align}
 \Delta^\scrA & \to    \Delta^\scrA_{\scrG \circ \scrF } \\
 \Delta^\scrB_{\scrF \circ \scrG }  & \to \Delta^\scrB.
\end{align}
Given that the category of functors embeds fully faithfully into the category of bimodules, we conclude:
\begin{lem}
An adjunction between $\scrF$ and $\scrG$ induces $A_{\infty}$-natural transformations
\begin{align}
  \Id_{\scrA} & \to \scrG \circ \scrF  \\
\scrF \circ \scrG &  \to \Id_{\scrB}. 
\end{align} \qed
\end{lem}
These maps are the unit and counit of the adjunction. Repeating the construction at the cohomological level, we obtain adjunctions:
\begin{align}
  \Id_{H\scrA} & \to H\scrG \circ H\scrF  \\
H\scrF \circ H\scrG &  \to \Id_{H\scrB}.
\end{align}

\begin{prop}
If $\scrF$ is formal and is either left or right adjoint to $\scrG$, then $\scrG$ is formal. Moreover, the units and counits are formal; i.e. the unit and counit of the adjuction on $H \scrA$ and $H \scrB$ are cohomologous (as $A_{\infty}$-natural transformations) to the pullback of the unit and counit on $ \scrA$ and $ \scrB$.
\end{prop}
\begin{proof}
Using the homological perturbation lemma, it suffices to prove the result assuming that $\scrA$ and $\scrB$ agree with their cohomology. In this case, we have a composite quasi-isomorphism
\begin{equation}
  \Delta^\scrA_\scrG  \cong  {_\scrF}\Delta^\scrB  \cong {_{H(\scrF)}}\Delta^\scrB  \cong   \Delta^\scrA_{H(\scrG)}
\end{equation}
where the first map is the data of the adjuction, the second is the formality of $\scrF$, and the last is induced by the adjunction after passing to cohomology. This proves that $\scrG$ is formal.

To show that the unit of the adjunction is formal, we use the embedding of the category of functors in the category of bimodules. We begin by noting that the compatibility of compositions with natural transformations implies that the formality quasi-isomorphisms for $\scrF$, $\scrG$, induce a formality quasi-isomorphism for their composite. Passing to bimodules, the compatibility of compositions with tensor products yields a (homotopy) commutative diagram
\begin{equation}
   \xymatrix{ \Delta^ \scrB_{\scrF \circ \scrG } \ar[r] \ar[d] &  \Delta^\scrB_\scrF   \tens{\scrA} \Delta^\scrA_\scrG  \ar[d]\\
 \Delta^\scrB_{H(\scrF) \circ H(\scrG)} \ar[r] &  \Delta^\scrB_{H(\scrF)}   \tens{\scrA} \Delta^\scrA_{H(\scrG)}  . }   
\end{equation}

Since the adjunction is represented by the composition
\begin{equation}
\Delta^ \scrB_{\scrF \circ \scrG } \to \Delta^\scrB_\scrF   \tens{\scrA} \Delta^\scrA_\scrG    \to  \Delta^\scrB,
\end{equation}
it thus suffices to compare the second map with the corresponding map of cohomological bimodules.  But this map was defined, using the data of the adjunction, via a factorisation
\begin{equation}
   \Delta^\scrB_\scrF   \tens{\scrA} \Delta^\scrA_\scrG  \to   \Delta^\scrB_\scrF   \tens{\scrA}   {_\scrF}\Delta^\scrB  \to  \Delta^\scrB.
\end{equation}
This implies that the chosen equivalence $\Delta^\scrA_\scrG \cong  \Delta^\scrA_{H(\scrG)}$ which we constructed above fits in a commutative diagram
\begin{equation}
  \xymatrix{ \Delta^\scrB_\scrF   \tens{\scrA} \Delta^\scrA_\scrG \ar[r] \ar[d] & \Delta^\scrB_\scrF   \tens{\scrA}   {_\scrF}\Delta^\scrB \ar[r] \ar[d] &  \Delta^\scrB \ar[d] \\
 \Delta^\scrB_{H(\scrF)}   \tens{\scrA} \Delta^\scrA_{H(\scrG)} \ar[r] &   \Delta^\scrB_{H(\scrF)}   \tens{\scrA}    {_{H(\scrF)}}\Delta^\scrB  \ar[r] &  \Delta^\scrB, }
\end{equation}
where the bottom maps are cohomological. This proves formality of the unit, and a similar argument yields formality of the counit. \end{proof}

\section{Bimodules over Fukaya categories}\label{Sec:Bimodule_Generalities}

This section collects geometric generalities in the spirit of \cite[Section 3]{AbSm}, but now concerning bimodules.  We recall the standing hypotheses of the prequel \cite{AbSm}, whose notation and conventions we adopt. We construct the Fukaya category in a slight generalisation of the framework introduced by Seidel \cite{FCPLT}, that allows for clean intersections of Lagrangians by combining Floer theory with Morse theory of auxiliary functions on intersections, as in Biran and Cornea's ``pearly trajectories'' \cite{BC:Pearl}. This approach was implemented by Seidel in \cite{Seidel:genus} and Sheridan in \cite{Sheridan}.  Depending on the situation, we achieve compactness for moduli spaces of Floer holomorphic discs by one of three mechanisms:  by exactness;  by using the existence of an effective compactification divisor at infinity which is nef on rational curves; or from the existence of a holomorphic map to a bounded domain in $\bC^r$, and maximum principle in the base coupled to exactness in the fibre.

This section is a direct extension of the results of \cite[Section 3]{AbSm} from categories to bimodules, and the reader will find a more leisurely exposition in the prequel.

\subsection{Geometric set-up}\label{Sec:Set-up}

The set-up introduced in \cite{AbSm} as a general setting for proving formality of a Floer $A_{\infty}$-algebra had the following ingredients. 
We begin with a smooth projective variety $\Mdbar$ of complex dimension $n$, equipped with a triple of reduced (not necessarily smooth or irreducible) effective divisors $D_0$, $D_{\infty}$, $D_r$. We denote by $\Mbar$ the symplectic manifold obtained by removing $D_{\infty}$ from $\Mdbar$, and by $M$ the symplectic manifold obtained by removing the three divisors from $\Mdbar$.  When the meaning is clear from context, we shall sometimes write $D_{0}$ for $D_{0} \cap \Mbar$, and $D_{r}$ for both $D_{r} \cap \Mbar$ and $D_{r} \cap M$. We assume:

\begin{Hypothesis} \label{Hyp:Main} 
  \begin{align} \label{eq:hyp_main_1}
& \parbox{30em}{the union  $D_0 \cup D_{\infty} \cup D_r$ supports an  ample divisor $D$ with strictly positive coefficients of each of $D_0, D_{\infty}, D_r$.} \\
& \parbox{30em}{$D_{\infty}$ is nef (or, at least, non-negative on rational curves).}\\ \label{eq:hyp_main_3}
&  \parbox{30em}{$\Mbar$ admits a meromorphic volume form $\eta$ which has simple poles along $D_0 \cap \Mbar$, and is holomorphic and non-vanishing away from this divisor (in particular in $M$).} \\
%. is holomorphic and non-vanishing in $M$, holomorphic on $D_r\cap\Mbar$ and with simple poles along $D_0 \cap \Mbar$.}  \\
& \parbox{30em}{Each irreducible component of the divisor $D_0 \cap \Mbar$ moves in $\Mbar$, with base locus containing no rational curves.} 
%& \parbox{30em}{}
\end{align}
\end{Hypothesis}

To slightly clarify the hypotheses, in our examples the ample divisor in the first part will have class Poincar\'e dual to the class $n_0 D_0 + n_{\infty} D_{\infty} + n_r D_r$ with $\{n_0,n_{\infty},n_r\} \in \bR_{+}$; since the $D_i$ are not assumed irreducible, one could also consider K\"ahler forms with class in the interior of the positive cone spanned by the irreducible components of the $D_i$.  In the final part, the ``moving" hypothesis asserts that if $D_0^{\dagger}$ is an irreducible component of $D_0$,  the line bundle $\mathcal{O}(D_0^{\dagger}) \to \Mdbar$ has global holomorphic sections when restricted to $\Mbar$, and that for any rational curve in $D_0^{\dagger} \cap \Mbar$ there is a holomorphic section of $\mathcal{O}(D_0^{\dagger})|_{\Mbar}$ which is not identically zero on that curve.   It may be worth pointing out that, when this set-up is applied to the Milnor fibre in \cite[Section 4]{AbSm}, one of the irreducible components of $D_0$ is a curve of negative square which does \emph{not} move on $\Mdbar$; in the sequel, we shall only use the existence of push-offs $D_0'$ of $D_0$ on the open piece $\Mbar$.

\begin{Remark} \label{Remark:Duke_correction}
The third hypothesis strengthens (and corrects) the hypothesis of \cite[Hypothesis 3.1]{AbSm}, where we allowed the meromorphic volume form on $\Mbar$ to have zeroes (but not poles) on $D_r$. In the two situations of relevance in \cite{AbSm}, the stronger hypothesis holds. In the first, $M$ is a Milnor fibre, and $D_r = \emptyset$, so the situations become identical. In the second, $M$ is the nilpotent slice viewed as an open subvariety of the Hilbert scheme of a Milnor fibre, and $D_r$ is a relative Hilbert scheme. This is contained in the exceptional locus of the crepant Hilbert-Chow morphism; the meromorphic volume form on the partially compactified Milnor fibre induces one on its symmetric product, and this pulls back to one on the Hilbert scheme which has no zeroes (cf. proof of \cite[Lemma 6.3]{AbSm}).  The stronger hypothesis is required for \cite[Lemma 3.2]{AbSm}, or equivalently to conclude \eqref{eqn:chern_number} below. 
\end{Remark}

\begin{Remark} 
All the holomorphic curves that we study lie in $\Mbar$, so that $\Mdbar$ is only used to ensure compactness for moduli spaces of stable discs in $\Mbar$, via algebraic arguments such as positivity of intersection. Any other method for ensuring compactness would achieve the same end, and our use of this specific formalism is in great part due to the fact that we rely on many results and constructions from \cite{AbSm}. At the end of this section (see Condition \eqref{eq:condition_domain_in_M}), we introduce certain domains $\mathring{M} \subset M$ and $\mathring{\Mbar} \subset \Mbar$ that were not needed in \cite{AbSm}, but which will be required to compare equivariant structures on different spaces.
\end{Remark}

Let $D'_0 \subset \Mbar$ be a divisor linearly equivalent to and sharing no irreducible component with $D_0$, and $B_0 = D_0 \cap D'_0$, which is then a subvariety of  $\Mbar$ of complex codimension $2$.

Fix a K\"ahler form $\omega_{\Mdbar}$ in the cohomology class Poincar\'e dual to $D$. Ampleness implies that $M$ is an affine variety, in particular an exact symplectic manifold which can be completed to a Weinstein manifold of finite type, modelled on the symplectization of a contact manifold near infinity; see \cite{Seidel:bias} for details.  We will denote by $\lambda$ a primitive of the symplectic form $\omega_M$ given by restricting $\omega_{\Mdbar}$ to $M$, so $d\lambda = \omega_M$.     By the third assumption above, $M$ has vanishing first Chern class; the fourth hypothesis then implies that for rational curves $C\subset \Mbar$, there is an identity 
\begin{equation} \label{eqn:chern_number}
\langle c_1(\Mbar), C\rangle = \langle D_0,C\rangle \geq 0.
\end{equation}

Let $A \in H_2(\Mbar; \bZ)$ be a $2$-dimensional homology class, with the property that
\begin{equation} \label{eq:intersection_0_1}
  \langle D_{r},A \rangle = 0 \textrm{ and }   \langle D_{0}, A \rangle = 1. 
\end{equation}
Consider the moduli space of stable rational curves in $\Mbar$ with one marked point
\begin{equation}
\scrM_{1}(\Mbar| 1) = \coprod_{\substack{A \in  H_2(\bar{M};\bZ) \\ \textrm{Condition \eqref{eq:intersection_0_1} holds}}} \scrM_{1; A}(\Mbar)
\end{equation}
which can be decomposed according to  the homology class $A \in H_2(\Mbar;\bZ)$ represented by each element.  Let $H_*^{lf}(\bullet;\bZ)$ denote the locally finite (or Borel-Moore) homology of a space $\bullet$. The evaluation image $\ev_1(\scrM_1(\Mbar| 1))$ defines a locally finite chain in $\Mbar$, and hence by Poincar\'e duality we obtain a well-defined associated Gromov-Witten invariant  $GW_1 \in H^2(M;\bZ)$ counting such curves (see \cite[Lemma 3.6]{AbSm}).

\begin{Hypothesis}  \label{Hyp:GWvanishes} 
  \begin{align}
 & GW_{1} = \sum_A GW_{1;A}|_M = 0 \in  H_{2n-2}^{lf}(M;\bZ) \cong H^2(M; \bZ); \\
& \parbox{30em}{$B_0$ is homologous to a cycle supported on the union $(D_0 \cap D_{r}) \cup D_0^{sing}$, where $D_0^{sing}$ denotes the singular locus of $D_0$.}
  \end{align}
\end{Hypothesis}

Appealing to Hypothesis \ref{Hyp:GWvanishes}, we fix cochains 
\begin{align} \label{eq:bounding_chain_gw}
  gw_{1} & \in C^1(M ; \bZ) \\ \label{eq:bounding_chain_infinity}
\beta_0 & \in C_{2n-3}^{lf}(D_0 , (D_{0} \cap D_{r})  \cup D_{0}^{sing}  ; \bZ)
\end{align}
satisfying 
\begin{align}
\partial  (gw_{1}) & = GW_{1} \\
\partial (\beta_0) & = [B_0].
\end{align}

Let $L \subset M$ denote an exact Lagrangian brane. Let $\Delta$ denote the closed unit disc, $\Modbar{(0,1)}{1}(L)$ denote the moduli space of Floer solutions comprising maps from $(\Delta, \partial \Delta)$ to $(M,L)$, with the origin being the unique point  mapping to $D_0$, and a point on the interval $(0,1)\subset \Delta$ mapping to $D_0'$ (the last condition imposes an additional real codimension one constraint).  We assume that perturbation data for the Floer equation is chosen as in \cite[Section 3.4]{AbSm}, so as in particular to make the fibre product arising in \eqref{eqn:building_equivt_str} below transverse.  By evaluation at $1$, we obtain a cochain
\begin{equation}
   \tilde{b}_{D}^{0} = \ev_{*}[\Modbar{(0,1)}{1}(L)  ]  \in C^{1}(L; \bZ).
\end{equation} 
Consider also the moduli space 
\begin{equation}
  \scrR_{1}^{1}(\Mbar; (1,0)| L  )
\end{equation}
of discs in $\Mbar$ with boundary on $L$ and one interior and one boundary marked point, with intersection numbers $(1,0)$ with $(D_0,D_r)$; this has a natural map $\scrR_1^1(\Mbar; (1,0)|L ) \rightarrow \Mbar$ via evaluation at the interior marked point.  We set
\begin{equation} \label{eqn:building_equivt_str}
\sCO^{0}( \beta_0)  =  \ev_{*}  [\beta_{0}\times_{\Mbar} \scrR_{1}^{1}(\Mbar; (1,0)| L  )]  \in C^{1}(L; \bZ).
\end{equation}
We recall from \cite[Lemma 3.11]{AbSm} that the sum of the restriction of $gw_{1}$ with $b_{D}^{0}$ and $\sCO^{0} \beta_0$ defines a cycle
\begin{equation} \label{eq:0-part-CC_cocycle}
b^0_D =  \tilde{b}_{D}^{0} + gw_1|L + \sCO^{0}( \beta_0) \in C^1(L; \bZ).
\end{equation}
A Lagrangian brane $L$ is \emph{(infinitesimally) equivariant} if the cycle in Equation \eqref{eq:0-part-CC_cocycle} is null-homologous. An \emph{(infinitesimally) equivariant} structure on $L$, over $\bk$,  is a choice of bounding cochain in $ C^0(L; \bk) $ for this cycle. 

\subsection{Restriction to subdomains}
\label{sec:restr-subd}

For our applications, we shall need to specify a domain $\mathring{M} \subset M$ which will contain all Lagrangians of interest. It will be important to know not only that stable curves in $M$ which intersect $\mathring{M}$ are contained in $\mathring{M}$, but that the same property holds for curves in $\bar{M}$ with respect to the closure $\mathring{\bar{M}}$ of $\mathring{M}$ in $\bar{M}$.

To this end, we assume that
\begin{equation} \label{eq:condition_domain_in_M}
 \parbox{38em}{$ \mathring{\bar{M}} \subset \bar{M}$ is a weakly pseudoconvex domain, i.e. a properly embedded codimension $0$ submanifold whose boundary is defined by a weakly plurisubharmonic function.}
\end{equation}
In our applications, $ \mathring{\bar{M}}$ will be the inverse image of a domain in $\bC$ under a holomorphic function defined on $\bar{M} $.

In this setting, an integrated version of the maximum principle, see \cite{Abouzaid-Seidel},  implies:
\begin{lem} \label{lem:integrated_maximum_principle}
If $\Sigma$ is a stable Riemann surface with boundary, and $u \co \Sigma \to \bar{M}$ a holomorphic map mapping the boundary to $\mathring{\bar{M}} $, then the image of $u$ is contained in $\mathring{\bar{M}}$. \qed
\end{lem}

From this, we conclude that exact Lagrangians in $\mathring{M} $  are the objects of a Fukaya category $\scrF(\mathring{M})$ whose definition uses only curves that remain in $\mathring{M}$, In particular, we have a fully faithful embedding
\begin{equation}
    \scrF(\mathring{M}) \subset\scrF(M).
  \end{equation}
The equivariant structure on such Lagrangians depends only on curves in  $ \mathring{\bar{M}}$. More precisely, let us denote by $D^{\mathring{\bar{M}}}_0$,  $D^{\prime,\mathring{M}}_0$, $B^{\mathring{M}}$, and $D^{\mathring{\bar{M}}}_r$ the intersections of $D_0$, $D_0'$, $B$, and $D_r$ with $\mathring{\bar{M}}$. Since the embedding $\mathring{\bar{M}} \subset \bar{M}$ is proper, the (relative) chains $ gw_{1}$ and $\beta_0$ restrict to (relative) chains $gw^{\mathring{M}}_{1} \in C^1(\mathring{M} ; \bZ)$, and $\beta^{\mathring{M}}_0 \in C_{2n-3}^{lf}(D^{\mathring{\bar{M}}}_0 , (D^{\mathring{\bar{M}}}_0 \cap D^{\mathring{\bar{M}}}_{r}) \cup D^{\mathring{\bar{M}},sing}_{0}  ; \bZ) $.

From Equation \eqref{eq:0-part-CC_cocycle} and Lemma \ref{lem:integrated_maximum_principle}, we conclude the following result:
\begin{lem}
For each Lagrangian $L \in \mathring{M}$, the obstruction $b^0_D$ to equivariance  depends only on (i) the moduli space of stable holomorphic discs in $\mathring{M}$ and $\mathring{\Mbar}$ with boundary on $L$, and (ii) the chains $gw^{\mathring{M}}_{1}$ and $\beta^{\mathring{M}}_0$. In particular, it is independent of the inclusions $\mathring{\Mbar} \subset \Mbar \subset \Mdbar$. \qed
\end{lem}

\subsection{Bimodules from product Lagrangians} \label{Sec:BimodulesWithoutQuilts}

Consider two quasiprojective K\"ahler manifolds $M$, $N$, with (partial) compactifications $M \subset \Mbar \subset \Mdbar$ and $ N \subset \Nbar \subset \Ndbar$  by divisors $D$, $D^{N}$ each satisfying Hypotheses \ref{Hyp:Main} and \ref{Hyp:GWvanishes}, in particular $D^{N}$ is supported on the union of three divisors $D^{N}_{0}$, $D^{N}_{r}$, and $D^{N}_{\infty}$. Assume in addition that we are given weakly pseudo-convex domains $\mathring{\bar{M}} \subset \bar{M}$ and $\mathring{\bar{N}} \subset \bar{N}$.

Let $W$ be a quasiprojective K\"ahler manifold with a (partial) compactification $W \subset \bar{W} \subset \bar{\bar{W}}$, equipped with an embedding
\begin{equation}
  \mathring{\bar{M}} \times \mathring{\bar{N}} \equiv  \mathring{\bar{W}}  \subset \bar{W}   
\end{equation}
which is compatible with the divisors in the sense that
\begin{align}
  D^{ \mathring{\bar{W}}}_0 & = (D^{\mathring{\bar{M}}}_0 \times  \mathring{\bar{N}})  \cup  (\mathring{\bar{M}}  \times D^{\mathring{\bar{N}}}_0) \\
  D^{ \mathring{\bar{W}}}_r & = (D^{\mathring{\bar{M}}}_r \times  \mathring{\bar{N}}) \cup  (\mathring{\bar{M}}  \times D^{\mathring{\bar{N}}}_r ).
\end{align}
We write $\pi_N$ for the projection from $\mathring{\bar{W}}$ to $\mathring{\bar{N}}$.  The domain $\mathring{\bar{W}}$ admits embeddings in both $\bar{\bar{W}}$ and $\Mdbar \times \Ndbar$, but the resulting Fukaya category of Lagrangians in $\mathring{\bar{W}}$ is  independent of these embeddings by Lemma \ref{lem:integrated_maximum_principle}, as is the equivariant structure on Lagrangians in $\mathring{W} \subset \mathring{\bar{W}}$. 

We shall assign  a $(\Fuk(\mathring{W})$-$\Fuk(\mathring{M}))$ bimodule $\scrK$ to each Lagrangian brane $K \subset \mathring{N}$.  By construction, this bimodule is representable, and hence corresponds to a functor on Fukaya categories. The theory of pseudo-holomorphic quilts due to Mau-Wehrheim-Woodward \cite{WW,MWW} could be used to give a more general construction than the one we provide here, provided the relevant moduli spaces of holomorphic curves with boundary on the correspondence behave well. We have avoided using quilt theory because we would need to discuss the behaviour of the correspondence at the boundary.
%the correspondence we consider later will not be compact in $\mathring{\Mbar} \times \mathring{W}$, and we do not know that spaces of holomorphic curves with boundary on that correspondence are compact. 

Given objects $L$ and $L'$ of $\Fuk(\mathring{M})$ and $\Fuk(\mathring{W})$, we choose a compactly supported Hamiltonian 
\begin{equation} \label{eq:Hamiltonian_bimodule}
  H_{L,L'} \co [0,1] \times \mathring{W} \to \bR
\end{equation}
whose flow maps $L \times K$ to a Lagrangian transverse to $L'$ and define 
\begin{equation}
  \scrK(L, L') \equiv CF^{*}(L \times K, L').
\end{equation}
We shall write $\Chord(L \times K,L')$ for the set of intersection points of the time-$1$ image of $L \times K$ and $L'$, so the vector space underlying the Floer complex is by definition the sum of the orientation lines $\ro_x$ associated to points $x\in\Chord(L\times K, L')$. 
Note that, given our conventions for multiplication from the right in Floer complexes, together with the usual conventions for bimodules as set out in Section \ref{Sec:Bimodules}, this is indeed an $\scrF(\mathring{W})$-left module and an $\scrF(\mathring{M})$-right module. To define the structure maps of this $A_{\infty}$ bimodule, we introduce the moduli space 
\begin{equation}
  \scrR^{r|1|s} \equiv \scrR^{r+s+2}
\end{equation}
which is a copy of the interior of the Stasheff associahedron.  Ordered counter-clockwise, the marked points on the boundary of an element of $ \scrR^{r|1|s}  $ are denoted
 \begin{equation}  \label{eq:marked_points_boundary_r|1|s}
  y_{0}, z_{|s}, \ldots, z_{|1}, y_{1},z_{1|}, \ldots, z_{r|}.
\end{equation}
This space admits a natural compactification with top dimensional boundary strata:
\begin{align} \label{eq:first_stratum_bimodule_moduli}
& \coprod_{1 \leq S \leq R \leq r }   \scrR^{r-R+S|1|s} \times \scrR^{R-S+2} \\ \label{eq:second_stratum_bimodule_moduli}
&  \coprod_{\substack{1 \leq R \leq r \\ 1 \leq S \leq s}  }  \scrR^{r-R|1|s-S} \times  \scrR^{R|1|S} \\ \label{eq:third_stratum_bimodule_moduli}
&  \coprod_{1 \leq R \leq S \leq s }   \scrR^{r|1|s-S+R} \times \scrR^{S-R+2} .
\end{align}
Here, the middle stratum corresponds to the case where the points $y_{0}$ and $y_{1}$ belong to different components of the stable curve. These three degenerations are depicted in Figure \ref{Fig:BimoduleFacets}.
\begin{center}
\begin{figure}[ht]
\includegraphics[scale=0.4]{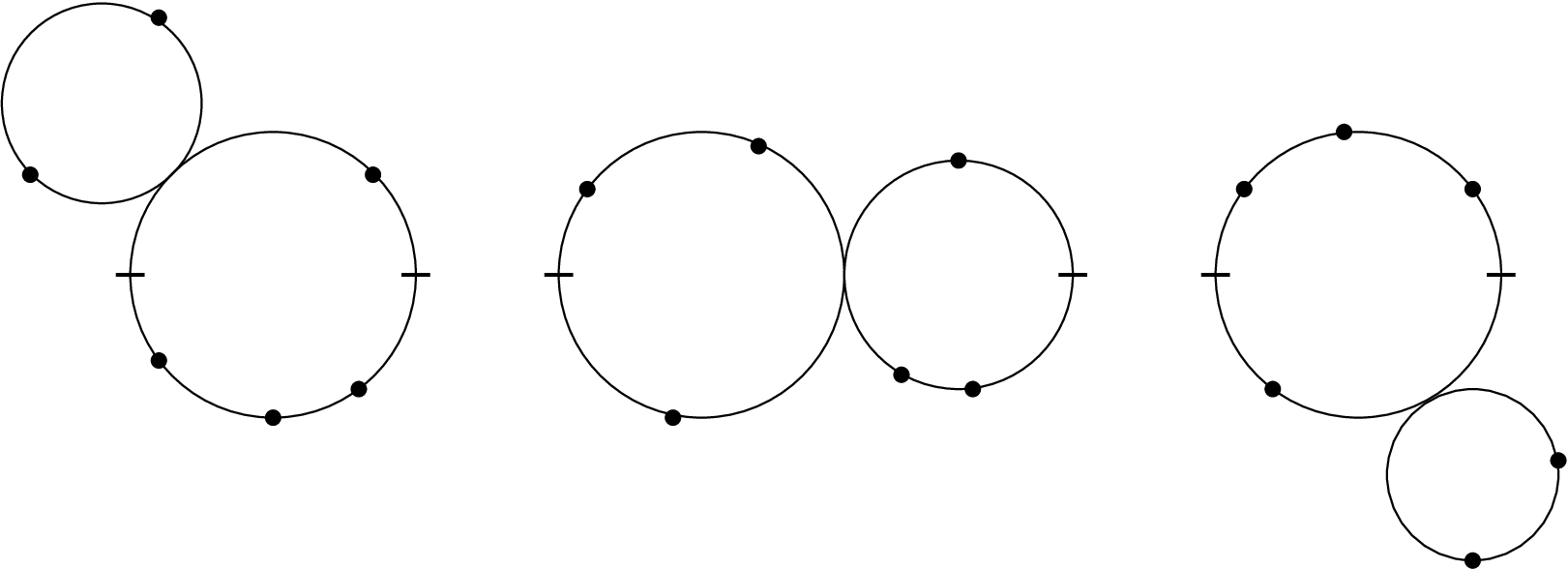}
\caption{Codimension one boundary facets to the moduli space $\scrR^{r|1|s}$\label{Fig:BimoduleFacets}}
\end{figure}
\end{center}

On a given surface, the Floer data defining the pseudo-holomorphic curve equation comprise a closed 1-form $\alpha_{\Sigma}$ vanishing on the boundary and a family of Hamiltonians $H_{z}$ on $\mathring{M}$ parametrised by $z \in \Sigma$. If $X_z$ is the Hamiltonian flow of $H_z$, the structure maps for the bimodule $\scrK$ are defined by solutions to the family of inhomogeneous equations 
\begin{equation} \label{eq:dbar_equation_discs_two_punctures}
(du - X_z \otimes \alpha_{\Sigma})^{0,1} = 0
\end{equation}
which are now parametrised by the universal family of punctured discs over $ \scrR^{r|1|s}  $.  
We require all the $H_z$ to be compactly supported, and use the product complex  structure on $\mathring{W}$. Near $y_{0}$ and $y_{1}$, the inhomogeneous term is induced by the function appearing in \eqref{eq:Hamiltonian_bimodule}, near $\{ z_{j|} | 1 \leq j \leq r \}$ by the Hamiltonians used in $\Fuk(\mathring{W})$, and near $\{z_{|j} | 1 \leq j \leq s\}$ by the pullbacks of the Hamiltonians used in the definition of $\Fuk(\mathring{M})$ under the projection $\mathring{W} \to \mathring{M}$.

More precisely, we first choose strip like ends $\{ \epsilon_{j} | j = 0,1 \}$, $\{ \epsilon_{j|} | 1 \leq j \leq r \}$, and  $\{\epsilon_{|j} | 1 \leq j \leq s\}$ near each of the punctures, which vary smoothly with respect to the modulus, and which are compatible near the boundary with  the data obtained from the boundary strata by gluing. The strip-like ends are incoming at all punctures except $y_{0}$, and outgoing at $y_{0}$, i.e. modelled on $Z_+$  at $\{y_{1}, z_{|j}, z_{i|}\}$ (respectively $Z_-$ at $y_{0}$), where 
\begin{equation} \label{eqn:in-out-striplike-ends}
  Z_- = (-\infty,0] \times [0,1] \textrm{ and } Z_+ = [0,\infty) \times [0,1]
\end{equation}
The inhomogeneous data consist of a $1$-form $\alpha$ on each punctured disc $\Sigma$ in $ \scrR^{r|1|s} $, and a Hamiltonian $H_{z}$ for each point $z \in \Sigma$. We require the pullback of $\alpha$ under each strip-like end to agree with $dt$.

To state the conditions on the Hamiltonian, fix sequences $(L_0, \ldots, L_{s})$ and $(Q_0, \ldots, Q_{r})$ of Lagrangian branes in $\mathring{M}$ and $\mathring{W}$, i.e. objects of $\Fuk(\mathring{M})$ and $\Fuk(\mathring{W})$. We then require that\footnote{We hope the use of $s$ as an index for boundary marked points on disks and as a co-ordinate on the strip $\bR\times [0,1]$ and hence for the Hamiltonians will cause no confusion.} 
\begin{align}
H_{\epsilon_{1}(s,t)} & = H_{L_{0}, Q_{0}}(t) \\
H_{\epsilon_{0}(s ,t)} & = H_{L_{s}, Q_{r}}(t) \\
H_{\epsilon_{i|}(s ,t)} & = H_{Q_{i-1}, Q_{i}}(t) \quad 1 \leq i \leq r \\
H_{\epsilon_{|j}(s ,t)} & = H_{L_{j}, L_{j-1}}(t) \quad 1 \leq j \leq s.
\end{align}
The Hamiltonians $H_{Q_{i-1}, Q_{i}} $ and $H_{L_{j}, L_{j-1}} $ are those which are respectively used in the definition of the Floer complexes $CF^*(Q_{i-1}, Q_{i})$ and $CF^*(L_{j}, L_{j-1})  $; in the second case, we omit the projection to $\mathring{M}$ from the notation. With these assumptions, finite energy solutions $u$ of Equation  \eqref{eq:dbar_equation_discs_two_punctures}, with boundary conditions given by $L_j'$ and $L_{j} \times K$  have the following convergence properties:
\begin{align}
\lim_{s \to +\infty} u \circ \epsilon_{1}(s ,t) & \in  \Chord(L_0 \times K, Q_0 ) \\
\lim_{s \to -\infty} u \circ \epsilon_{0}(s ,t) & \in  \Chord(L_s \times K, Q_r ) \\
\lim_{s \to +\infty} u \circ \epsilon_{j|}(s ,t) & \in  \Chord(Q_{i-1}, Q_{i} )  \quad 1 \leq i \leq r \\
\lim_{s \to +\infty} u \circ \epsilon_{|j}(s ,t) & \in  \Chord(L_{j}, L_{j-1} ) \times K \quad 1 \leq j \leq s.
\end{align}
Note that the last of these properties is implied by the removal of singularities theorem \cite{Oh}. Indeed, $u \circ \epsilon_{|j} $ is the product of a solution to Floer's equation on the half-strip with target $\mathring{M}$ and boundary conditions $L_{j-1}$ and $L_{j}$, with an ordinary holomorphic equation with target $\mathring{N}$ and boundary condition $K$. Oh's removal of singularities \cite{Oh} implies that the factor with target $\mathring{N}$ converges in the limit $s \to +\infty$ to a point in $K$.  (If there were no inhomogeneous perturbations, exactness of $K$ in $\mathring{N}$ would imply that the second factor was actually a constant disk.)

\begin{center}
\begin{figure}[ht]

\includegraphics[scale=0.4]{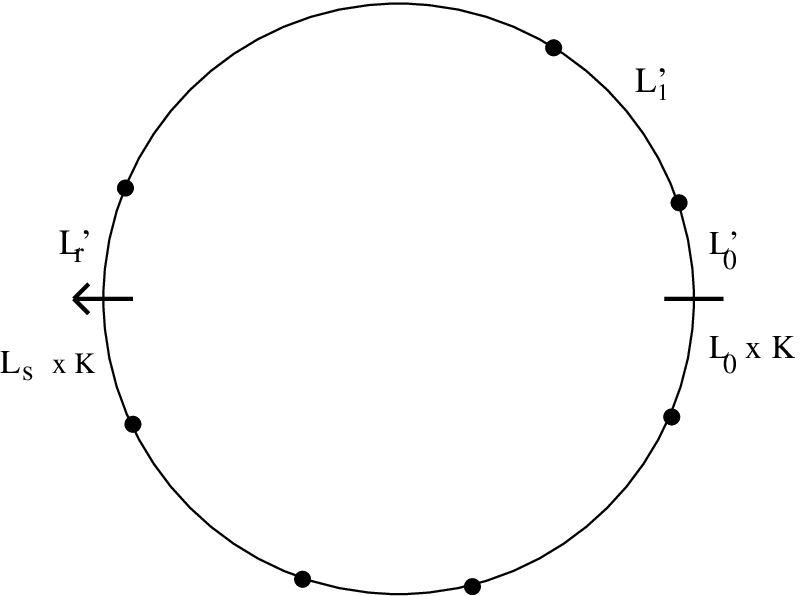}
\caption{Boundary conditions for the bimodule count\label{Fig:BimoduleLabels}}
\end{figure}
\end{center}

With this in mind, given chords
\begin{align} \label{eq:chord_bimodule_output}
  m_{0} & \in   \Chord(L_s \times K, Q_r )  \\
m_{1} & \in   \Chord(L_0 \times K, Q_0 )   \\
y_{i|} & \in  \Chord(Q_{i-1}, Q_{i} )  \quad 1 \leq i \leq r \\ \label{eq:chord_bimodule_input_M}
x_{|j} & \in  \Chord(L_{j}, L_{j-1} ), \quad 1\leq j\leq s
\end{align}
we define the moduli space
\begin{equation} 
   \scrR^{r|1|s}(m_{0}; y_{r|}, \ldots, y_{1|}, m_{1}, x_{|1}, \ldots, x_{|s})
\end{equation}
to be the space of  maps
\begin{equation}
  u \co \Sigma \to \mathring{W},
\end{equation}
for $\Sigma \in  \scrR^{r|1|s}$ with boundary segments mapping counter clockwise to 
\begin{equation}
(L_{s} \times K, \ldots, L_{0} \times K, Q_0, \ldots, Q_{r})  
\end{equation}
and asymptotic conditions along the ends given by the chords in Equations \eqref{eq:chord_bimodule_output}-\eqref{eq:chord_bimodule_input_M}, cf. Figure \ref{Fig:BimoduleLabels}. More precisely, near the puncture $z_{|j}$, we require that
\begin{equation}
  \lim_{s \to +\infty} \pi_{\mathring{M}} \circ u \circ \epsilon_{|j}(s ,t) = x_{|j}(t)
\end{equation}
where $\pi$ is the projection from $\mathring{W}$ to $\mathring{M}$. In particular, for each integer $1 \leq j \leq s$, we have natural evaluation maps
\begin{equation}\label{eq:just_now}
 \ev_{|j} \co  \scrR^{r|1|s}(\mathring{W}|m_{0}; y_{r|}, \ldots, y_{1|}, m_{1}, x_{|1}, \ldots, x_{|s}) \to K.
\end{equation}

For generic choices of Floer data, the domain of \eqref{eq:just_now} is a smooth manifold of dimension
\begin{equation}
\deg(m_{0}) + r + s -1 - \deg(m_{1}) - \sum_{j=1}^{s} \deg(x_{j})  - \sum_{i=1}^{r} \deg(y_{i}),
\end{equation}
and, having fixed brane data on all Lagrangians, it is naturally oriented relative to the orientation lines of the input and output chords. If
\begin{equation} \label{eq:degree_output_bimodule_operation}
  \deg(m_{0}) =  \deg(m_{1}) + \sum_{j=1}^{s} \deg(x_{j}) + \sum_{j=1}^{r} \deg(y_{i}) - r- s + 1,
\end{equation}
one therefore naturally associates to $u \in \scrR^{r|1|s}(\mathring{W}|m_{0}; y_{r|}, \ldots, y_{1|}, m_{1}, x_{|1}, \ldots, x_{|s}) $ a map
\begin{equation}
\ro_{y_{r|}} \otimes \cdots \otimes \ro_{y_{1|}} \otimes \ro_{m_{1}} \otimes  \ro_{x_{|1}} \otimes \cdots \otimes \ro_{x_{|s}} \to \ro_{m_{0}}.
\end{equation}
Taking the sum over all such pseudo-holomorphic curves $u$, for all collections of chords satisfying Equation \eqref{eq:degree_output_bimodule_operation}, we obtain the structure maps of $\scrK$ as an $A_{\infty}$-bimodule
\begin{equation}
\begin{aligned}
  \mu_{\scrK}^{r|1|s} \co \ CF^{*}(Q_{r-1}, Q_{r}) \otimes \cdots \otimes CF^{*}(Q_{0}, Q_{1}) \otimes   \scrK(L_{0},Q_{0})    \qquad \qquad &  \\ \qquad \qquad   \otimes\, CF^{*}(L_{1}, L_{0}) \otimes \cdots \otimes CF^{*}(L_{s}, L_{s-1}) \longrightarrow   \scrK(L_{s},Q_{r}).
  \end{aligned}
\end{equation}

To check that these operations satisfy the $A_{\infty}$-relation, we assume that the Floer data are compatible with the boundary decomposition in Equations \eqref{eq:first_stratum_bimodule_moduli}-\eqref{eq:third_stratum_bimodule_moduli} of the moduli spaces of domains. For the first stratum, this means that the component  that lies in $ \scrR^{R-S+2}$ carries the data used in the definition of $\Fuk(\mathring{W})$, while for the third stratum, the component  that lies in $ \scrR^{S-R+2}$ carries the product of the data used in the definition of $\Fuk(\mathring{M})$ with the (homogeneous) holomorphic curve equation with target $\mathring{N}$.  The strata of the moduli space  $\bar{\scrR}^{r|1|s}(\mathring{W}|m_{0}; y_{r|}, \ldots, y_{1|}, m_{1}, x_{|1}, \ldots, x_{|s})$ of virtual codimension $1$ are
%\\ \deg(y) = \deg(y_{R}) + \cdots + \deg(y_{S}) - k +2 }
%\deg(x) = \deg(x_{R}) + \cdots + \deg(x_{S}) - k +2
\begin{equation}
\begin{aligned} \label{eq:first_stratum_bimodule_moduli_maps}
 \coprod_{\substack{1 \leq S \leq R \leq r \\ y \in \Chord(Q_{S-1}, Q_{R})} }   \scrR^{r-R+S|1|s}(\mathring{W}|m_{0}; y_{r|}, \ldots, y_{R+1|}, y, y_{S-1|}, \ldots, y_{1|}, m_{1}, x_{|1}, \ldots, x_{|s})  & \\  \qquad \qquad \times \  \scrR^{R-S+2}(\mathring{W}|y; y_{R|}, y, y_{S|}) 
 \end{aligned} \end{equation} \begin{equation}
 \begin{aligned} \label{eq:second_stratum_bimodule_moduli_maps}
  \coprod_{\substack{1 \leq R \leq r \\ 1 \leq S \leq s \\ m \in \Chord(L_{S} \times K, Q_{R})  }}  \scrR^{r-R|1|s-S}(\mathring{W}|m_{0}; y_{r|}, \ldots, y_{R+1|},  m, x_{|S+1}, \ldots, x_{|s}) & \\ \qquad \qquad  \times  \scrR^{R|1|S}(\mathring{W}|m; y_{R|}, \ldots, y_{1|},  m_{1}, x_{|1}, \ldots, x_{|S})
  \end{aligned}\end{equation} \begin{equation}
  \begin{aligned}   \label{eq:third_stratum_bimodule_moduli_maps}
  \coprod_{\substack{1 \leq R \leq S \leq s \\  x \in \Chord(L_{S}, L_{R-1})} }   \scrR^{r|1|s-S+R}(\mathring{W}|m_{0}; y_{r|},  \ldots, y_{1|}, m_{1}, x_{|1}, \ldots, x_{|R-1}, x, x_{|S+1}, \ldots,x_{|s})  & \\ \qquad \times \scrR^{S-R+2}(\mathring{M} |x; x_{|R}, \ldots , x_{|S}) .
\end{aligned}
\end{equation}
The last stratum can be more formally described as the fibre product
\begin{equation} \label{eq:stratum_fibre_product}
\begin{aligned}
    \scrR^{r|1|s-S+R}(\mathring{W}|m_{0}; y_{r|},  \ldots, y_{1|}, m_{1}, x_{|1}, \ldots, x_{|R-1}, x, x_{|S+1}, \ldots,x_{|s}) _{\ev_{|R}}\times_{\pi_{\mathring{N}}}  & \\ \qquad \qquad \left( \scrR^{S-R+2}(\mathring{M} |x; x_{R}, x, x_{S})  \times K \right). \end{aligned}
\end{equation}
Indeed, the space of maps from elements of $\scrR^{S-R+2} $ to $\mathring{W}$ that we are considering splits as the product of $ \scrR^{S-R+2}(\mathring{M} |x; x_{R}, y, x_{S})  $ with holomorphic discs with boundary on $K$. Since $K$  is an exact Lagrangian, all such latter curves are constant. Since the fibre product in Equation \eqref{eq:stratum_fibre_product} is taken over $K$, we can remove this factor from the right side, and obtain Equation \eqref{eq:third_stratum_bimodule_moduli_maps}.

Since constant discs are regular, and the data defining $  \scrR^{S-R+2}(\mathring{M} |x; x_{R}, x, x_{S})  $ are regular (part of the hypotheses that they could be used to define the Fukaya category),  the right factor in Equation \eqref{eq:stratum_fibre_product} is also regular. Proceeding by induction on the value of the sum $r+s$, we conclude:
\begin{lem}
  For generic choice of Floer data, 
  \begin{equation}
    \bar{\scrR}^{r|1|s}(\mathring{W}|m_{0}; y_{r|}, \ldots, y_{1|}, m_{1}, x_{|1}, \ldots, x_{|s})
  \end{equation}
is a compact $1$-dimensional manifold with boundary whenever
\begin{equation}
    \deg(m_{0}) =  \deg(m_{1}) + \sum_{j=1}^{s} \deg(x_{j}) + \sum_{i=1}^{r} \deg(y_{i}) - r- s + 2.
\end{equation}
Its boundary strata are given by Equations \eqref{eq:first_stratum_bimodule_moduli_maps}-\eqref{eq:third_stratum_bimodule_moduli_maps}. \qed
\end{lem}

We now note that the stratum \eqref{eq:first_stratum_bimodule_moduli_maps} corresponds to the first term in Equation \ref{eq:relation_bimodule}, the stratum \eqref{eq:second_stratum_bimodule_moduli_maps} to the second term, and \eqref{eq:third_stratum_bimodule_moduli_maps} to the last term. We omit the discussion of signs, which is essentially the same as those for the usual $A_{\infty}$-relations in the Fukaya category; the details appear for instance in \cite{Ganatra}.  It follows that $\scrK$ does indeed define an $A_{\infty}$-bimodule. We will refer to  a bimodule $\scrK$ defined from a fixed Lagrangian brane $K$ in this way as an  \emph{elementary} bimodule.

Finally, we note that, fixing an object $L$ of $\Fuk(\mathring{M})$, the module over $\Fuk(\mathring{W})$ obtained by considering $\scrK(L \times K, \_)$ is quasi-isomorphic to the Yoneda module of $L \times K$. Appealing to Lemma \ref{lem:representable_bimod_functor}, we conclude that there is a functor
\begin{equation} \label{eq:kunneth_functor}
 \Fuk(\mathring{M}) \longrightarrow \scrF(\mathring{W}),
\end{equation}
unique up to quasi-isomorphism, which represents $\scrK$. Composing with the embedding of $\mathring{W}$ into $W$, we obtain: 
\begin{equation} \label{eq:kunneth_functor}
\scrF_K \co  \Fuk(\mathring{M}) \longrightarrow \scrF(W).
\end{equation}

\subsection{Equivariance for elementary bimodules}
We continue in the setting of the previous section, in particular assuming Hypothesies \ref{Hyp:GWvanishes}.  We have explained that, by counting  Maslov index two discs with two interior marked points, constrained to a half-line and meeting $D_0$ and $D_0'$,  this allows us to construct a cycle $b_D^0 \in C^1(L;\bZ)$, the vanishing of which amounts to existence of an equivariant structure on the brane $L$. \cite[Sections 3.5--3.9]{AbSm} explain that, given a choice of cochain $c_L$ with $dc_L = b_D^0$ for each object $L$ in (a subcategory of) $\scrF(M)$, by further considering moduli spaces of Maslov index two discs with two interior marked points and any number of boundary marked points, and with boundary constraints given by the cycles $c_L$ and the chains appearing in Hypothesis \ref{Hyp:GWvanishes}, one can construct the higher order terms of a \nc-vector field $b \in CC^1(\scrF(M),\scrF(M))$, with $b^0_M = b_D^0$. We refer to \cite{AbSm} for the details of the construction. Restricting to $\mathring{M}$, we obtain an  \nc-vector field on the Fukaya category of $\mathring{M}$, which, by Lemma \ref{lem:integrated_maximum_principle}, only depends on data defined on $\mathring{M}$ and $\mathring{\Mbar}$.

Suppose, then, we have \nc-vector fields $b^{\mathring{M}} \in CC^1(\Fuk(\mathring{M}),\Fuk(\mathring{M}))$ and $b^{\mathring{N}} \in CC^1(\Fuk(\mathring{N}),\Fuk(\mathring{N}))$, and $b^{W} \in CC^1(\Fuk(W),\Fuk(W))$ induced by choosing bounding cochains
\begin{alignat}{2}
  gw^{\mathring{M}}_{1} & \in C^1(\mathring{M} ; \bZ)  \quad  &  \beta^{\mathring{M}}_0 & \in C^1(D^{\mathring{\bar{M}}}_0 \setminus (D^{\mathring{\bar{M}}}_{0} \cap D^{\mathring{\bar{M}}}_{r})  \cup D_{0}^{\mathring{M},sing}  ; \bZ) \\ 
 gw_{1}^{\mathring{N}} & \in C^1(\mathring{N} ; \bZ)  \quad  & \beta_0^{\mathring{N}} & \in C^1(D_0^{\mathring{N}} \setminus (D_{0}^{\mathring{\bar{N}}} \cap D_{r}^{\mathring{\bar{N}}})  \cup D_{0}^{\mathring{\bar{N}},sing}  ; \bZ)  \\
  gw^{W}_{1} & \in C^1(W ; \bZ)  \quad  &  \beta^{W}_0 & \in C^1(D^{W}_0 \setminus (D^{W}_{0} \cap D^{W}_{r})  \cup D_{0}^{W,sing}  ; \bZ) 
\end{alignat}
For the bounding cochains on $W$, we assume that our chosen push-off of $D^{W}_0$ is given, in $\mathring{\bar{W}}$, by
\begin{align}
  D^{\prime, \mathring{\bar{W}}}_0 & = (D^{\prime,\mathring{\bar{M}}}_0 \times \mathring{\bar{N}})   \cup  (\mathring{\bar{M}}  \times D^{\prime,\mathring{\bar{N}}}_0).
\end{align}
We write $ gw^{\mathring{W}}_1 $ and $\beta^{\mathring{W}}_0 $ for the restrictions to $\mathring{W}$, and assume that we have equalities
\begin{align}
  gw^{\mathring{W}}_1 & = gw^{\mathring{M}}_1 \times [\mathring{N}] + [\mathring{M}] \times gw_1^{\mathring{N}}  \in C^1(\mathring{M} \times \mathring{N}; \bZ) \\ 
\beta^{\mathring{W}}_0 & = \beta^{\mathring{M}}_0 \times [\mathring{\Nbar}] + [\mathring{\Mbar}] \times gw_1^{\mathring{N}} \in C^1(D^{\mathring{W}}_0 \setminus (D^{\mathring{W}}_{0} \cap D^{\mathring{W}}_{r})  \cup D^{\mathring{W},sing}_{0}  ; \bZ)
\end{align}
where the fundamental cycles are defined by taking the sum of all top-dimensional simplices for a chosen triangulation. 
\begin{rem}
It would be sufficient to assume that $gw^{\mathring{W}}_1  $ and $\beta^{\mathring{W}}_0 $ differ by a boundary from the product chains, but this would complicate the discussion by introducing more notation.
\end{rem}

Let us now, in addition, assume that $K$ is equivariant with respect to $b^{\mathring{N}}$ over the field $\bk$. In particular, we fix a cochain
\begin{equation} \label{eq:bounding_cochain_K}
  c_{K} \in C^{0}(K; \bk)
\end{equation}
whose boundary is
\begin{equation} 
 b_{N}^{0} + gw_1^{\mathring{N}}|K + \sCO^{0}( \beta_0^{\mathring{N}}) \in  C^{1}(K; \bk).
\end{equation}

 For $k\geq 0$, let $\Mod{(0,1)}{k+1}$ denote the moduli space of domains comprising the disc $\Delta$
\begin{enumerate}
\item with two marked points  $z_0 = 0$ and $z_1  \in (0,1)$, and
\item with $k+1$ boundary punctures at  $p_0 = 1 \in\partial \Delta$ and points $\{p_1,\ldots,p_k\} \subset \partial \Delta \backslash \{1\}$ ordered counter-clockwise.
\end{enumerate}
It is important to note that the point $1$ will play the role of an output, while the points $p_i$ will correspond to inputs. 
Introduce the moduli spaces
\begin{equation} \label{eq:identify_mod_r_1_s}
    \Mod{(0,1)}{r|1|s} \cong   \Mod{(0,1)}{r+s+2},
\end{equation}
where the marked points on the boundary are labelled as in Equation \eqref{eq:marked_points_boundary_r|1|s}. We can describe the boundary of these moduli spaces exactly as in Section \ref{Sec:BimodulesWithoutQuilts}, except that it is more important to keep track of the position of the marked point $y_{1}$. In particular, the strata in Equations \eqref{eq:first_stratum_bimodule_moduli} split as
\begin{align} \label{eq:boundary_mod_r_1_s_first}
  &  \coprod  \Mod{(0,1)}{r-R+S|1|s} \times \scrR^{R-S+2}  \\
& \coprod  \Mod{(0,1)}{r-R|1|s-S} \times \scrR^{R|1|S}  \\ \label{eq:boundary_mod_r_1_s_third}
& \coprod  \Mod{(0,1)}{r|1|s-S+R} \times \scrR^{S-R+2} 
\end{align}
while the boundary strata in Equation  \eqref{eq:third_stratum_bimodule_moduli} split as
\begin{align} \label{eq:boundary_mod_r_1_s_fourth}
  &  \coprod  \scrR^{r-R+S|1|s} \times\Mod{(0,1)}{R-S+2}  \\ \label{eq:boundary_mod_r_1_s_fifth}
& \coprod \scrR^{r-R|1|s-S} \times  \Mod{(0,1)}{R|1|S}  \\ \label{eq:boundary_mod_r_1_s_sixth}
& \coprod   \scrR^{r|1|s-S+R} \times \Mod{(0,1)}{S-R+2} 
\end{align}

Given sequences $(L_0, \ldots, L_{s})$ and $(Q_0, \ldots, Q_{r})$ of Lagrangian branes in $\mathring{M}$ and $\mathring{W}$, and chords as in Equations  \eqref{eq:chord_bimodule_output}-\eqref{eq:chord_bimodule_input_M}, we obtain a moduli space
\begin{equation} \label{eq:5}
   \Mod{(0,1)}{r|1|s} (m_{0}; y_{r|}, \ldots, y_{1|}, m_{1}, x_{|1}, \ldots, x_{|s})
\end{equation}
of maps $u$ whose source lies in $    \Mod{(0,1)}{r|1|s} $, with boundary conditions as in Figure \ref{Fig:BimoduleLabels} and such that
\begin{equation}
  u(z_0) \in D_0 \textrm{ and } u(z_1) \in D'_0.
\end{equation}
These moduli spaces define a cochain 
\begin{equation}
  \tilde{c}_{\scrK} \in \mathcal{E}_{\scrB,\scrA}(\scrK)
\end{equation}
whenever $\scrA \subset \scrF(\mathring{M})$ and $\scrB \subset \scrF(\mathring{W})$ are subcategories of the Fukaya categories of $\mathring{M}$ and $\mathring{W}$ containing the branes $L_j$ respectively $Q_i$. In order to produce a cycle, we must add the contributions $\sCO_{\scrK} ( \beta^{\mathring{W}}_0)$  and $ \CO_{\scrK}(gw^{\mathring{W}}_1)$ which are defined in analogy with the corresponding structures at the level of categories; here $\CO_{\scrK}$ is the open-closed map
\[
\CO_{\scrK} \colon C^*(\mathring{W};\bk) \longrightarrow \mathcal{E}_{\scrB,\scrA}(\scrK)
\]
which is induced from the usual open-closed map of $\mathring{W}$.  In addition, we define an operator
\begin{equation}
  \nu(c_{K}) \in \mathcal{E}_{\scrB,\scrA}(\scrK)
\end{equation}
by counting the moduli spaces
  \begin{equation} \label{eq:fibre_product_moduli_product_bounding_cochain}
   \Mod{(0,1)}{r|1|s} (m_{0}; y_{r|}, \ldots, y_{1|}, m_{1}, x_{|1},  \ldots , x_{|s}) _{\ev_{|R}}\times_{K}  c_{K}
\end{equation}
(where $R$ varies). 
Here, $c_K$ is represented by a cycle in $K$, e.g. in terms of descending manifolds of critical points of a Morse function.

With this in mind, we define
\begin{equation} \label{eq:K_equivariant_structure}
  c_{\scrK} \equiv  \tilde{c}_\scrK  + \sCO_{\scrK} ( \beta^{\mathring{W}}_0)  + \CO_{\scrK}(gw^{\mathring{W}}_1) +   \nu(c_{K}) \in \mathcal{E}_{\scrB,\scrA}(\scrK).
\end{equation}

\begin{prop}
 The cochain $c_{\scrK}$ defines an equivariant structure on $\scrK$.
\end{prop}
\begin{proof}[Sketch of proof]
The reader may wish to compare to the corresponding assertion in \cite[Proposition 3.20]{AbSm}.   We must verify Equation \eqref{Eqn:EquivtBimodule}, and identify the boundary strata of the moduli spaces with the terms in this equation. We discuss the essential part of the computation, neglecting the contributions of the maps $\CO$ and $\sCO$ which are straightforward adaptations of the situation for categories discussed in the prequel.

First, we use the embedding $\mathring{\bar{W}} \subset \Mdbar \times \Ndbar$ to see that all moduli spaces of stable discs that we are considering are compact.  The left hand side of Equation \eqref{Eqn:EquivtBimodule} corresponds to the boundary stratum of $ \Mod{(0,1)}{r+s+1} $ given in Equations \eqref{eq:boundary_mod_r_1_s_first}-\eqref{eq:boundary_mod_r_1_s_third} and Equation \eqref{eq:boundary_mod_r_1_s_sixth}. The two terms on the right hand side of Equation \eqref{Eqn:EquivtBimodule} respectively correspond to the boundary strata in Equations \eqref{eq:boundary_mod_r_1_s_fourth} and those in Equation \eqref{eq:boundary_mod_r_1_s_sixth} for which the map on the factor $\Mod{(0,1)}{S-R+2} $ is constant in the factor $\mathring{N}$.  If the map is not constant in this factor, then we obtain
\begin{equation}
   \Mod{(0,1)}{r|1|s} (m_{0}; y_{r|}, \ldots, y_{1|}, m_{1}, x_{|1},  x_{|s}) _{\ev_{|R}}\times_{K}   \Mod{(0,1)}{1}(K).
\end{equation}
In the absence of contributions from the maps $\sCO$ and $\CO$, this is precisely the boundary of the moduli space in Equation \eqref{eq:fibre_product_moduli_product_bounding_cochain}, and it is accounted for by the term $ \nu(c_{K})  $ in Equation \eqref{eq:K_equivariant_structure}.
\end{proof}

\subsection{Purity for elementary bimodules} \label{sec:equiv-elem}

We now further suppose that:

\begin{Hypothesis} \label{Hypothesis:Pure} The \nc-vector field $b_{M}$ is pure on $\scrA \subset \Fuk(\mathring{M})$, and $K \in \Fuk(\mathring{N})$ is pure.
\end{Hypothesis}
Denote by $\scrB = \scrA \times \{ K \}$ the full subcategory of $\Fuk(\mathring{W})$ comprising objects of the form $L' = L \times K$, for $L\in\scrA$.  The brane $K$ defines an elementary $(\scrB-\scrA)$-bimodule $\scrK$, as in the previous section.
\begin{Lemma} \label{lem:bimodule-pure}
Assuming Hypothesis \ref{Hypothesis:Pure}, $c_{\scrK} $ defines a pure equivariant structure on $\scrK$.
\end{Lemma}
\begin{proof}
After unwinding definitions, this is a repackaging of the K\"unneth theorem in Floer cohomology.  Let $L_i \in \scrA$,  for $i=1,2$, and let $L_i' = L_i\times K \in \scrB \subset \Fuk(\mathring{N})$.  By choosing the Hamiltonian perturbation in Equation \eqref{eq:Hamiltonian_bimodule} to be the sum of the Hamiltonians on $\mathring{M}$ and $\mathring{N}$, we obtain an isomorphism of cochain complexes
\begin{equation} \label{first}
\scrK(L_1,L_2') \equiv CF^*(L_1,L_2) \otimes CF^*(K,K).
\end{equation}
The equivariant structure $c_{\scrK}$, together with the given \nc-vector fields on $\scrA$ and $\scrB$,  yields an endomorphism  of the cochain complex $\scrK(L_1,L_2')$ via \eqref{eq:equivariant_endo}. 
By forcing all equations to split, our choices of auxiliary data ensure that this endomorphism also splits as a sum
\begin{equation}
     \id \otimes b_{N}^{1} + b^{1}_M \otimes \id \ \in \ \mathrm{Aut}(\scrK(L_1,L_2')).
\end{equation}
Purity for $c_{\scrK}$ therefore follows from purity for $b_{N}^{1} $ and $b^{1}_M$.
\end{proof}
\begin{cor} \label{cor:Kunneth-formal}
The K\"unneth functor $\scrF_K: \scrA \rightarrow \scrB$ of \eqref{eq:kunneth_functor} is formal relative to the formality quasi-isomorphisms $\scrA \cong H \scrA$ and $\scrB \cong H \scrB$ induced by the equivariant structures on $\scrA$ and $\scrB$.
\end{cor}
\begin{proof}
This follows from the discussion before Proposition \ref{prop:purity_implies_formal_functor} together with Lemma \ref{lem:bimodule-pure}.
\end{proof}
%\begin{rem}
  %Note that, by constructing a bimodule and appealing to representability of bimodules,  we have avoided altogether appealing to a  K\"unneth theorem on chain-level $A_{\infty}$ structures.
%\end{rem}

\section{Cup bimodules}

\subsection{The Milnor fibre and the Hilbert scheme} \label{Sec:Hilb}

Let $A_{2n-1}$ denote the Milnor fibre
\[
\left \{ x^2 +y^2 + \prod_{i=1}^{2n} (z-i) = 0\right\} \subset \bC^3
\]
equipped with the restriction of the standard K\"ahler form from $\bC^3$. 
This is the total space of a Lefschetz fibration (from projection $\pi$ to the $z$-plane) with nodal fibres over $\{1,2,\ldots,2n\}$.  The Hilbert scheme of $n$ points $\Hilb^{[n]}(A_{2n-1})$ contains a distinguished divisor $D_{rel}$  which is the ``relative Hilbert scheme" of the projection to the $z$-plane. (This is a minor abuse of notation, initiated in the prequel, which comes from the fact that when $n=2$, $D_{rel}$ is closely related to the relative second symmetric product of that projection.)  More precisely, the inclusion $\bC[z] \subset \bC[x,y,z]$ yields an inclusion
\[
R = \bC[z] \cong \bC[z]/(\bC[z] \cap \langle x^2 +y^2 + \prod_{i=1}^{2n} (z-i)\rangle ) \subset \bC[x,y,z]/ \langle x^2 +y^2 + \prod_{i=1}^{2n} (z-i)\rangle = \mathcal{O}(A_{2n-1})
\]
and hence an ideal $\mathcal{I} \subset \mathcal{O}$ yields an ideal $\mathcal{I} \cap R \subset R$, which we regard as the projection of the subscheme $\mathcal{O}/\mathcal{I}$ under projection $A_{2n-1} \to \bC_z$. The scheme $R/(I \cap R)$ has length at most $n$, and we define $D_{rel}$ to comprise the ideals $I$ for which it has length $<n$.  When $n=2$, the relative Hilbert scheme is in fact smooth, being a desingularisation of the relative second symmetric product of the fibration over $\bC_z$. Note that there are subschemes supported on a single fibre which do \emph{not} lie in the relative Hilbert scheme: for instance, any length two subscheme supported at a smooth point of a fibre with infinitesimal tangent direction not tangent to the fibre.  The definition of the relative Hilbert scheme, in our sense,  extends naturally to fibred surfaces $Z \to B$ which are not globally affine: it is straightforward to define $D_{rel}^{n=2} \subset \Hilb^{[2]}(Z)$ by the interpretation as a desingularised relative symmetric product, and we can define the divisor  $D_{rel}$ for larger $n$ to be the image of the natural map $D_{rel}^{n=2} \times \Hilb^{[n-2]}(Z) \to \Hilb^{[n]}(Z)$.

 Let $\scrY_n \subset \Hilb^{[n]}(A_{2n-1})$ denote the open subset of the Hilbert scheme given by removing the relative Hilbert scheme (as defined above).   By Manolescu's work \cite{Manolescu}, this is holomorphically isomorphic to the Springer fibre appearing in \cite{SS} (see Section \ref{Sec:Flags}). We equip $\scrY_n$ with an exact K\"ahler form which is product-like (induced from the product form on $A_{2n-1}^n$) outside an arbitrarily small neighbourhood of the diagonal.

Let $p: \mathrm{Bl}_0(\bC^2) \to \bC^2$ denote the blow-up at the origin. The composition of this  with the trivial projection $\bC^2 \to \bC$, $(x,y) \mapsto x$, is a Lefschetz fibration. Indeed, $\mathrm{Bl}_0(\bC^2) = \{((x,y), [z:w]) \in \bC^2\times \bP^1 \, | \, xw=yz\}$ and in the chart $w=1$ the map becomes $\{x=yz\} \mapsto \bC_x$, which is just a presentation of the usual Lefschetz singularity.  It follows that the blow-up $Z$ of $\bP^1\times \bP^1$ at $2n$ points lying along a fixed section of the trivial projection $\bP^1 \times \bP^1 \to \bP^1$ is a Lefschetz fibration with $2n$ singular fibres, which we may view as a compactification $\bar{\bar{A}}_{2n-1}$ of the Milnor fibre.   $Z$ contains divisors defined by sections $s_0, s_{\infty}$ of its Lefschetz fibration structure over $\bP^1$ (one of which is the proper transform of the section that was blown up, the other coming from a disjoint section) and a smooth fibre $F_{\infty}$ at infinity. We will write $\bar{A}_{2n-1}$ for the properification of $A_{2n-1} \to \bC$ which is the complement of $F_{\infty} \subset Z$; this admits a Lefschetz fibration over $\bC$ with general fibre $\bP^1$ and $2n$ reducible fibres $\bP^1 \vee \bP^1$.  In the notation of Section \ref{Sec:Set-up}:

\begin{itemize}
\item The projective variety is $\Mdbar = \Hilb^{[n]}(Z)$.
\item $D_0$ is the divisor of subschemes whose support meets $s_0 \cup s_{\infty}$ in $Z$.  
\item $D_{\infty}$ is the divisor of subschemes whose support meets $F_{\infty}$.
\item $D_r \subset \Mdbar$ is the relative Hilbert scheme in the sense defined above.  % \item  $B_r$ is the locus of points in $\Mdbar$ lying on a stable Chern one sphere which meets $D_r$.  
\end{itemize}

The previous discussion  implies that $M=\scrY_n$ and $\Mbar = \Mdbar \backslash D_{\infty} = \Hilb^{[n]}(\bar{A}_{2n-1})$. By taking push-offs of the sections $s_0, s_{\infty}$ in $Z\backslash F_{\infty}$ one obtains the linearly equivalent divisor $D_0'$.  When we consider bimodules over products $\scrY_n \times \scrY_{n-1}$ we compactify each factor as above.  That Hypotheses \ref{Hyp:Main} and  \ref{Hyp:GWvanishes} hold for these choices (in particular vanishing of the Gromov-Witten invariant $GW_1$ in this situation) is proved in \cite[Section 6]{AbSm}, which also gives an explicit choice of bounding cochain $gw_1$.  For the strengthened version of Hypothesis \ref{Hyp:Main} part (3), cf. Remark \ref{Remark:Duke_correction}, the fact that the holomorphic volume form on the Hilbert scheme has no zeroes or poles on $D_r$ follows from the explicit formula for the volume form given in \cite[p.766]{Beauville}.

\subsection{Crossingless matchings}

Any path $\gamma$ between critical values of $A_{2n-1} \rightarrow \bC$ defines a Lagrangian matching sphere, so an $n$-tuple of pairwise disjoint paths $\wp$ which join the critical points in pairs defines a Lagrangian submanifold $L_{\wp} \cong (S^2)^n \subset \scrY_n$.  The ``Markov I'' move of \cite[Lemma 48]{SS} implies that one can slide any arc of $\wp$ over any other arc (keeping the end-points fixed) without changing the Hamiltonian isotopy class of the associated Lagrangian submanifold $L_{\wp}$, even though the corresponding isotopy necessarily passes through the diagonal locus in the Hilbert scheme (where the K\"ahler form is not product-like).

\begin{Definition}
The symplectic arc algebra $\scrH_n^{symp}$ is the graded algebra $\oplus_{\wp,\wp'} HF^*(L_{\wp}, L_{\wp'})$, where  $\wp, \wp'$ run over the set of upper-half-plane crossingless matchings in $\bC$. The cochain level $A_{\infty}$-algebra will be denoted $ C\scrH_n^{symp} $.
\end{Definition}

Since $\scrY_n$ is exact and convex at infinity, and the Lagrangians $L_{\wp}$ are simply-connected (hence exact), Floer cohomology is defined unproblematically.  The Lagrangians $L_{\wp}$ are orientable and admit unique spin structures, so one can work in characteristic zero.  There are $\frac{1}{n+1}{2n \choose n}$ upper half-plane matchings, so the underlying graded algebra of $\scrH_n^{symp}$ is a finite-dimensional $\bk$-algebra.  In fact $c_1(\scrY_n)=0$ and the Lagrangians $L_{\wp}$ have vanishing Maslov class, hence admit gradings in the sense of \cite{seidel-graded}.  Any choice of gradings for the $L_{\wp}$ equips  the graded algebra $\scrH_n^{symp}$ with an absolute $\bZ$-grading.  There is no obvious preferred choice of grading, and indeed two distinct gradings will be important in the argument: one which is ``symmetric", which is useful for purity, discussed in Section \ref{Sec:Purity}; and one which is concentrated in even degrees, introduced in Section \ref{Sec:Basis}, which is useful for controlling signs in the cohomological algebra.  

For an upper-half-plane crossingless matching $\wp$, we will denote by $\bar{\wp}$ its reflection into the lower half-plane.  An iterative application of the Markov I move shows that $L_{\wp}$ is Hamiltonian isotopic to $L_{\bar{\wp}}$. For any upper half-plane matchings $\wp, \wp'$, the Lagrangians $L_{\wp}$ and $L_{\bar{\wp'}}$ meet transversely. As shown in \cite[Section 4.5]{AbSm}, the mod 2 degrees of all intersection points co-incide, so the Floer differential in $CF^*(L_{\wp}, L_{\bar{\wp'}})$ vanishes identically.   (As noted above, we will introduce a different grading on the symplectic arc algebra in Section \ref{Sec:Basis};  in the current situation, the  Floer differential vanishes because the relevant moduli spaces are empty, and hence vanishes for any choice of grading.) There is therefore a vector space isomorphism
\[
\scrH_n^{symp} \ \cong \ \bigoplus_{\wp, \wp'} CF^*(L_{\wp}, L_{\bar{\wp'}}).
\] 
As a relatively graded group, $CF^*(L_{\wp}, L_{\bar{\wp'}}) \cong H^*(S^2)^{\otimes c(\wp,\wp')}$, where $c(\wp,\wp')$ is the number of components of the planar unlink $\wp \cup \bar{\wp'}$, see \cite[Proposition 5.12]{AbSm}.

%%%%%%%%%%%%%%%%%%%%%%%%%%%%%%%%%%%%

\subsection{Reduction to a basic case}\label{Sec:MeetInCodim2}

We will say that crossingless matchings $\wp, \wp'$ of $2n$ points \emph{meet in codimension} $k$, or that the corresponding Lagrangian submanifolds have that property, if $HF^*(L_{\wp}, L_{\wp'}) \cong H^*(S^2)^{\otimes (n-k)}$ has real rank $2n-2k$.  (The terminology is inherited from the structure of the ``compact core'' of the quiver variety underlying $\scrY_n$, cf. Lemma \ref{Lem:Wehrli}.)  Thus, in general, $\wp$ and $\wp'$ meet in codimension $n-c(\wp,\wp')$. The following is then manifest:

\begin{Lemma} \label{Lem:Meet_codimension_one}
The matchings $\wp$ and $\wp'$ meet in codimension zero if and only if $\wp = \wp'$, and meet in codimension one if and only if $\wp$ and $\wp'$ share exactly $n-2$ arcs. \qed
\end{Lemma}
If $\wp$ and $\wp'  $ meet in codimension one, then, for any crossingless matching $\wp''$, the codimension of their intersection with $\wp''$ differs by one, i.e. 
\begin{equation}
  c(\wp,\wp'') =   c(\wp',\wp'') \pm 1.
\end{equation}
The following result asserts that there is a way of interpolating between any two crossingless matchings so that the codimensions form a monotone sequence.
%\begin{proof} Straightforward. \end{proof}

\begin{Lemma} \label{Lem:Codim1Interpolate}
For any pair of crossingless matchings $\wp_0$, $\wp_k$ which meet in codimension $k$, there is an interpolating sequence
\[
\wp_0, \wp_1, \ldots, \wp_k
\]
such that $\wp_i$ and $\wp_{i+1}$ meet in codimension one, $\wp_0$ and $\wp_{i}$ in codimension $i$,  and $\wp_i$ and $\wp_{k}$ in codimension $k-i$.
\end{Lemma}

\begin{proof} The proof is by induction on the the codimension. In the inductive step, we pick any arc in $\wp_k $ which does not lie in $\wp_i$; we define $\wp_{i+1}$ to be the unique matching containing this arc which meets $\wp_{i}$ in codimension $1$ (this process is illustrated in Figure \ref{Fig:IntermediateLags}). \end{proof}

\begin{center}
\begin{figure}[ht]
\includegraphics[scale=0.15]{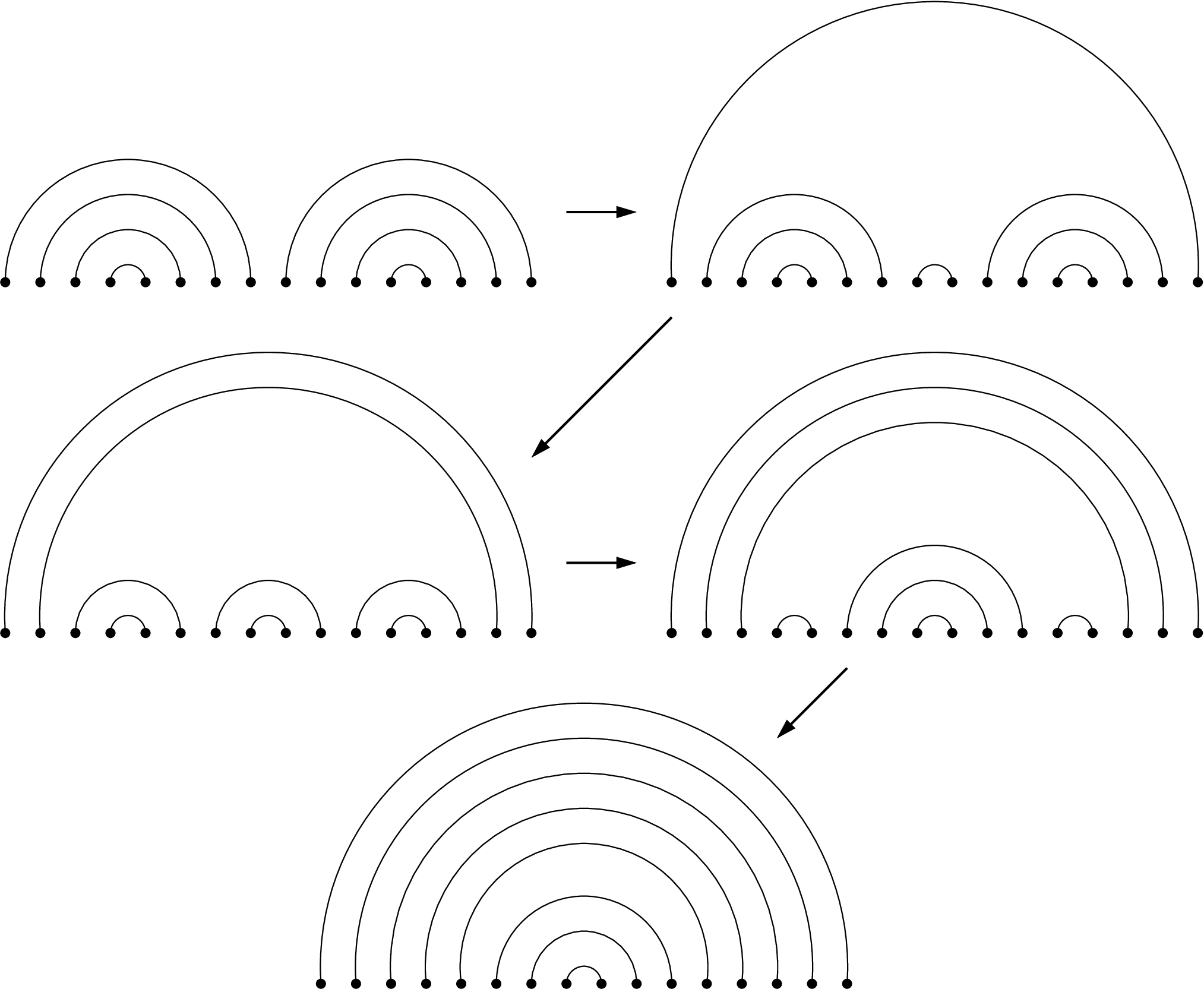}
\caption{Interpolating crossingless matchings by ones which intersect in codimension one\label{Fig:IntermediateLags}}
\end{figure}
\end{center}

Since $CF^*(L_{\wp}, L_{\bar{\wp'}}) \cong H^*(S^2)^{\otimes c(\wp,\wp'})$ as a relatively graded group, there is a well-defined rank one subspace 
\begin{equation} \label{Eqn:Min_degree}
k_{min}(\wp,\wp') \subset CF^*(L_{\wp}, L_{\bar{\wp'}}) \subset \scrH_n^{symp}
\end{equation} which is spanned by any minimal degree generator.   This subspace does not depend on the choice of graded structures on the Lagrangians.

\begin{Lemma}  \label{Lem:Constant_triangle}
Let $\wp, \wp'$ and $\wp''$ be crossingless matchings such that $\wp$ and $\wp'$ meet in codimension one.  Suppose moreover that $c(\wp,\wp'') = c(\wp',\wp'') -1$.  Then the Floer product
\[
k_{min}(\wp', \wp'') \otimes k_{min}(\wp,\wp') \longrightarrow k_{min}(\wp, \wp'')
\]
is non-trivial.
\end{Lemma}

\begin{proof} This follows from Lemma \ref{lem:Conv3} in conjunction with Lemma \ref{lem:OneArcChanged}, which together give an explicit plumbing model for the relevant Floer product.  One can also give a direct proof.    Lemma \ref{Lem:Meet_codimension_one} implies that $\wp$ and $\wp'$ share $n-2$ arcs, and the remaining two arcs belong to distinct components of $\wp' \cup \wp''$ but the same component of $\wp \cup \wp''$.  Place $\wp''$ in the lower half-plane and $\wp, \wp'$ in the upper half-plane. The minimal degree generator for $CF^*(L_{\wp}, L_{\wp'})$ comprises the odd intersection points if the critical points are labelled $\{1,\ldots, 2n\}$.   One sees that the constant triangle is a solution; that it is then regular follows from the fact that the triple of arcs at each intersection point in the $A_{2n-1}$-surface has the correct cyclic order, compare to \cite[Figure 14 \& Lemma 5.18]{AbSm}. Since the Lagrangians are exact, the area of any contributing holomorphic triangle is determined by the actions of the isolated intersection points, which implies that the constant triangle is then the only contribution to the product. 
\end{proof}

\subsection{Purity of the symplectic arc algebra}\label{Sec:Purity}

\newcommand{\mix}{\wp_\mathrm{mix}}

We summarise the main result of \cite{AbSm}. Certain of the Lagrangian submanifolds $L_{\wp}$ have a special role.
\begin{itemize}
\item The Lagrangian associated to the crossingless matching comprising a sequence of adjacent arcs joining the critical values in pairs $\{1,2\}, \{3,4\}, \ldots, \{2n-1,2n\}$ is denoted $L_{\wpb}$; this is the \emph{plait} matching.  
\item The Lagrangian associated to the crossingless matching comprising a sequence of nested arcs joining the critical values in pairs $\{1,2n\}, \{2,3\}, \ldots, \{2n-2,2n-1\}$ is denoted $L_{\mix}$; this is the \emph{mixed} matching.
\end{itemize}
\begin{center}
\begin{figure}[ht]
\includegraphics[scale=0.4]{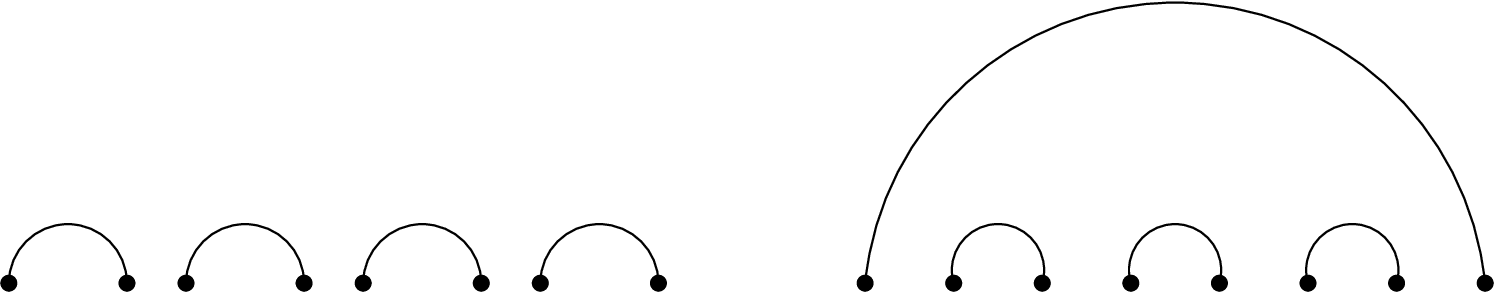}
\caption{The plait (left) and mixed (right) matchings of $2n$ points\label{Fig:PlaitMixed}.}
\end{figure}
\end{center}

\begin{Proposition}\cite{AbSm} \label{Prop:hittopclass}
For any $\wp, \wp'$, the Floer cohomology $HF^*(L_{\wp}, L_{\wp'})$ is a cyclic module, generated by $k_{min}(\wp,\wp')$, over each of $HF^*(L_{\wp}, L_{\wp}) $ and $HF^*(L_{\wp'}, L_{\wp'}) $.  Furthermore, the rank one subspace  of $HF^*(L_{\wp}, L_{\wp'})$ of largest cohomological degree lies in the image of the product 
\begin{equation} \label{eqn:keyproduct}
HF^*(L_{\wpb}, L_{\wp'}) \otimes HF^*(L_{\wp}, L_{\wpb}) \longrightarrow HF^*(L_\wp, L_{\wp'}).
\end{equation} \qed
\end{Proposition}

Given a grading on  the plait Lagrangian $L_{\wpb}$, there is a unique grading on each matching Lagrangian $ L_{\wp}$ so that $HF^*(L_{\wpb}, L_{\wp})$ is symmetrically graded in the sense that the groups $HF^*(L_{\wp}, L_{\wpb})$ and $HF^*(L_{\wpb}, L_{\wp})$ are supported in the same range of degrees  $n-c(\wp,\wpb) \leq \ast \leq n+c(\wp,\wpb)$. 

The Gromov-Witten invariant $GW_1 \in H^*(\scrY_n)$ is proved to vanish in \cite{AbSm}; since the $L_{\wp}$ are simply-connected, it follows that they admit equivariant structures (relative to the $\nc$-vector field constructed from counting discs in the compactification $\Mdbar$).  Given a choice of brane structures on the Lagrangians (which fix absolute gradings on Floer cohomology groups), we wish to choose these equivariant structures so that the weight gradings enjoy the same symmetry as the cohomological gradings.
\begin{Theorem} \cite{AbSm} \label{Thm:Arc_is_Pure}
For any choice of equivariant structure on $L_{\wpb}$ , there are unique equivariant structures on the Lagrangians $ L_{\wp}$ such that $C\scrH^{symp}_n $ is pure. \qed
\end{Theorem}
By Theorem \ref{Thm:Pure}, we obtain a fixed quasi-equivalence between $C\scrH^{symp}_n $ and its cohomological graded algebra, which is $\scrH^{symp}_n$.

\subsection{Some elementary bimodules} \label{Sec:MeetCup}

Consider $\scrY_n \subset  \Hilb^{[n]}(A_{2n-1})$, with $z: A_{2n-1} \rightarrow \bC$ having  critical values at $\{1,2,\ldots, 2n\}  \subset \bC$.  A choice of domain $U \subset \bC$ yields subdomains $A_{2n-1}|_U \subset (\bar{A}_{2n-1})|_U$, by restricting $z: A_{2n-1} \to \bC$ and its properification $\bar{A}_{2n-1} \to \bC$ to $U$, and hence yields an open subset 
\[
\mathring{\bar{M}}  = \bar{M}|_U = \Hilb^{[n]}((\bar{A}_{2n-1})|_U).
\]
Analogously, one can consider loci in the Hilbert scheme of $n$-tuples of points in $\bar{A}_{2n-1}$ precisely one of which is required to lie over $U$ and the remainder over a disjoint open set $U' \subset \bC\backslash U$, which brings one into the setting studied in Section \ref{Sec:BimodulesWithoutQuilts} with product-like embeddings 
\[
(\bar{A}_{2n-1})|_U \times \Hilb^{[n]}((\bar{A}_{2n-1})|_{U'}) = 
\mathring{\bar{M}} \times \mathring{\bar{N}} = \mathring{\bar{W}}
\]
in an obvious notation. 

%A deformation of an open domain $U\subset \bC$ for which no critical point of the fibration $A_{2n-1}\to\bC$ crosses $\partial U$ induces a deformation of pairs $(\bar{M}, \bar{M}|_U)$ which, by a usual continuation or cobordism argument, induces equivalences of the associated Fukaya categories and equivariant structures; this is one justification in the sequel for not recording the particular open set $U$ in the notation for $\mathring{M}$ and $\scrF(\mathring{M})$, when the equivalence class under such deformations of $U$ is clear from context. 

We now specialise the general discussion to the particular setting of interest. Let $U^i \subset A_{2n+1}$ be the $z$-preimage a 2-disk $z(U^i) \subset \bC$ encircling the $(i,i+1)$-critical values, so $U^i$ retracts to a Lagrangian 2-sphere $L_i \subset A_{2n+1}$. See Figure \ref{Fig:Subdisk}.  Let $(U^i)' \subset A_{2n+1}$ denote the corresponding $z$-preimage for the larger open disk of Figure \ref{Fig:Subdisk}, which encircles all the other critical values.  
For each $1 \leq i\leq 2n$, let $\scrY^i_n \subset \scrY_{n+1} = M_{n+1}$ denote the open set $\mathring{M}_{n+1}$ associated to the open disc $z((U^i)') \subset \bC$, i.e. the open subset with quasiconvex boundary defined by the subschemes whose support lies in $(U^i)' \subset A_{2n+1}$. 
%correspond to the intersection with $\scrY_n$ of the Hilbert scheme of $n$ points lying in a subdomain $A^i_{2n+1}$ of the $A_{2n+1}$-surface defined by drawing a disk which encircles all but the $(i,i+1)$-pair of critical values of $A_{2n+1}$. 
There is an obvious inclusion
\begin{equation} \label{Eqn:product}
U^i \times \scrY^i_n \hookrightarrow \scrY_{n+1}.
\end{equation}
%where $U^i \subset A_{2n+1}$ is the $z$-preimage of a 2-disk encircling the $(i,i+1)$-critical values, so $U^i$ is a neighbourhood of a Lagrangian 2-sphere $L_i \subset A_{2n+1}$. See Figure \ref{Fig:Subdisk}.

\begin{center}
\begin{figure}[ht]
\includegraphics[scale=0.3]{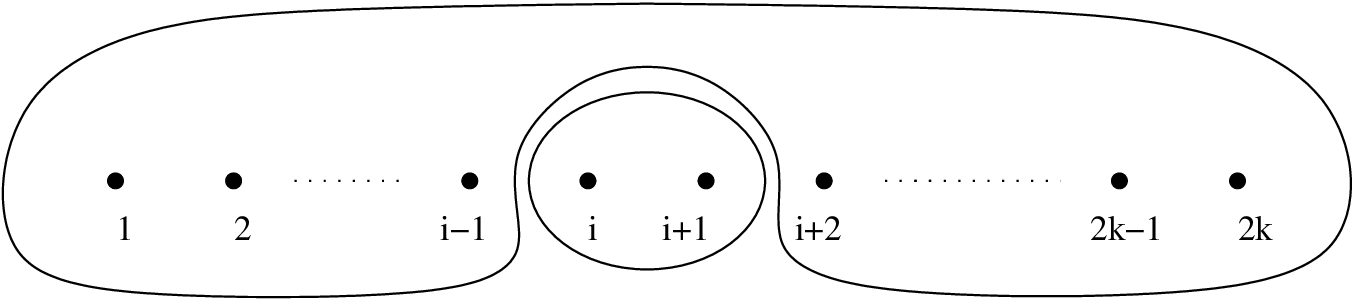}
\caption{Subsets of the base of the $A_{2k-1}$-surface.\label{Fig:Subdisk}}
\end{figure}
\end{center}

\begin{Lemma} \label{lem:Deform}
There is a distinguished equivalence $\scrF(\scrY^i_n) \to \scrF(\scrY_n)$. 
\end{Lemma}

\begin{proof}
A deformation of an open domain $U_t\subset \bC$ for which no critical point of the fibration $A_{2n-1}\to\bC$ crosses $\partial U_t$ at any time $t$ induces a deformation of pairs $(\bar{M}, \bar{M}|_{U_t})$, i.e. one has a family of codimension zero subdomains $\mathring{\bar{M}}(U_t) :=\bar{M}|_{U_t} \subset \bar{M}$ which all have pseudoconvex boundary.  Counting continuation solutions inside $\bar{M}$ yields equivalences between the categories $\scrF(\mathring{\bar{M}}(U_t))$ for different parameter values $t$.  

In the case at hand, we move the $(i,i+1)$-critical point pair towards infinity and then deform the subdisk defining $(U^i)' \subset A_{2n+1}$ to a standard disk in the base $\bC_z$ over which $\bar{A}_{2n-1}$    fibres.  This deformation of quasiconvex subdomains of (compactified) Milnor fibres induces a deformation of the corresponding quasiconvex subdomains $\mathring{\bar{M}}$ of Hilbert schemes. These yield the required quasi-isomorphisms $\scrF(\scrY_n) \simeq \scrF(\scrY^i_n)$.
\end{proof}

%; since the deformation is through surfaces fibring over $\bC$, it preserves the relative Hilbert schemes,  and hence their complements, which are our models for the nilpotent slices.  

%The deformations are through subdomains inside manifolds satisfying Hypotheses \ref{Hyp:Main} and \ref{Hyp:GWvanishes}, in particular the manifolds admit holomorphic projection maps to a continuously varying family of Stein domains. 
%There are induced quasi-isomorphisms
%\begin{equation}
%\scrF(\scrY_n) \simeq \scrF(\scrY^i_n).
%\end{equation}

The inclusion $\scrY_n^i \times U^i \subset \scrY_{n+1}$ is not an equality, and $\bar{\bar{\scrY}}_{n+1}$ does not co-incide with the product compactification $\bar{\bar{\scrY}}^i_n \times \bar{\bar{U^i}}$. However, there is an embedding of partial compactifications (complements of the divisors of subschemes supported on the fibres over $\infty \in \bP^1$)
\begin{equation}
\overline{\scrY}^i_n \times \bar{U}^i \rightarrow \overline{\scrY}_{n+1}
\end{equation}
coming from the subset of the Hilbert scheme of $\bar{A}_{2n+1}$ of subschemes with a length one subscheme supported over $U^i$ and a length $n$ subscheme over $(U^i)'$. This embedding of partial compactifications  is compatible with the divisors $D_0, D_r$ at infinity,  which is exactly the situation considered in Section \ref{Sec:BimodulesWithoutQuilts} in the general discussion of product embeddings $\mathring{\bar{M}} \times \mathring{\bar{N}} = \mathring{\bar{W}}$.  

It follows that, via the inclusion $\scrY_{n+1} \supset \scrY_n^i \times U^i$, and the identification $\scrY_n^i \cong \scrY_n$ of Lemma \ref{lem:Deform}, the Lagrangian sphere $L^i \subset U^i$ defines an elementary $(\scrF(\scrY_{n+1}), \scrF(\scrY_{n}))$-bimodule.  By Lemma \ref{lem:representable_bimod_functor}, there is a corresponding functor 
\[
\cup_i: \scrF(\scrY_n) \longrightarrow \scrF(\scrY_{n+1})
\] 
which we call $\cup_i$.

%The  adjoints of these bimodules are called \emph{cap bimodules} $\cap_i$.  Later, we will prove that the left and right adjoints co-incide, justifying omitting the choice from the notation.

\subsection{Cup functors are formal} \label{sec:cup-bimodules-are}

The theory of Section \ref{sec:equiv-elem} defines an equivariant structure on $\cup_i$. It remains to prove  that it is in fact pure, and hence formal.  We separate out cases depending on the parity of the index $i$ at which we include the new component of the matching when applying $\cup_i$.

\begin{Lemma} \label{lem:OddCupFormal}
If $i$ is odd, the cup functor $\cup_i$ is formal.
\end{Lemma}
\begin{proof}
 Let $C\scrH_{n+1, i}^{symp}$ denote the subcategory of $C\scrH_{n+1}^{symp}  $ corresponding to crossingless matchings containing the $i$\th cup.  By construction, $\cup_i$ factors through the inclusion $C\scrH_{n+1, i}^{symp} \subset  C\scrH_{n+1}^{symp}$. By Corollary \ref{cor:Kunneth-formal}, the functor $\cup_i \co  C\scrH_{n}^{symp} \to C\scrH_{n+1, i}^{symp} $ is formal with respect to the product equivariant structure on $C\scrH_{n+1, i}^{symp}  $. Since the quasi-isomorphism from the cochains to cohomology is determined by the equivariant structure, it remains to show that we can choose the equivariant structure on the the vanishing cycle $S^2$ factor so that the product equivariant structure agrees with the restriction of the equivariant structure on $C\scrH_{n+1}^{symp}   $ fixed in Theorem \ref{Thm:Arc_is_Pure}.

 To this end, it is useful to recall that the set of graded structures on a Lagrangian $L$, if non-empty, is parametrised by $H^0(L)$; similarly,  the set of equivariant structures on $L$, if non-empty, is parametrised by $H^0(L)$.  Elements of $H^0(L)$ act respectively by shifting the homological gradings and the weights. The  hypothesis that $i$ is odd implies that the plait matching is an object of $C\scrH_{n+1, i}^{symp}$, and indeed that $L^{n+1}_{\wpb} = \cup_i L^{n}_{\wpb}$.  More precisely, $\cup_i L^{n}_{\wpb}$ is Lagrangian isotopic to a Lagrangian lying in the open set $\scrY_n^i \times U_i \subset \scrY_{n+1}$, and under an appropriate identification $\scrY_n^i \cong \scrY_n$, has the form $L^n_{\wpb} \times L^i$.  We can therefore choose the graded and equivariant structures on the vanishing cycle $L^{i}$ so that the product equivariant structure on $L^{n+1}_{\wpb} = \cup_i L^{n}_{\wpb} $ agrees with the one fixed in Theorem \ref{Thm:Arc_is_Pure}.

For each Lagrangian $L_{\wp} \in  C\scrH_{n}^{symp}$, the K\"unneth formula in Floer cohomology implies that, when equipped with product gradings, the Floer cohomology between $\cup_i L_{\wpb}$ and $\cup_i L_{\wp}$ is still symmetrically graded; hence the product grading on $ \cup_i L_{\wp} $ agrees with the one fixed for Lagrangians in $\scrY_{n+1}$. The product equivariant structure on $  \cup_i L_{\wp} $ has the property that weights and gradings agree (i.e. it is pure, again by K\"unneth); hence it also  agrees with the equivariant structure fixed in Theorem \ref{Thm:Arc_is_Pure}.

We conclude that the functor $\cup_i$ is pure with respect to the fixed equivariant structures from  Theorem \ref{Thm:Arc_is_Pure}, hence is formal.
\end{proof}

\begin{Lemma} \label{lem:EvenCupFormal}
If $i$ is even, the cup functor $\cup_i$ is formal.
\end{Lemma}
\begin{proof}
While the distinguished component $L_{\wpb}^{n+1} $ is not in the image of the functor $\cup_i$, the matchings $\wpb^{n+1}$ and $\cup_i \wpb^n$ differ by a single handle-slide.   It follows that the corresponding Lagrangian submanifolds $L_{\wpb^{n+1}}$ and $L_{\cup_i \, \wpb^n}$ intersect in codimension one in the sense of Section \ref{Sec:MeetInCodim2}.   By Lemma \ref{Lem:Constant_triangle}, we therefore know that the Floer triangle product for the triple
\begin{equation}\label{eq:triple}
(L_{\wpb}^{n+1}, \cup_i L_{\wpb}^n, \cup_i L_{\wp}^n)
\end{equation}
is non-trivial.   As with the proof of purity of the symplectic arc algebra in \cite[Proposition 6.10]{AbSm}, in the presence of a non-trivial Floer product, symmetry of the weights for two of the three possible pairs of components implies symmetry of the weights for the third pair.  
%\begin{equation}
%HF^*(\cup_i L_{\wpb}^n, L_{\wpb}^{n+1} ) \otimes HF^*(\cup_i L_{\wp}, \cup_i L_{\wpb}^n)  \longrightarrow HF^*(\cup_i L_{\wp}, L_{\wpb}^{n+1}).
%\end{equation} 
Indeed, by Proposition \ref{Prop:hittopclass} the top degree class lies in the image of the product 
\[
HF^*(L^{n+1}_{\wpb}, \cup_i L^n_{\wp}) \otimes HF^*(\cup_i L^n_{\wpb}, L^{n+1}_{\wpb}) \longrightarrow HF^*(\cup_i L^n_{\wpb}, \cup_i L^n_{\wp})
\]
which determines its weight given the weights on the domain groups; weights of all other classes in the target are fixed by the cyclicity of the module action over $H^*(\cup_i L^n_{\wp})$ say.     For the two components of the triple \eqref{eq:triple} in the image of $\cup_i$  we have symmetry of weights by construction, and for the first two components we have symmetry by the choice of equivariant structure on $L^i=S^2$.  Varying $\wp$, we see that the equivariant structure on the target category is obtained by restriction from $C\scrH_{n+1}^{symp} $ as required.
\end{proof}

Note that, having fixed a quasi-equivalence $C\scrH_n^{symp} \simeq \scrH_n^{symp}$, we may achieve formality for all the $\cup_i$-bimodules simultaneously.

%%%%%%%%%%%%%%%%%%
%%%%%%%%%%%%%%%%%%
%%%%%%%%%%%%%%%%%%
% \input{LES.tex}

%%%%%%%%%%%%%
%%%%%%%%%%%%%%
%%%%%%%%%%%%%
%%%%%%%%%%%%%%%

\section{Cohomology bases and conormal models}

In the combinatorial arc algebra $H_n$ from \cite{Khovanov:functor}, the module associated to a pair of matchings $\wp \cup \overline{\wp'}$ is $V^{\otimes c(\wp,\wp')}$, where $V = \bZ[x]/(x^2)$ has a distinguished basis $\{1,x\}$. We will pin down explicit bases for the corresponding Floer groups appearing in $\scrH_n^{symp}$ inductively in $n$,  by making systematic use of the constraints imposed by compatibility with module structures and cup-functors.

\noindent \emph{For this and the next section, all (Floer) cohomology groups are taken with $\bZ$ coefficients.}

\subsection{A summary of the argument}

To help the reader navigate the rest of this and the subsequent section, we give a brief overview of the structure of the argument which identifies the combinatorial and symplectic arc algebras.  The combinatorial arc algebra is built upon a TQFT in which part of the structure is a co-product $V \to V\otimes V$ which takes $1 \mapsto 1\otimes x + x \otimes 1$.  The fact that one only encounters terms of the shape $1\otimes x + x\otimes 1$, and not $1\otimes x - x \otimes 1$, is the key feature which distinguishes Khovanov homology from its sign-twisted sibling ``odd Khovanov" homology \cite{ROS:odd_khovanov}.  The symplectic arc algebra is built, as a graded vector space, out of Floer cohomology groups $HF^*(L_{\wp}, L_{\wp'})$ which are identified in clean intersection models with cohomology groups $H^*(L_{\wp} \cap L_{\wp'})$ (which are in turn certain tensor products of copies of $V$). The Floer products 
\begin{equation} \label{eqn:nice_triple}
H^*(L_{\wp'} \cap L_{\wp''}) \otimes H^*(L_{\wp} \cap L_{\wp'}) \to H^*(L_{\wp} \cap L_{\wp''})
\end{equation}
do not have any classical description for general triples $(\wp,\wp',\wp'')$, because there is no universal model for a neighbourhood of three cleanly intersecting Lagrangian submanifolds; but for certain triples we prove plumbing-type models are available, and infer the product \eqref{eqn:nice_triple}  can be described in terms of a convolution product.  

That convolution product depends on choices of orientation, which are themselves not canonical.  An important point is that if one picks complex orientations on all the $L_{\wp}$ and their iterated intersections, then the signs do not work out correctly (this observation goes back to \cite{SW}).

To simplify the situation, we prove that each group $H^*(L_{\wp} \cap L_{\wp'})$ is a cyclic module over $H^*(\scrY_n)$.  This allows us to focus attention on products of minimal degree generators, on the one hand, and the action of $H^*(\scrY_n)$  on the other.  We fix a convenient basis for $H^*(\scrY_n)$ (and hence each $H^*(L_{\wp})$ by appealing to a certain ``topological model" for the compact core, discussed in the next subsection.  Given that, the rest of the proof has three essential ingredients:
\begin{enumerate}
\item We show that all Floer products are determined as compositions of products which can be understood as convolutions. This involves an essential breaking of symmetry between the components of the core, and indeed gives a distinguished role to one particular matching $\wp_{plait}$, which has the feature that any $\wp$ is determined by $L_{\wp} \cap L_{\wp_{plait}}$.
\item We build by hand a basis for $\scrH_2^{symp}$, and then extend it inductively to all higher $\scrH_n^{symp}$ by insisting that the cup-functors $\cup_i$ are then compatible with bases -- in the sense that they take basis elements to positive linear combinations of basis elements -- for certain distinguished pairs, involving $\wp_{plait}$.
\item We prove that the bases constructed in fact behave well for all possible Floer products.  This is essentially a consistency check, which involves proving that there are sufficiently many non-vanishing Floer products which are ``well-understood".
\end{enumerate}
The details of the argument are quite intricate and use numerous features of our particular situation, which makes it hard to axiomative effectively, but we hope this will help keep the general picture in mind as we proceed.

\subsection{Spaces of flags and the compact core}\label{Sec:Flags}

\newcommand{\open}{\mathrm{open}}

Let $\Slice = \Slice_n \subset \mathfrak{gl}_{2n}(\bC)$ be the affine subspace consisting of matrices of the form
\begin{equation} \label{eq:y-matrix}
A = \begin{pmatrix}
A_{1} & I &&& \\
A_{2} && I && \\
\dots &&& \dots & \\
A_{n-1} &&&& I \\
A_{n} &&&& 0
\end{pmatrix}
\end{equation}
with  $A_k \in \mathfrak{gl}_2(\bC)$, and where $I \in \mathfrak{gl}_2(\bC)$ is the identity matrix. (In \cite{SS} we considered the codimension one subspace lying in $\frak{sl}_{2n}(\bC)$ and the adjoint quotient on configurations of total mass zero, but the results carry over \emph{mutatis mutandis} without the trace zero condition.)  The symmetric product $\Sym^{2n}(\bC)$ is identified, via symmetric polynomials, with $\bC^{2n}$.  Grothendieck \cite{Slodowy} described a simultaneous resolution of the adjoint quotient map (which takes a matrix to its collection of eigenvalues)
\begin{equation} \label{Eqn:AdjointQuotient}
\chi: \Slice \rightarrow \mathrm{Sym}^{2n}(\bC)
%; \qquad A \mapsto \det(x^n I - x^{n-1} A_1 - x^{n-2} A_2 - \cdots - A_n)
\end{equation}
via the space of pairs $(A,\bf{F})$, where $\bf{F}$ is a flag and $A\in \Slice$ preserves the flag. The simultaneous resolution maps to $\bC^{2n}$, the space of ordered configurations, since the flag orders the eigenspaces. The resolution of $\chi^{-1}(0)$ contains a ``compact core" which is the space of flags fixed by the distinguished nilpotent matrix where $A_j = 0 \ \forall j$.  We now summarise results of Cautis and Kamnitzer \cite{CK, K, CK2}, which give a fibrewise compactification of $\chi$.

Fix a vector space $\bC^2$ with basis $\{e_1,e_2\}$ and consider $\bC[z]$-submodules $F_i$ of $\bC^2\otimes \bC(z)$ which contain $F_0 = \bC^2\otimes \bC[z]$.  More precisely, for a tuple $\mathbf{w}=(w_1,\ldots, w_{2n}) \in \bC^{2n}$  we consider the space of flags of modules:
\[
Y_n^{\mathbf{w}} = \left\{F_0 \subset F_1 \subset \cdots \subset F_{2n} \subset \bC^2\otimes \bC(z), \ \mathrm{rk} (F_i/F_{i-1}) = 1, (z-w_i I) F_i \subset F_{i-1} \right\}.
\]
As we vary $\mathbf{w}$, these spaces fit into a family $\mathbf{Y}_n \rightarrow \bC^{2n}$; we think of the base as the space of \emph{ordered} configurations of $2n$ points in $\bC$.   There is an open subset
\[
Y_{n}^{\mathbf{w}, \open} = \left\{ (F_0,\ldots, F_n) \in Y_n^{\mathbf{w}}  \ \big| \  F_{2n} / F_0 = \langle z^{-1}e_1, z^{-1}e_2, \ldots, z^{-n}e_1, z^{-n}e_2\rangle \right\}
\]
We say $\mathbf{w}$ is generic if each $w_i \neq w_j$.  At the other extreme, when $\mathbf{w}=(0,\ldots,0)$, we have the resolution of the nilpotent cone in $\Slice$. The compact core 
\begin{equation} \label{Eqn:CompactCore}
Z \ = \ \left\{  (F_0,\ldots, F_{2n}) \ | \ z^{n}F_{2n} = F_0 \right\} \ \subset \ Y_{n}^{\mathbf{0}, \open}
\end{equation}
is the locus lying over the matrix given by $A_j = 0$ for each $j$ in \eqref{eq:y-matrix}. 

\begin{Proposition}\label{Prop:flags_etc} The family $\mathbf{Y}_n \rightarrow \bC^{2n}$ has the following properties.
\begin{enumerate}
\item Each fibre $Y_n^{\mathbf{w}}$ is an iterated $\bP^1$-bundle, diffeomorphic to $(S^2)^{2n}$.  If $\mathbf{w}$ is generic, then $Y_n^{\mathbf{w}} \cong (\bP^1)^{2n}$ is holomorphically a product.
\item $\mathbf{Y}_n$ is a fibrewise compactification of the simultaneous resolution \eqref{Eqn:AdjointQuotient}. In particular, if $\mathbf{w}$ is generic, then $Y_n^{\mathbf{w},\open} \cong \scrY_n^{\mathbf{w}}$ are holomorphically isomorphic.
\item The complement $Y_n^{\mathbf{w}} \backslash Y_{n}^{\mathbf{w}, \open}$ is an irreducible divisor, of class $(1,\ldots, 1) \in H^2(\bP^1)^{2n}$ if $\mathbf{w}$ is generic.
\end{enumerate}
\end{Proposition}

\begin{proof} The first two statements are directly from \cite[Section 2.2]{CK} (see also \cite{K,CK2}). For the third, since irreducibility is an open condition and $\chi$ is equivariant for a $\bC^*$-action rescaling all eigenvalues, it suffices to work on the zero-fibre, i.e. with $\mathbf{w}=0$. In that case, $Y_n^{\mathbf{0}} \backslash Y_{n}^{\mathbf{0}, \open}$ is the locus where the natural map $z^n F_{2n} \rightarrow F_0$ is not an isomorphism, which is a divisor of determinantal type defined by vanishing of a section of the tensor product of the bundles $\mathcal{L}_i$ with fibres $F_i / F_{i-1}$.  These deform to $\mathcal{O}(1)$-bundles on the factors of $(\bP^1)^{2n}$, which gives the statement on the degree.  Finally, the divisor maps with degree one to the space of flags of length $2n-1$, which implies its irreducibility.  (An Euler characteristic computation shows that, although irreducible, the divisor is singular, but we will not require that fact.)
\end{proof}

 %\begin{Remark}
%Compactifying $\scrY_n$ to $(\bP^1)^{2n}$ yields a spectral sequence $Kh_{symp}(\kappa) \Rightarrow HF(\kappa)$ where the $\bZ/4$-graded target Floer group is given by taking Floer cohomology of $L_{\wpcirc}$ and $\beta_{\kappa}(L_{\wpcirc})$ inside $(\bP^1)^{2n}$.  This is a geometric analogue of the Lee spectral sequence \cite{Lee}. \end{Remark}
 
 For an upper-half-plane crossingless matching $\wp$, let $\hat{L}_{\wp}$ denote the Lagrangian multi-antidiagonal  of $(\bP^1)^{2n}$ in which points paired by $\wp$ should take antipodal values, with the factor labelled by the odd co-ordinates positively oriented. For instance, when $n=2$ and we embed in $(\bP^1)^4$, the two core components are 
 \[
 \hat{L}_{\wpb} = \{(z, -1/\bar{z}, w, -1/\bar{w})\} \ \textrm{and} \ \hat{L}_{\mix} = \{ (z, -1/\bar{w}, w, -1/\bar{z})\}.
 \]

 \begin{Lemma}  \label{Lem:Wehrli}
 The compact core $Z$ is homeomorphic to $\cup_{\wp} \hat{L}_{\wp}$ by a map which is a diffeomorphism on each component. In particular, $Z$ is a union of copies of $(S^2)^n$, indexed by crossingless matchings, meeting pairwise cleanly, and the ``small antidiagonal" 
 \begin{equation} \label{Eqn:SmallAntidiagonal}
 \left\{ (z, -1/\bar{z}, z, -1/\bar{z}, \ldots, z, -1/\bar{z}) \right\} \subset (\bP^1)^{2n}
 \end{equation}
 lies in the common intersection locus of all of the $\hat{L}_{\wp}$.   \qed
 \end{Lemma}
 
 \begin{proof} See \cite[Appendix A]{Russell-Tymoczko} and \cite{Wehrli} for two different proofs.
 \end{proof}
 
  Since the divisor at infinity in $\scrY_n^{\mathbf{w}}$ is ample when $\mathbf{w}$ is generic, one can define the Fukaya category of $\scrY_n$ with respect to a finite volume exact K\"ahler form which extends smoothly to the compactification, arguing via positivity of intersection at infinity to ensure compactness of moduli spaces of pseudoholomorphic discs.   The map $\mathbf{Y}_n \rightarrow \bC^{2n}$ is a differentiably trivial fibration with fibre $(S^2)^{2n}$, which is  symplectically trivial for a K\"ahler form in the class of the divisor at infinity by fragility of the two-dimensional Dehn twist \cite{Seidel:4dtwist} (of course the monodromy is not fragile relative to the divisor at infinity).  
  
  \begin{Lemma} \label{Lem:IsotopicCoreVCycle}
  The iterated vanishing cycle $L_{\wp}$ is smoothly isotopic to $\hat{L}_{\wp}$.
  \end{Lemma}
  
  \begin{proof} This follows from the description of the Morse-Bott degeneration of the spaces $Y_n^{\mathbf{w},\open} \subset Y_n^{\mathbf{w}}$ as a pair of eigenvalues coalesce, given in \cite[Section 3]{K}, which in turns relies on the factorisation property of the affine Grassmannian due to Beilinson and Drinfeld \cite[Section 5.3.10]{BD}.  In short, if we restrict the fibrewise compactification of $\Slice$ to a disc $D_{\epsilon}$ parametrising $\mathbf{w} = (\mu+\epsilon, \mu-\epsilon, w_3,\ldots,w_{2n})$, Kamnitzer shows the singular fibre at $\epsilon = 0$ is \emph{globally} a product, with one factor the compactification of the fibre of the smaller slice $\Slice_{n-1}$ over $\mathbf{w}' = (w_3,\ldots, w_{2n})$, and the other a singular quadric cone, i.e. a Hirzebruch surface $\mathbf{F}_2$ with the $(-2)$-sphere collapsed to a point, which arises as the singular fibre of the compactified slice $\Slice_1$. (The global product structure is not compatible with the divisors at infinity, so is not inherited by the resolution of the open slice.) On the compactification, this product structure shows the co-isotropic vanishing cycle is isotopic to the anti-diagonal in $(S^2\times S^2)$ stabilised by the submanifold $(S^2)^{2n-2}$ corresponding to $\mathbf{w}'$.  Since the vanishing cycles are iterated convolutions of the correspondences, the result follows. \end{proof}
  
 Lemma \ref{Lem:IsotopicCoreVCycle} constructs an isotopy in $Y_n^{\mathbf{w}}$ and not necessarily in $Y_{n}^{\mathbf{w}, \open}$, but that is sufficient to control the cohomological restriction maps $H^*(\scrY_n) \rightarrow H^*(L_{\wp})$.
  
Fix as usual the eigenvalues $\bf{l} = \{1,\ldots, 2n\} \subset \bC$ and consider the fibre $\scrY_n = \scrY_n^{\bf{l}}$ of the adjoint quotient.   According to \cite[Theorem 2]{Khovanov:crossingless}, the cohomology $H^*(\scrY_n;\bZ)$ is generated as a ring by $H^2(\scrY_n;\bZ)$, and is torsion-free, so we can identify $H^2(\scrY_n)$ and $H_2(\scrY_n)^{\vee}$ integrally.  By \cite[Proposition 5]{Khovanov:crossingless},  the map $H^2((\bP^1)^{2n}) \to H^2(\scrY_n)$ is surjective, and indeed $ H^2(\scrY_n) $ is the quotient of $H^2(\bP^1)^{2n}$ by the $(1, \ldots, 1)$ class; this can also be inferred from the Lefschetz hyperplane theorem and Proposition \ref{Prop:flags_etc}.
Via complex orientations, there is a distinguished basis $e_i$ for $H^2(\bP^1)^{2n} = \bZ^{2n}$, where $e_i$ has $1$ in the $i$-th place and 0's elsewhere. The elements $(-1)^{i+1} e_i$, with $1\leq i\leq 2n-1$, therefore define a basis of $H^2(\scrY_n;\bZ)$.  By restricting these elements to $L_{\wp}$, we obtain bases for $H^2(L_{\wp})$ which make the following true tautologically:
\begin{Corollary} \label{Cor:Same-1}
The elements $v_i=(-1)^{i+1} e_i$, with $1\leq i\leq 2n-1$, define a basis of $H^2(\scrY_n;\bZ)$ with the following property: under the restriction map
\[
H^2(\scrY_n;\bZ) \longrightarrow H^2(L_{\wp};\bZ)
\]
each basis element maps either to zero or to a basis element.
\end{Corollary}

The force of the construction  is that we restrict to zero or a basis element, and not to the negative of a basis element. 
We immediately obtain monomial bases for the entire cohomology $H^*(L_{\wp})$, which is generated in degree two.  Henceforth, we will abuse notation and denote by $v_j$ a basis element of either of $H^2(\scrY_n)$ of $H^2(L_{\wp})$, depending on the context.  

The choice of basis of $H^2(\scrY_n)$ also orients the 2-spheres  $V_i$ in view of the proof of Lemma \ref{Lem:IsotopicCoreVCycle}, which identifies them cohomologically with antidiagonals in factors of $(S^2)^{2n}$ (the $i$\th and $i+1$\st basis elements map to the same generator of the top cohomology of $V_i$). In particular, $H^2(L_{ \cup_i \wp})$ inherits a basis from the K\"unneth theorem, i.e. the decomposition
\begin{equation}
H^2(L_{ \cup_i \wp};\bZ) \longrightarrow H^0(L_{\wp}) \otimes H^2(V_i) \oplus H^2(L_{\wp}) \otimes H^0(V_i).
\end{equation}
  %$\cup_i \co \scrY_{n-1} \times V_i \hookrightarrow \scrY_{n}$, denote the embedding associated to the vanishing cycle $V_i$. 

\begin{Corollary} \label{Cor:Same}
The basis of $H^2(L_{\wp};\bZ)$ is preserved by the cup-functors, in the sense that the choice of basis induced by restriction from $H^2(\scrY_n;\bZ) $ agrees with the basis induced by applying the   K\"unneth theorem to a presentation as the image of a cup functor.
\end{Corollary}
\begin{proof}
This follows since the Lagrangians $L_{\wp}$ are iterated convolutions of the correspondences, by Lemma \ref{Lem:IsotopicCoreVCycle}.  The orientations on the $V_i$ define bases for $H_2(L_{\wp})$, and hence bases for $H^2(L_{\wp})$, which by construction agree with those coming from $H^2(\scrY_n)$.
\end{proof}
%The corollary follows immediately from the previous results, after identifying the cohomological restrictions to $L_{\wp}$ with those to $\hat{L}_{\wp}$. 

\subsection{Another orientation convention} \label{Sec:Basis}

In \cite{SS}, and for the discussion of purity in Section \ref{Sec:Purity}, we adopted a grading convention in which the symplectic arc algebra was ``symmetrically" graded, so for distinct matchings $\wp, \wp'$ meeting in codimension $k$ the group $HF^*(L_{\wp}, L_{\wp'})$ was supported in degrees $k \leq \ast \leq 2n-k$.  For the purpose of computing products in the symplectic arc algebra, a different convention is more convenient.

\begin{Definition} \label{Def:Depth} Let $\wp$ be a crossingless matching in the upper half-plane $\frak{h}$. The \emph{depth} of an arc $\gamma \subset \wp$ is the number of other arcs $\gamma' \subset \wp\backslash \gamma$ which meet a vertical line from $\gamma$ to $\infty \in \frak{h}$.
\end{Definition}

Informally, depth is the number of arcs under which $\gamma$ is nested in $\wp$.  We orient the arcs $\gamma$ by the convention that arcs of even depth are oriented clockwise, and arcs of odd depth are oriented anticlockwise: see Figure \ref{Fig:Depth} for two examples. Equivalently, if the critical points are numbered $\{1,2,\ldots, 2n\}$ then all arcs are oriented \emph{towards} their even end-point.
\begin{center}
\begin{figure}[ht]
\includegraphics[scale=0.4]{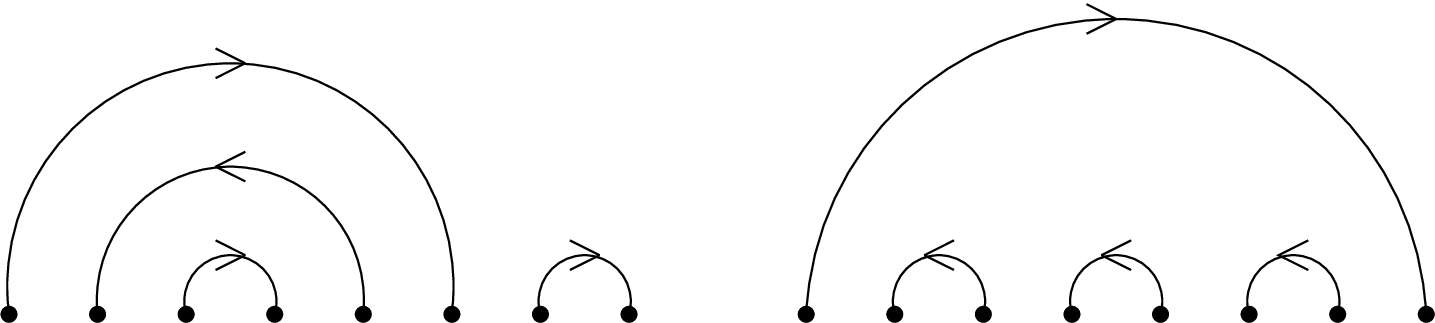}
\caption{Cohomology bases for components of the compact core\label{Fig:Depth}}
\end{figure}
\end{center}

In the Lefschetz fibration $z: A_{2n-1} \rightarrow \bC$, the monodromy of the generic fibre $T^*S^1$ around a critical point is a Dehn twist in the $S^1$, which preserves the orientation on $S^1$.  Fix once and for all the standard (anticlockwise) orientation on $S^1 \subset T^*S^1 = \bC^* \subset \bC$.  Then a choice of orientation for a matching path $\gamma$ determines a well-defined orientation on the Lagrangian sphere $L_{\gamma} \subset A_{2n-1}$, which means that Definition \ref{Def:Depth} fixes bases for $H_2(L_{\wp})$ for all the Lagrangians $L_{\wp}$.  

\begin{Lemma} \label{Lem:Orient_even}
The orientation convention of Definition \ref{Def:Depth} yields the same bases for $H_2(L_{\wp})$ as the bases of Corollary \ref{Cor:Same}. Furthermore, with these orientations the symplectic arc algebra is graded in even degrees.
%\begin{itemize}
%\item ; 
%\item there are gradings compatible with these orientations so that for any $\{\wp,\wp'\}$ the minimal degree subspaces $k_{min}(\wp,\wp')$ and $k_{min}(\wp',\wp)$ are in degrees $\{0,2n\}$.
%\end{itemize}
\end{Lemma}

\begin{proof}
An isolated intersection point of $L_{\wp}$ and $L_{\bar{\wp'}}$ is a tuple of intersections of matching paths in the $A_{2n-1}$-surface, and the sign of the intersection is given by the product of the corresponding signs on the surface.
At any critical point of the fibration $A_{2n-1}\rightarrow \bC$, the two matchings paths of $\wp$ and $\bar{\wp'}$ meeting at that critical point are coherently oriented, in the sense that the local isotopy of Lefschetz thimbles given by rotating one path to the other through thimbles is an orientation-preserving isotopy.  Therefore the local intersection number of the two thimbles in $A_{2n-1}$  is positive. (It may be helpful to compare to the case of two copies of the same matching sphere in $A_1$, which defines $S^2 \subset T^*S^2$, and to recall that Maslov indices agree with Morse indices in a cotangent bundle $T^*Q$ if $\langle \partial_{q_1},\ldots, \partial_{q_n},\partial_{p_1},\ldots,\partial_{p_n}\rangle$ is an oriented basis; this differs from the symplectic orientation of $T^*Q$  by a global sign $(-1)^{n(n+1)/2}$.) We conclude that the symplectic arc algebra is graded in even degrees. 

%Fix now the real matching paths $[i,i+1] \subset \bC$ for the $A_{2n-1}$-surface, $1 \leq i \leq 2n-1$.  The convention of Definition \ref{Def:Depth} is compatible with the orientation of the vanishing cycle factors $V_i$ relevant to Corollary \ref{Cor:Same} because of the choice of basis $(-1)^{i+1}e_i$ for $H^2(\bP^1)^{2n}$. 
%Fix a grading on the first such sphere arbitrarily, and then shift the remaining gradings consecutively so that the isolated intersection points $\{2,3,\ldots,2n-1\}$  have gradings $0$ from $L_{[i-1,i]}$ to $L_{[i,i+1]}$ if $i$ is odd and $2$ if $i$ is even (so reading left to right the degrees alternate; the degrees in the reverse directions are then fixed by Poincar\'e duality). For any pair of matchings $\wp,\wp'$, if $\wp$ is placed in the upper half-plane and $\wp'$ in the lower, then the collection of odd critical points defines a generator of $CF^*(L_{\wp},L_{\wp'})$ which lies in degree $0$.
\end{proof}

%In the sequel, we shall only make use of the following (easy) consequence of Lemma \ref{Lem:Orient_even}: there are orientations on the $L_{\wp}$ for which $\scrH_n^{symp}$ is graded in even degrees; and there are gradings for any pair $L_{\wp}, L_{\wp'}$ with respect to which the minimal degree generators live in degrees $\{0,2n\}$. These weaker conclusions can obviously be achieved by introducing appropriate shifts of any given choice of graded structures.

%\begin{Remark} The bases of $H_2(L_{\wp})$  obtained above arise from the following orientation convention for arcs in the $A_{2n-1}$-surface with critical values at $\{1,2,\ldots, 2n\}$: orient each matching path towards its even end-point. (We will not need this interpretation.)   \end{Remark}

\subsection{$A_2$-fibrations and the plumbing model for $n=2$}

 Recall that two submanifolds $Y_0, Y_1$  meet cleanly if they intersect along a submanifold, and
  \begin{equation}
  T_p (Y_0 \cap Y_1) = T_pY_0 \cap T_pY_1\end{equation}
  for every $p\in Y_0 \cap Y_1$.     The Lagrangians $\hat{L}_{\wp}$ meet pairwise cleanly. It is not known if $\hat{L}_{\wp}$ and $L_{\wp}$ are Lagrangian isotopic (in $\scrY_n$ rather than its compactification) in general, but \emph{pairs} of crossingless matching Lagrangians can be Hamiltonian isotoped into a clean intersection model. This follows from the description in \cite{SS} of the degeneration of the fibre $\scrY_n$ when three eigenvalues coalesce.

Take a tuple $\mu = (\mu_1, \dots,\mu_{2n})$ with $\mu_1 = \mu_2 = \mu_3$, and which are otherwise pairwise
distinct. The adjoint fibre $\scrY_n^{\mu}$ contains two orbits:
the regular orbit $\OO^{\reg}$ has an indecomposable Jordan block of
size 3 for the eigenvalue $\mu_1$, whilst the subregular orbit $\OO^{\sub}$
has two Jordan blocks of sizes $1,2$ (the minimal orbit  in the adjoint fibre $\chi^{-1}(\mu)$ consists
of matrices with three independent $\mu_1$-eigenvectors, but this orbit is disjoint from $\Slice$). The singular set of the fibre $\scrY_n^{\mu}$ is smooth, and can be canonically identified with  $\scrY_{n-1}^{\mu_{red}}$ where $\mu_{red} = (\mu_1, \mu_4, \ldots, \mu_{2n})$, see \cite[Lemma 25]{SS}.  

At a point $y \in \OO^{\sub} \cap \Slice$ let $E_y$ be the
$\mu_1$-eigenspace of its semisimple part $y_s$. These spaces form the fibres of a line bundle $\scrF \rightarrow \OO^{\sub} \cap \Slice$. We consider the associated vector bundle 
\begin{equation} \label{eq:associated-c4}
(\FF \setminus 0) \times_{\bC^*} \bC^4 = \bC \oplus \FF^{-2} \oplus
\FF^2 \oplus \bC.
\end{equation}
We also introduce the map (the versal deformation of the $A_2$-singularity)
\begin{equation} \label{eq:a2-map}
p: \bC^4 \longrightarrow \bC^2, \quad p(a,b,c,d) =
 (d,a^3-ad+bc).
\end{equation}
The relevant statement from \cite{SS} is then:

\begin{Lemma} \label{lem:local_A2_degeneration}
Let $P \hookrightarrow \Conf_{2n}(\bC)$ be a small
bidisc parametrized by $(d,z)$, corresponding to the set of
eigenvalues
\[
(\mu_1+\{\text{all solutions of $\lambda^3 - d \lambda  + z = 0
$}\},\mu_4,\dots,\mu_{2n}).
\]
There is a neighbourhood of $\OO^{\sub} \cap \Slice$ inside
$\chi^{-1}(P) \cap \Slice$, and an isomorphism of that with a
neighbourhood of the zero-section inside $(\FF \setminus 0)
\times_{\bC^*} \bC^4$, which fits into a diagram
\[
\begin{CD}
 \chi^{-1}(P) \cap \Slice
 @>{\text{local $\iso$ defined near $\OO^{\sub} \cap
 \Slice$}}>> {(\FF \setminus 0) \times_{\bC^*} \bC^4} \\
 @V{\chi}VV @V{p}VV \\
 P @>{\quad\qquad\quad(d,z)\quad\qquad\quad}>> \bC^2
\end{CD}
\]
where $p$ is given by \eqref{eq:a2-map} on each $\bC^4$ fibre.
\end{Lemma}

In particular, there is an open subset of a generic fibre $\scrY_n$, for a tuple of eigenvalues sufficiently close to $\bf{\mu}$, which is an $A_2$-fibration over $\scrY_{n-1}$.   In the lowest non-trivial case, an open subset of $\scrY_2$ is a (non-trivial) $A_2$-fibration over $T^*S^2$. For any arc $\gamma$ in the $A_2$-space, there is an associated Lagrangian submanifold of $\scrY_2$ given by taking the matching sphere $L_\gamma \subset A_2$ fibrewise over the zero-section of $T^*S^2$.  We will refer to such a Lagrangian as lying in \emph{fibred position}.

\begin{Lemma}[Section 4.3 of \cite{SS}] \label{Lem:CleanFor-n=2}
The Lagrangians $L_{\wpb}$ and $L_{\mix}$ in $\scrY_2$ can be Hamiltonian isotoped to lie in fibred position, given by arcs in $A_2$ which meet transversely once. \qed
\end{Lemma}

    \begin{Lemma} \label{lem:ConormalModelForPairs}
  Let $L, L' \subset X$ be cleanly intersecting Lagrangian submanifolds of a symplectic manifold $X$. There is an open neighbourhood $U = U(L\cap L') \subset X$ of $L\cap L'$ in $X$, and a symplectic embedding
  $U \hookrightarrow T^*L$ taking $U \cap L$ to the zero-section and $U \cap L'$ to the conormal bundle of $L\cap L'$. 
    \end{Lemma}

\begin{proof}
See e.g.  \cite[Proposition 3.4.1]{Pozniak}.  Briefly, Weinstein's theorem gives a symplectomorphism $\psi: U \rightarrow T^*L$ taking $U \cap L'$ into the conormal bundle $\nu^*_{L'}$.  The image  $\psi(L \cap U)$ is tangent to the zero-section along $L\cap L'$, hence can locally be written as the graph of a 1-form $\phi$ which vanishes along $L\cap L'$, so which is exact $\phi=dk$ in a small neighbourhood of the intersection locus.  Composing the given symplectic embedding $\psi$ with the time-one Hamiltonian flow of $-k$, which preserves $\nu^*_{L'}$ and takes $\psi(L\cap U)$ into the zero-section, yields the desired map $U\hookrightarrow T^*L$. 
\end{proof}

According to \cite{Abouzaid:plumbing},  
 there is a ``plumbing model" for the Fukaya category of a pair of exact graded Lagrangians which meet pairwise cleanly.   We only require the somewhat simpler cohomological result, which goes back to Pozniak and follows fairly straightforwardly from Lemma \ref{lem:ConormalModelForPairs}; see for instance \cite{KOh}.  It may help to recall that given a diagram of cleanly intersecting submanifolds
\[
\xymatrix@!=0.6pc{
 &M\ar@{<-_)}[dl]_{i_1}\ar@{<-^)}[dr]^{i_2} &\\
Q_1\ar@{<-^)}[dr]_{j_1}&&Q_2\ar@{<-_)}[dl]^{j_2}\\ &Q_1 \cap Q_2&
}
\]
%then on cohomology 
%\begin{equation} \label{eq:convolution}
%i_{1}^*(i_2)_*\phi =(j_1)_*(e(E)\cup j_{2}^*\phi)
%\end{equation}
%where $E=j_1^*i_1^*TM/(j_1^*TQ_1+j_2^*TQ_2)$ can be identified by projection with the normal bundle of $Q_1\cap Q_2$ in $Q_i$.  
The convolution product 
\[
H^*(Q_1\cap Q_2)^{\otimes 2} \longrightarrow H^*(Q_i) 
\]
is by definition given by cup-product composed with the transfer (obtained from Poincar\'e duality on $Q_1\cap Q_2$, the push-forward on homology, and Poincar\'e duality on $Q_i$).

 \begin{Proposition} \label{Prop:plumb}
 Let $Q_1$ and $Q_2$ be exact Lagrangian submanifolds of $(M,\omega)$ which meet cleanly along $C = Q_1 \cap Q_2$ of codimension $d$. Grade the $Q_i$ so that the minimal degree generator of $HF^*(Q_1,Q_2)$ lies in degree $0$.  The Donaldson category with objects $Q_1, Q_2$ is  equivalent to the (ordinary) category 
\[
\xymatrixcolsep{5pc}\xymatrix{
Q_1\ar@(ul,dl)[]_ {H^*(Q_1)}  \ar@/^/[r]^{H^*(C)} & Q_2 \ar@(ur,dr)[]^{H^{*}(Q_2)} \ar@/^/[l]^{H^{*-d}(C)}
}
\]
and with the non-trivial compositions given by convolution.
\end{Proposition}

Note that the grading convention breaks symmetry: it amounts to choosing an ordering of $\{Q_1,Q_2\}$.
%, which in turn fixes a distinguished co-orientation for $C \subset Q_1$ and hence a sign for the Euler class of the normal bundle, which is relevant when pushing forward in view of \eqref{eq:convolution}.

\begin{proof}
In the $dg$-model from \cite{Abouzaid:plumbing}, the endomorphisms of objects are chain complexes underlying classical cohomology, and the morphisms between the objects are given by cochains on a neighbourhood $U \subset Q_1$ of the intersection locus $C$ in one direction, and cochains  relative boundary $C^*(U,\partial U)$ on that neighbourhood in the other. % $S^2\times S^2$ and $C\cong S^2$ are formal (in the classical sense), hence  cochains are quasi-isomorphic to cohomology.
 The identification of the morphism groups as given follows on passing to cohomology, and replacing $H^*(U,\partial U) \simeq H^{*-d}(C)$,  by the Thom isomorphism theorem. For the product structure, note that in the model provided by Lemma \ref{lem:ConormalModelForPairs} all holomorphic triangles are constant; there is a real-valued action functional and all intersection points have the same value of the action, so standard Morse-Bott techniques apply.  The result then follows easily.
\end{proof}

The first arc algebra $H_1 = \bZ\langle 1,x\rangle$ is isomorphic to $\scrH_1^{symp} = H^*(S^2)$, where the cohomology group is based by Corollary \ref{Cor:Same}, which in this case just amounts to saying that the anti-diagonal in $\bP^1\times \bP^1$ inherits a unique orientation from the complex orientation of the first factor, and the opposite orientation of the second factor. This comes with three functors $\cup_i^{comb}: H_1 \rightarrow H_2$.  We also have the symplectic cup functors $\cup_i: \scrH_1^{symp} \rightarrow \scrH_2^{symp}$, which come from the associated elementary bimodules.

\begin{Proposition} \label{Prop:n=2}
There are bases of $HF^*(L_{\wpb}, L_{\mix})$ and of $HF^*(L_{\mix}, L_{\wpb})$ such that the algebra
\[
\scrH_2^{symp}  \ = \ H^*(L_{\wpb}) \oplus HF^*(L_{\wpb}, L_{\mix}) \oplus HF^*(L_{\mix}, L_{\wpb}) \oplus H^*(L_{\mix}) 
\]
is isomorphic to the combinatorial arc algebra $H_2$ in a manner which is compatible with the three cup functors $\cup_i: \scrH_1 \rightarrow \scrH_2$ respectively $\cup_i^{comb}: H_1 \rightarrow H_2$, $i\in \{1,2,3\}$. 
\end{Proposition}

\begin{proof}
 Lemma \ref{Lem:CleanFor-n=2} shows that $L_{\wpb}$ and $L_{\mix}$ can be isotoped to meet cleanly in a two-sphere $C$.  Knowledge of  $H^*(\scrY_2)$, or the flag description given in Section \ref{Sec:Flags}, shows that $C$ \emph{a priori} represents a class of square $\pm 2$ in each factor; with our orientation conventions, the plumbing model is that of two copies of $S^2 \times S^2$ which meet along the diagonal submanifold, of square $+2$.

We take the given bases of the groups $H^*(L_{\wpb})$, $H^*(L_{\mix})$ from Corollary \ref{Cor:Same}, and the induced basis of $H^*(C)$ coming from restriction: this gives a uniquely defined choice of generator of $H^2(C)$.  The gradings agree with those in Proposition \ref{Prop:plumb} by Lemma \ref{Lem:Orient_even}. The Floer products in the plumbing model are then cohomological (convolution products), as in Proposition \ref{Prop:plumb}. 
Since the Euler class of the normal bundle of the intersection locus is the $(1,1)$-class, products of positive generators all have positive coefficients.  Given this, it is straightforward to compare to the arc algebra $H_2$.

We spell out the final comparison in the two most interesting examples.  Consider the products 
\begin{align*}
\Hom_{H_2}(\wpb,\mix) \otimes \Hom_{H_2}(\mix,\wpb) \longrightarrow \Hom_{H_2}(\mix,\mix) \\
\Hom_{H_2}(\mix,\wpb) \otimes \Hom_{H_2}(\wpb,\mix) \longrightarrow \Hom_{H_2}(\wpb,\wpb).
\end{align*}
In both cases, the diagrammatic product involves first merging two circles, and then splitting again, hence is given by the composition $\Delta \circ m$, where 
\[
m: \bZ\langle 1,x\rangle^{\otimes 2} \longrightarrow \bZ\langle 1,x\rangle \quad \textrm{and} \quad \Delta: \bZ\langle 1,x\rangle \longrightarrow\bZ\langle 1,x\rangle^{\otimes 2}
\]
denote the product respectively co-product  in the Frobenius algebra $H^*(S^2) = \bZ\langle 1,x\rangle$. The composite $\Delta \circ m$ therefore takes 
\[
1\otimes 1 \mapsto 1\otimes x + x\otimes 1, \ 1\otimes x \mapsto x\otimes x, \ x\otimes 1 \mapsto x\otimes x,  \  x\otimes x \mapsto 0.
\]
The cohomology convolution product has exactly the same effect, with $1\otimes x + x\otimes 1$ arising from the cohomology class of the diagonal $C \subset S^2\times S^2$. The other cases, and compatibility with the cup-functors, follow similarly on unwinding the definitions.
\end{proof}

\subsection{Iterated $A_2$-fibrations and plumbing models for arbitrary pairs}

  Recall $c(\wp,\wp')$ denotes the number of components of the planar unlink $\wp\cup\overline{\wp'}$, and that the two associated Lagrangians \emph{meet in codimension one} if $c(\wp,\wp') = n-1$. We shall say that two arcs in the unlink are \emph{consecutive} if one belongs to $\wp$ and one to $\overline{\wp'}$ and they share exactly one end-point.

\begin{Lemma} \label{Lem:CleanPairs}
Any pair $L_{\wp}, L_{\wp'}$ may be Hamiltonian isotoped to meet pairwise  cleanly in a submanifold $(S^2)^{c(\wp,\wp')}$. In particular, any pair admits a plumbing model.
\end{Lemma}

\begin{proof}
We give an inductive argument.   Consider an innermost component of the unlink $\wp \cup \overline{\wp'}$, which is formed of matchings $\wp_{in}, \overline{\wp'_{in}}$ on a subset $2m \leq 2n$ of the critical points. If $\wp_{in} = \wp'_{in}$ then these reduced matchings are both composed of a single arc which joins adjacent critical points. There is then a Morse-Bott degeneration of $\scrY_n$ which brings these eigenvalues together along the common arc (which might as well lie on the real axis). In the corresponding open subset of $\scrY_n$, which is a $T^*S^2$-bundle over $\scrY_{n-1}$ by a simpler version of Lemma \ref{lem:local_A2_degeneration}, the Lagrangians $L_{\wp}$ and $L_{\wp'}$ are both fibred by copies of the zero-section. By the inductive hypothesis, the bases of these fibrations, which are Lagrangians in $\scrY_{n-1}$, can be isotoped to meet cleanly.  Therefore, assume no innermost component is an unknot meeting only two of the critical points.

 In that case, the unknot $\wp_{in} \cup \overline{\wp'_{in}}$  admits a pair of consecutive arcs.  We consider a degeneration of $\scrY_n$ in which the three eigenvalues which are end-points of these arcs come together at one of the outermost of the three points, by moving the eigenvalues along the given paths.  This yields an open subset of $\scrY_n$ which is an $A_2$-fibration over $\scrY_{n-1}$, and in which the matchings $\wp_{in}, \overline{\wp'_{in}}$ are explicitly fibred over (not necessarily half-plane) crossingless matchings in the base, with fibres being the core arcs of the $A_2$-space which meet transversely once.  This procedure can be iterated, until the Lagrangians in the base have no consecutive arcs, which happens only when they co-incide up to isotopy (the last conclusion uses the innermost condition; otherwise the arcs might differ by Markov I moves).  The upshot is that there is an open subset of $\scrY_n$ which is an iterated $A_2$-fibration over $\scrY_{n-k}$, where $n-k=c(\wp,\wp')$, in which the two Lagrangians are fibred over Hamiltonian isotopic Lagrangians in $\scrY_{n-k}$, with fibres which are themselves products of pairwise transverse vanishing cycles in the $A_2$-fibres.  The result follows. 
\end{proof}

\begin{Lemma} \label{Lem:IntersectionLocusIsCorrect}
In the situation of Lemma \ref{Lem:CleanPairs}, one may further assume that the intersection  of the cleanly intersecting Hamiltonian images of $L_{\wp}, L_{\wp'}$ is smoothly isotopic to the iterated antidiagonal $\hat{L}_{\wp} \cap \hat{L}_{\wp'}$. In particular, we have an isomorphism $HF^*(L_{\wp},L_{\wp'})  \cong  H^*(\hat{L}_{\wp} \cap \hat{L}_{\wp'}) $ compatible with the natural module structures of both cohomologies over $H^*(\scrY_n) $. 
\end{Lemma}

\begin{proof}
In the inductive construction of the previous set-up, consider some innermost component of $\wp \cup \overline{\wp'}$, which involves a subset $I \subset \{1,2,\ldots, 2n\}$ of cardinality $2m \leq 2n$ of the critical points. The degenerations along successive arcs for this component do not affect any of the other components, by the innermost condition. Furthermore, they exhibit the space $\scrY_m$ as an iterated $A_2$-fibration over $T^*S^2$, and the Lagrangians $L_{\wp_{in}}$ and $L_{\wp'_{in}}$ intersect along a section of this iterated fibration, i.e. they both fibre over the zero-section with fibres meeting transversely once. It suffices to show that this intersection sphere is naturally identified with the small antidiagonal in a copy $(S^2)^{2m} \subset (S^2)^{2n}$ indexed by the subset $I$. As in Lemma \ref{Lem:IsotopicCoreVCycle},  this holds because the simultaneous resolution shows the vanishing cycles of an $A_1$- respectively $A_2$-degeneration depend up to smooth isotopy only on the corresponding pair respectively triple of critical points which are brought together.
\end{proof}

\begin{cor} \label{cor:HF_simple_over_center}
The Floer cohomology group $HF^*(L_{\wp},L_{\wp'})  $ is a cyclic module over $ H^*(\scrY_n) $, generated by any minimal degree generator. \qed
\end{cor}

\subsection{Conormal triples}
Lemma \ref{lem:ConormalModelForPairs} asserts that a pair of cleanly intersecting Lagrangians can always be modelled by taking one to be a  conormal bundle in the cotangent bundle of the other. For three Lagrangians, the situation is slightly more subtle. We recall the linear situation.

%\begin{proof}
%Again see \cite{KOh,NadlerZaslow}.
%\end{proof}

\begin{Lemma} \cite{CLM}
Let $(V,\omega)$ be a symplectic vector space. An ordered triple of Lagrangian subspaces $(\Lambda_0, \Lambda_1, \Lambda_2)$ is determined up to symplectomorphism by the quintuple of integers
\begin{equation} \label{Eqn:quintuple}
b_{ij} = \dim (\Lambda_i \cap \Lambda_j), \ b_{012} = \dim (\Lambda_0 \cap \Lambda_1 \cap \Lambda_2), \ s(\Lambda_0,\Lambda_1,\Lambda_2)
\end{equation}
where $s$ is the Maslov triple index, i.e. the signature of the quadratic form on $\Lambda_0 \oplus \Lambda_1 \oplus \Lambda_2$ given by  $\omega(u_0,u_1) + \omega(u_1,u_2) + \omega(u_2,u_0).$  \qed
\end{Lemma}

\begin{Lemma}
If the dimensions in \eqref{Eqn:quintuple} satisfy 
\begin{equation} \label{numerical_maslov_constraint}
\dim_{\bC}(V)+2b_{012} = b_{01}+b_{02}+b_{12}
\end{equation} then the Maslov triple index $s=0$. 
\end{Lemma}

\begin{proof}
Let $m=\dim_{\bC}(V)$.  Let $\Lambda_{012}$ denote the triple intersection $\Lambda_0 \cap \Lambda_1 \cap \Lambda_2$ and $\Lambda_{ij}$ the corresponding double intersections for $i\neq j$.   Both the Maslov triple index and the condition \eqref{numerical_maslov_constraint} are invariant under symplectic reduction, so passing to $\Lambda_{012}^{\perp_{\omega}} / \Lambda_{012}$ we can assume $b_{012} = 0$. This means that the vector space $W = \Lambda_{01} \oplus \Lambda_{12} \oplus \Lambda_{20}$ has $\dim_{\bR}(W) =m$; since the quadratic form vanishes identically on $W$, it suffices to prove that the signature of $V/W$ is zero. But the $3m$-dimensional space $\Lambda_0 \oplus \Lambda_1 \oplus \Lambda_2$ contains isotropic subspaces $S = \Lambda_0 \oplus \Lambda_{12} \oplus \{0\}$ and $T = \{0\} \oplus \Lambda_{12} \oplus \Lambda_2$, which project to $V/W$ as transverse $m$-dimensional subspaces, hence $V/W$ is the direct sum of their images. 
\end{proof}

\begin{Lemma} \label{lem:model3}
Three cleanly intersecting real $m$-dimensional Lagrangians $L_0, L_1, L_2$ are locally symplectomorphic near $L_0 \cap L_1 \cap L_2$  to two conormal bundles inside $T^*L_0$ if and only if the dimension constraint holds:
\begin{equation} \label{dimconstraint}
m+ 2b_{012} = b_{01} + b_{12} + b_{02}.
\end{equation}
 In particular, this condition is symmetric under permuting indices.
\end{Lemma}

\begin{proof}[Sketch]
Given two submanifolds $B_0, B_1 \subset L_0$ intersecting cleanly in a submanifold of dimension $d$, their conormal bundles $\nu^*_{B_i} \subset T^*L_0$ meet cleanly in a submanifold of dimension $2d$. This implies that a pair of conormal bundles satisfies the desired equality.  Conversely, suppose the dimension constraint holds, and work inside $T^*L_0$ with $L_1 = \nu^*_{L_0 \cap L_1}$. From the linear algebra classification of triples of Lagrangian subspaces, and the vanishing of the Maslov triple index established above, at a point  $p$ in the triple intersection $B$ we can suppose that $T_pL_2 = T_p(\nu^*_{L_0 \cap L_2})$.  One then constructs a local Hamiltonian isotopy, in an open neighbourhood $U$ of $B$,  taking $L_2\cap U $ to $\nu^*_{L_0 \cap L_2} \cap U$, via a flow tangent to $L_0\cap U$ and $\nu^*_{L_0 \cap L_1} \cap U$, as in Lemma \ref{lem:ConormalModelForPairs}.
\end{proof}

We point out that \eqref{dimconstraint} never holds if the triple intersection $L_0\cap L_1 \cap L_2$ co-incides with the three pairwise intersections $L_i \cap L_j$ (and the Lagrangians are not all identical). Given that, it is easy to see that there are triples of components $\hat{L}_{\wp_i}$ of the compact core $Z$ from Lemma \ref{Lem:Wehrli}  which, even though meeting cleanly, do not admit a conormal model.
%Conversely, consider three cleanly intersecting Lagrangian submanifolds $L_i$ satisfying \eqref{dimconstraint}.  By Lemma \ref {lem:ConormalModelForPairs}, we can suppose that we work in $T^*L_0$ and that $L_1 = \nu^*B_1$ is locally given by a conormal bundle of a submanifold $L_0 \cap L_1 = B_1 \subset L_0$.  Let $L_0 \cap L_2 = B_2 \subset L_0$.  The linear situation implies that at each  $p\in  B = B_1 \cap B_2$, the tangent space $T_p L_2$ agrees with the tangent space of the conormal $\nu^*_{B_2}$. One then constructs a local Hamiltonian isotopy preserving $L_0$ and $\nu^*_{B_1}$ over $B$, as in Lemma \ref{lem:ConormalModelForPairs}.
%\end{proof}

For triples of Lagrangians, $A_2$-degenerations are not sufficient to bring the Lagrangians into clean intersection position (consider the case where no triple of critical points contains the end-points of at least one arc in each matching). Nonetheless, we have:

\begin{Lemma}\label{lem:OneArcChanged}
Any triple of Lagrangians $L_{\wp}, L_{\wp'}, L_{\wp''}$, two of which meet in codimension one,  may be Hamiltonian isotoped to simultaneously meet pairwise cleanly.  Moreover, the analogue of Lemma \ref{Lem:IntersectionLocusIsCorrect} again holds, and the cleanly intersecting Hamiltonian images  satisfy the conormal condition \eqref{dimconstraint}.
\end{Lemma}

\begin{proof}
Suppose $L_{\wp}$ and $L_{\wp'}$ meet in codimension one, meaning that they share $n-2$ arcs. We repeat the argument of Lemma \ref{Lem:CleanPairs}, but considering consecutive arcs one of which is common to $\wp$ and $\wp'$ and the other of which lies in $\wp''$.  The corresponding $A_2$-degeneration bringing eigenvalues together along such a consecutive pair of arcs manifests all three Lagrangians as fibred simultaneously, with  $L_{\wp}$ and $L_{\wp'}$ having identical fibres in the local $A_2$-bundle, and $L_{\wp''}$ having fibre the other core sphere in the $A_2$-fibres. This can now be  iterated as before, and the analogue of Lemma \ref{Lem:IntersectionLocusIsCorrect} also holds as before.

For the final statement, 
in the notation of \eqref{dimconstraint}, let $L_0=L_{\wp}$, $L_1=L_{\wp'}$ and $L_2=L_{\wp''}$, and  set $m=2n$ to be the real dimension of $L_{\wp}$. Then $b_{01} = 2n-2$. The dimensions $b_{02}$ and $b_{12}$ differ by exactly 2, which means that (relabelling as necessary) we can assume $b_{02} = 2j$ and $b_{12} = 2j+2$.  The triple intersection has codimension at most 2 in $L_{\wp'}  \cap L_{\wp''}$, hence has dimension at least $2j$, but is also contained in the $2j$-dimensional submanifold $L_{\wp} \cap L_{\wp''}$. Therefore $b_{012} = 2j$, and $2n+2b_{012} = b_{01} + b_{02} + b_{12}$.
\end{proof}

The convolution model for Floer product for a pair of cleanly intersecting Lagrangians admits a generalisation to the case of two conormal bundles. Again from \cite{KOh,NadlerZaslow}:

\begin{Lemma} \label{lem:Conv3}
Let $A_i \subset Q$ be closed submanifolds of an oriented spin manifold $Q$, $i=1,2$.  Let $\nu_i^* \subset T^*Q$ denote the conormal bundle of $A_i$. The product 
\begin{equation}
HF^*(\nu_1, \nu_2) \otimes HF^*(Q, \nu_1)  \longrightarrow HF^*(Q,\nu_2)
\end{equation}
is a convolution product, given up to sign by restriction and push-forward
\begin{equation}
H^*(A_1 \cap A_2) \otimes H^*(A_1) \rightarrow H^*(A_1\cap A_2) \rightarrow H^*(A_2).
\end{equation} \qed
%In particular, a minimal degree generator of the image (which is well-defined up to sign) lies in the image of the product. 
\end{Lemma}

\subsection{Some non-zero Floer products}

We collect together several non-vanishing results for Floer products which will be used in the construction of positive bases.  For a pair of matchings $\wp, \wp'$ we will denote by
\begin{equation} \label{Eqn:MinDegreeGenerator}
\alpha_{\wp,\wp'} \in k_{min}(\wp,\wp') \subset HF^*(L_{\wp},L_{\wp'})
\end{equation}
a minimal degree generator of $HF^*(L_{\wp}, L_{\wp'})$. There are two choices, differing by sign. 
%It is sometimes helpful to recall that in the symmetric gradings of Floer groups used in the prequel paper, 
%\[
%k_{min}(\wp,\wp') = HF^{n-c(\wp,\wp')}(L_{\wp},L_{\wp'})
%\]
%and that Floer group is supported in the degrees of a fixed parity $n-c(\wp,\wp') \leq \ast \leq n+c(\wp,\wp')$.   
By Corollary \ref{cor:HF_simple_over_center}, the Floer product
\begin{equation} \label{Eq:Now_Known}
HF^*(L_{\wp'}, L_{\wp''}) \otimes HF^*(L_{\wp},L_{\wp'}) \longrightarrow HF^*(L_{\wp}, L_{\wp''})
\end{equation}
 is completely determined by that of the minimal degree generators and the module structure.
Despite its seeming innocuity, the following result is the main reason why the arc algebra and its symplectic analogue agree, since it asserts that, for the simplest products, basis elements appear with coefficients that have the same sign.

We now consider a triple of matchings $\wp, \wp', \wp''$, and suppose $\wp$ and $\wp'$ meet in codimension one. As remarked after Lemma \ref{Lem:Meet_codimension_one}, necessarily $c(\wp,\wp'') = c(\wp',\wp'') \pm 1$.  
\begin{Lemma} \label{Lem:NonzeroProducts}
\begin{enumerate}
\item If $c(\wp,\wp'') = c(\wp',\wp'') -1$, then  
\[
\alpha_{\wp',\wp''} \cdot \alpha_{\wp,\wp'} = \pm \alpha_{\wp,\wp''}.\]
\item If $c(\wp,\wp'') = c(\wp',\wp'') +1$, then
\[
\alpha_{\wp',\wp''}\cdot\alpha_{\wp,\wp'} = \pm (v_{i_j}+v_{i_k}) \alpha_{\wp,\wp''} \neq 0
\]
for some $v_{i_j} , v_{i_k} \in H^*(\scrY_n)$ elements of the distinguished basis. Specifically, $ [v_{i_j}+v_{i_k}] \in H^*(\hat{L}_{\wp}\cap \hat{L}_{\wp''}) = HF^*(L_{\wp}, L_{\wp''})$ is Poincar\'e dual to the homology class of the triple intersection $[\hat{L}_{\wp} \cap \hat{L}_{\wp'} \cap \hat{L}_{\wp''}]$.
\end{enumerate}
\end{Lemma}

\begin{proof}
%We work in a clean intersection model obtained by Hamiltonian isotopy as in Lemma \ref{lem:OneArcChanged}, and note that the three Lagrangians are then modelled on conormal bundles, as in Lemma \ref{lem:Conv3}.  The matchings $\wp$ and $\wp'$ differ in exactly two arcs. These either belong to a single component of $\wp\cup\overline{\wp''}$, and hence to two components on $\wp' \cup \overline{\wp''}$, or vice-versa. This implies $c(\wp,\wp'') = c(\wp',\wp'') \pm 1$.   

The non-vanishing of the product if $c(\wp,\wp'') = c(\wp',\wp'') -1$  follows immediately from Lemma \ref{lem:OneArcChanged} and Lemma \ref{lem:Conv3}. If $c(\wp,\wp'') = c(\wp',\wp'') +1$, the product $\alpha_{\wp',\wp''}\cdot\alpha_{\wp,\wp'}$ lands in degree exactly two higher than the minimal degree.  For the cleanly intersecting representatives of the three Lagrangians, the triple intersection $L_{\wp} \cap L_{\wp'} \cap L_{\wp''}$ co-incides with the intersection $L_{\wp} \cap L_{\wp''}$, cf. the proof of Lemma \ref{lem:OneArcChanged}. This defines a class 
\[
[L_{\wp} \cap L_{\wp'} \cap L_{\wp''}] \in H^2(L_{\wp'} \cap L_{\wp''}) \cong H^2((S^2)^{c(\wp',\wp'')}
\]
which is the product of the minimal degree generators by the convolution model for Floer products in conormal plumbings.  The cohomology class can be identified with the corresponding class in the model of Lemma \ref{Lem:Wehrli}, by Lemma \ref{Lem:IntersectionLocusIsCorrect}. 
Geometrically, the triple intersection is homologous to an antidiagonal in two  factors 
\[
S^2_{i_j} \times S^2_{i_k} \ \subset \ S^2_{i_1} \times \cdots \times S^2_{i_m} \ \subset (S^2)^{2n},
\] $m= c(\wp',\wp'')$, and each such is canonically the restriction of a class $(v_{i_j}+v_{i_k}) \in H^2(\scrY_n)$. This yields the second statement.
\end{proof}

\begin{Lemma} \label{Lem:EverythingFromMinimalDegree}
If $\wp, \wp'$ meet in codimension $k$, i.e. $c(\wp, \wp')=n-k$, the minimal degree generator $\alpha_{\wp,\wp'} \in HF^{*}(L_{\wp},L_{\wp'})$ can be written as a product
\[
\alpha_{\wp,\wp'} = \pm \prod_{j=k-1}^{0} \alpha_{\wp_i,\wp_{i+1}}
\]
where $\wp=\wp_0,\ldots,\wp_k=\wp'$ is a codimension one interpolating sequence as in Lemma \ref{Lem:Codim1Interpolate}.
\end{Lemma}

\begin{proof}
The fact that the interpolating sequence has minimal possible length implies that 
\[
c(\wp,\wp_{i+1}) = c(\wp,\wp_i)-1
\] for each $i\geq 1$.  Then use Lemma \ref{Lem:NonzeroProducts}. 
\end{proof}

\begin{Lemma} \label{Lem:FormOfMinDegreeProduct}
For any $\wp, \wp', \wp''$, there is an identity
\begin{equation} \label{Eqn:UseCodim1}
\alpha_{\wp',\wp''} \cdot \alpha_{\wp,\wp'} = \pm \prod_j (v_{j_1} + v_{j_2}) \alpha_{\wp,\wp''}
\end{equation}
where each $(v_{j_1}+v_{j_2}) \in H^2(\scrY_n)$ is a sum of positive basis elements and is dual to an antidiagonal in $(S^2)^{2n}$ (the right side of \eqref{Eqn:UseCodim1} may vanish).
\end{Lemma}

\begin{proof}
Choose a sequence $\wp = \wp_0, \wp_1,\ldots, \wp_k = \wp'$ with 
where $k = n-c(\wp,\wp')$ as before. We now write 
\[
\alpha_{\wp',\wp''}\cdot\alpha_{\wp,\wp'} \ = \ \pm \prod_j \left(((\alpha_{\wp',\wp''}\alpha_{\wp_{k-1},\wp_k})\alpha_{\wp_{k-2},\wp_{k-1}})\cdots\alpha_{\wp_{0},\wp_{1}}\right)
\]
and again appeal to  Lemma \ref{Lem:NonzeroProducts}.
\end{proof}

\begin{Lemma} \label{Lem:PlaitMixNonzeroProducts}
For any $\wp,\wp'$, the products
\[
\alpha_{\wpb,\wp'} \cdot \alpha_{\wp,\wpb} \in HF^*(L_{\wp}, L_{\wp'})  \quad  \textrm{and} \quad  \alpha_{\mix,\wp'} \cdot \alpha_{\wp,\mix}  \in HF^*(L_{\wp}, L_{\wp'}) 
\]
are both non-zero.
\end{Lemma}

\begin{proof}
For $\wpb$, this  is exactly \cite[Corollary 5.19]{AbSm}.  To argue for $\mix$, consider for a moment the fibre $\scrY_n^{\mathbf{\mu}}$ of the adjoint quotient corresponding to placing the $2n$ critical values at the roots of unity.  There are  ``cyclic" analogues  $\wpb^{cyc}, \mix^{cyc}$ of $\wpb, \mix$ for this configuration, which define Lagrangian submanifolds of $\scrY_n^{\mathbf{\mu}}$ which are Hamiltonian isotopic to our usual crossingless matching Lagrangians via parallel transport in the family over configuration space $\Conf_{2n}(\bC)$ joining $\mu$ and $\{1,2,\ldots, 2n\}$.  The matchings $\wpb^{cyc}$ and $\mix^{cyc}$ are exchanged by the symplectomorphism of $\scrY_n^{\mathbf{\mu}}$  induced by cyclic rotation by $\exp(i\pi/n)$ in the base of the complex surface $A_{2n-1}$. That, and Hamiltonian invariance of the statement of the Lemma,  implies the corresponding non-vanishing for products involving $L_{\mix}$.
\end{proof}

%Finally, we consider the (anti-holomorphic) involution of the Milnor fibre $uv + \prod (z-i) = 0 $ given by 
%\begin{equation}
%  (u,v,z) \mapsto (\frac{1}{\bar{u}}, \frac{1}{\bar{v}}  , \bar{z}),
%\end{equation}
%which maps upper to lower matchings. Since these matchings are Hamiltonian isotopic, we obtain an isomorphism
%\begin{equation}
%  HF^*(L_{\wp}, L_{\wp'}) \cong HF^*(L_{\wp'}, L_{\wp})
%\end{equation}
%which is compatible with products in the following sense:
%\begin{lem}
%  Triviality of the product on $  HF^*(L_{\wp}, L_{\wp''}) \otimes HF^*(L_{\wp'}, L_{\wp}) $ is equivalent to triviality of the product on $  HF^*(L_{\wp}, L_{\wp'}) \otimes HF^*(L_{\wp''}, L_{\wp}) $.
%\end{lem}
\section{An inductive construction of positive bases}

\subsection{Positive pairs, triangles, and triples} Our aim is to construct a  basis of the symplectic arc algebra which is positive and preserved by cup functors, in the following sense.

\begin{Definition} \label{Defn:PositiveBasis}
A basis of $\scrH_n^{symp}$ is \emph{positive} if every product of minimal degree generators is zero or of the form given in \eqref{Eqn:UseCodim1} with a $+$ sign. We say a map of based vector spaces \emph{preserves bases} if it takes every basis element to zero or to a basis element.
\end{Definition}

Fix monomial bases of $HF^*(L_{\wp}, L_{\wp}) = H^*(L_{\wp})$ for each $\wp$ as in Corollary \ref{Cor:Same}. Such a basis determines a trace %top-dimensional homology class 
\begin{equation}
\tr \co H^{2n}(L_{\wp}) \to \bZ
\end{equation}
which is dual to the cohomology class $[L_{\wp} ] = v_{i_1} \cdot v_{i_2} \cdot \cdots \cdot v_{i_n}$.
\begin{lem} \label{lem:top_degree_product}
An $n$-fold product $  \prod_{j=1}^{n} (v_{j_1} + v_{j_2}) $ of sums of positive generators of $H^2(L_{\wp})$ either vanishes, or is a positive multiple of $ [L_{\wp} ] $.   \qed
\end{lem}
The cohomological shadow of the fact that the Fukaya category is cyclic is the statement that the trace is symmetric, i.e that we have
\begin{equation} \label{eq:trace_cyclic}
\tr( \alpha \cdot \beta) = \tr(\beta \cdot \alpha)
\end{equation}
whenever $\alpha \otimes \beta \in HF^*(L,L') \otimes HF^*(L',L)$, cf. \cite[Section 12(e)]{FCPLT}. Note that there is no sign in the above formula, because all morphisms have even degree according to the grading convention fixed in Definition \ref{Def:Depth}.

For any $\{\wp,\wp'\}$, $HF^*(L_{\wp}, L_{\wp'})$ is a quotient of both $HF^*(L_{\wp}, L_{\wp})$ and $HF^*(L_{\wp'}, L_{\wp'})$ via the natural module structures. A choice of minimal degree generator $\alpha_{\wp,\wp'}$ yields a basis for $HF^*(L_{\wp}, L_{\wp'})$ by multiplication by elements of $H^*(L_{\wp})$ on the right or by elements of $H^*(L_{\wp'})$ on the left; these bases agree, since in both cases we can re-interpret the multiplication as coming from an element of $H^*(\scrY_n)$, which acts centrally.  In line with Definition \ref{Defn:PositiveBasis}, we will more generally say that  a product of  minimal degree generators is ``positive" (or is  ``positive with respect to the minimal degree generator") if it has the form of \eqref{Eqn:UseCodim1} with a $+$ sign. 
%A choice of minimal degree generator $\alpha_{\wp,\wp'}$ yields a choice of minimal degree generator $\alpha_{\wp',\wp}$ by using non-triviality of the product
%\[
%HF^*(L_{\wp'}, L_{\wp}) \otimes HF^*(L_{\wp}, L_{\wp'}) \longrightarrow H^*(L_{\wp})
%\]
%which hits $\pm [L_{\wp}]$ by Poincar\'e duality.  
We say that a pair $(\wp,\wp')$ is positive if we have chosen  minimal degree generators $\alpha_{\wp,\wp'}$ and $\alpha_{\wp',\wp}$ of the Floer groups between them so that any of the conditions in the following result hold:
\begin{lem} \label{lem:positive_basis_pair}
The following are equivalent:
\begin{enumerate}
\item There are positive basis elements $v_{j_k} \in H^2(\scrY_n)$  such that $\prod_{j} (v_{j_1} + v_{j_2}) \alpha_{\wp',\wp} \cdot \alpha_{\wp,\wp'} $ is a positive multiple of the fundamental class $ [L_{\wp} ]  $.
\item There are positive basis elements $v_{j_k} \in H^2(\scrY_n)$  such that $\prod_{j} (v_{j_1} + v_{j_2}) \alpha_{\wp,\wp'} \cdot \alpha_{\wp',\wp}  $ is a positive multiple of the fundamental class $ [L_{\wp'} ]  $.
\item The products $\alpha_{\wp,\wp'} \cdot \alpha_{\wp',\wp} $ and $\alpha_{\wp',\wp} \cdot \alpha_{\wp,\wp'}   $ are positive.
\end{enumerate}
\end{lem}
\begin{proof}
Pick $v= \prod_{j} (v_{j_1} + v_{j_2})  $ so that $v\cdot \alpha_{\wp,\wp'} $ is a top-degree generator of  $HF^*(L_{\wp}, L_{\wp'})$.  Since all Floer morphisms that we are considering have even degree, the equivalence of the first two statements follows from the fact that the Poincar\'e duality pairing introduces no signs, cf. \eqref{eq:trace_cyclic}.  The final part follows from cyclicity of the module structure. \end{proof}

We now consider three matchings $\{ \wp_i \}_{i=0,1,2}$, and choices of minimal degree generators $ \alpha_{\wp_i,\wp_{i+1}}$, where the index is cyclic (i.e. considered modulo $3$).
%, and similarly for the generators in the other direction given by the previous lemma.
\begin{Lemma} \label{lem:positive-triples-cyclic} The following conditions are equivalent:
\begin{enumerate}
%\item The product $ \alpha_{\wp_1,\wp_2} \cdot \alpha_{\wp_0,\wp_1}  $ is positive with respect to $\alpha_{\wp_0,\wp_2} $.
  \item The triple product $ \alpha_{\wp_1,\wp_2} \cdot \alpha_{\wp_0,\wp_1}  \cdot \alpha_{\wp_2,\wp_0}  $ is positive. 
\item The three possible triple  products involving  $ \alpha_{\wp_1,\wp_2}$, $ \alpha_{\wp_0,\wp_1}$ and $ \alpha_{\wp_2,\wp_0} $ are positive.
%\item All possible pairwise products are positive.
  \end{enumerate}
\end{Lemma}
\begin{proof}
%The equivalence of the first two conditions (and of the last two) is a variant of Lemma \ref{lem:positive_basis_pair}. To prove the equivalence of the first and the third condition, 
Consider a product $\prod_{j} (v_j + v'_j)$  of sums of positive generators so that 
\begin{equation}
\prod_{j} (v_{j} + v'_{j})  \alpha_{\wp_1,\wp_2} \cdot \alpha_{\wp_0,\wp_1}  \cdot \alpha_{\wp_2,\wp_0} 
\end{equation}
is a positive multiple of the top degree generator of the cohomology of $ L_{\wp_2}$. Associativity and cyclicity of the trace imply that the product in a different order (but with the same cyclic order) is a positive multiple of a top degree generator of the cohomology of $ L_{\wp_0}$ or $L_{\wp_1}  $. 
\end{proof}
We say that $\{ \wp_i \}_{i=0,1,2}$ form a \emph{positive triangle} if we have fixed minimal degree generators $\alpha_{\wp_i,\wp_j} $ so that the conditions of Lemma \ref{lem:positive-triples-cyclic} hold in either cyclic ordering. There is no obstruction to a triangle being positive: given two of the three morphisms, we can pick the third so that all cyclic products are positive, and the choices for the two cyclic orderings are independent.  
\begin{defn}
The matchings $( \wp_0,\wp_1,\wp_2)$ form a \emph{positive triple} if the corresponding triangle as well as all pairs are positive.
\end{defn}

\begin{lem} \label{lem:positive-triples}
If $ (\wp_0,\wp_1,\wp_2) $ form a positive triangle, two of the pairs are positive, and there are three non-cyclically ordered generators, involving all three Lagrangians, whose product does not vanish, then $ (\wp_0,\wp_1,\wp_2)$ form a positive triple. 
\end{lem}
\begin{proof}
%\marginpar{I didn't understand this and have screwed it up so now it makes no sense} 
By relabelling the matchings and using cyclic symmetry, we may assume that the pairs $(\wp_0, \wp_1)$ and $(\wp_0, \wp_2)$ are positive, that the triangle is positive, and that either  (1) $ \alpha_{\wp_2,\wp_0} \cdot \alpha_{\wp_1,\wp_2}  \cdot \alpha_{\wp_2,\wp_1}   $   or (2) $  \alpha_{\wp_2,\wp_1} \cdot \alpha_{\wp_0,\wp_2}  \cdot \alpha_{\wp_2,\wp_0}  $  is non-zero.   We must show in either case that the pair $ (\wp_1, \wp_2) $ is positive.

{\bf Case 1:} By associativity, we have 
\begin{equation}
   \alpha_{\wp_2,\wp_0} \cdot \left( \alpha_{\wp_1,\wp_2}  \cdot \alpha_{\wp_2,\wp_1} \right)    =   \left( \alpha_{\wp_2,\wp_0} \cdot \alpha_{\wp_1,\wp_2} \right)  \cdot \alpha_{\wp_2,\wp_1}.
\end{equation}
Since the pair $(\wp_0, \wp_1)$ is positive, positivity of the triangle implies that the product $  \alpha_{\wp_2,\wp_0} \cdot \alpha_{\wp_1,\wp_2} $ is  positive with respect to $\alpha_{\wp_1, \wp_0}$ (i.e. the product of a positive basis element of $H^*(\scrY_n)$ with this class). Since the pair $(\wp_0, \wp_2)$ is positive, positivity of the triangle further implies that the product of this class with $ \alpha_{\wp_2,\wp_1}  $ is positive with respect to $\alpha_{\wp_2, \wp_0} $. Equating the left and right hand sides above, we conclude that the product $ \alpha_{\wp_1,\wp_2}  \cdot \alpha_{\wp_2,\wp_1}  $ is positive.

{\bf Case 2:} By Poincar\'e duality, our assumption implies that we have a non-zero product
\begin{equation}
 \left(    \alpha_{\wp_1,\wp_2} \cdot  \alpha_{\wp_2,\wp_1} \right) \cdot \left( \alpha_{\wp_0,\wp_2}  \cdot \alpha_{\wp_2,\wp_0} \right) =    \alpha_{\wp_1,\wp_2} \cdot  \left( \alpha_{\wp_2,\wp_1}  \cdot  \alpha_{\wp_0,\wp_2} \right) \cdot \alpha_{\wp_2,\wp_0}.
\end{equation}
Using positivity of the triangle under both cyclic orderings, we see that the right hand side is positive as in the previous case. Since the pair  $(\wp_0, \wp_2)$ is positive, the expression $  \alpha_{\wp_2,\wp_1} \cdot \alpha_{\wp_0,\wp_2}  \cdot \alpha_{\wp_2,\wp_0}  $ is positive with respect to $\alpha_{\wp_2,\wp_1}   $. We conclude that $ \alpha_{\wp_1,\wp_2}  \cdot \alpha_{\wp_2,\wp_1}  $ is positive.
\end{proof}

%\begin{equation} \label{Eqn:NonZeroProduct}
%\alpha_{\wp_1,\wp_2} \cdot \alpha_{\wp_0,\wp_1} = + \prod_{j=1}^k (v_{j_1} + v_{j_2}) \alpha_{\wp_0,\wp_2} \, \neq 0
%\end{equation}
%for some positive basis elements $v_{j_i} \in H^2(\scrY_n)$.  (Once the product of minimal degree generators is non-trivial, it necessarily has the appropriate form, by Lemma \ref{Lem:FormOfMinDegreeProduct}.)  

%The Floer product 
%\[
%HF^*(L_{\wp_1}, L_{\wp_2}) \otimes HF^*(L_{\wp_0}, L_{\wp_1}) \longrightarrow HF^*(L_{\wp_0}, L_{\wp_2})
%\]
%can only be non-trivial if it is non-trivial on minimal degree generators; if it is non-trivial, choices of bases on the domain groups impose a distinguished choice on the target group which makes the triple a positive triple.  

 \begin{cor} \label{cor:positivity_triple_codim_1}
 Assume that $(\wp_0, \wp_1, \wp_2)$ form a positive triangle, and two of $(\wp_0, \wp_1, \wp_2)$ meet in codimension one. Given positive bases for two pairs of matchings  $(\wp_i, \wp_j) $, there are unique bases for the remaining pair which yields a positive triple. 
 \end{cor}

\begin{proof}
This follows on combining  Lemma \ref{Lem:NonzeroProducts}, which yields a suitable non-vanishing product,  with the previous result.
\end{proof}

To ease notation, in the next Lemma we write $\alpha_{ij}$ for $\alpha_{\wp_i,\wp_j}$. 
%\begin{Lemma} \label{Lem:ConormalTripleProducts} Let $\wp_0,\wp_1,\wp_2, \wp_3$ be matchings. If there are two positive triples, and one positive triangle, then the remaining triangle is positive whenever the cyclic products $ \alpha_{01}  \cdot \alpha_{30} \cdot \alpha_{23} \cdot \alpha_{12} $ and $ \alpha_{03} \cdot \alpha_{10} \cdot \alpha_{21} \cdot \alpha_{32}   $  do not vanish. \marginpar{I believe one of these non-vanishing implies the other}
%\end{Lemma}
%\begin{proof}
%Since positivity of a cyclic product does not depend on the cyclic ordering, we may (cyclically) relabel the matchings and assume that $(\wp_0, \wp_1, \wp_3) $ is one of the positive matchings. There are therefore three cases to consider, depending on whether the remaining positive triple is $  (\wp_1, \wp_2, \wp_3)$,  $(\wp_0, \wp_2, \wp_3)$, or $(\wp_0, \wp_1, \wp_2)$. The last two cases are equivalent by considering the reversed cyclic ordering:
%{\bf Case 1:} 
%{\bf Case 2:} We consider one of the two possible ordering: since $ (\wp_0, \wp_1, \wp_3)  $ is a positive triangle, the product $\alpha_{01}  \cdot \alpha_{30} $  is a positive multiple of $\alpha_{31}$. So $\alpha_{31} \cdot \alpha_{23} \cdot \alpha_{12}  $  is positive if and only if the quadruple cyclic product is positive. On the other hand, $ \alpha_{30} \cdot \alpha_{23}  $ is a positive multiple of $\alpha_{20} $ (because the triple  $ (\wp_0, \wp_2, \wp_3)  $ is positive), hence the quadruple product is positive because $ \alpha_{01}  \cdot \alpha_{20}  \cdot \alpha_{12} $ is positive.
%\end{proof}

\begin{Lemma} \label{Lem:ConormalTripleProducts} Let $\wp_0,\wp_1,\wp_2, \wp_3$ be matchings. If the triples containing $ (\wp_0, \wp_3) $  are positive and the triangle $(\wp_0, \wp_1, \wp_2) $ is positive, then the remaining triangle $(\wp_1, \wp_2, \wp_3)$ is positive whenever the cyclic products $ \alpha_{01}  \cdot \alpha_{30} \cdot \alpha_{23} \cdot \alpha_{12} $ and $ \alpha_{03} \cdot \alpha_{10} \cdot \alpha_{21} \cdot \alpha_{32}   $  do not vanish.
% \marginpar{I believe one of these non-vanishing implies the other}
\end{Lemma}
\begin{proof}
We consider one of the two possible ordering: since $ (\wp_0, \wp_1, \wp_3)  $ is a positive triangle, the product $\alpha_{01}  \cdot \alpha_{30} $  is a positive multiple of $\alpha_{31}$. So $\alpha_{31} \cdot \alpha_{23} \cdot \alpha_{12}  $  is positive if and only if the quadruple cyclic product is positive. On the other hand, $ \alpha_{30} \cdot \alpha_{23}  $ is a positive multiple of $\alpha_{20} $ (because the triple  $ (\wp_0, \wp_2, \wp_3)  $ is positive), hence the quadruple product is positive because $ \alpha_{01}  \cdot \alpha_{20}  \cdot \alpha_{12} $ is positive.
\end{proof}

We note that non-trivial iterated products as required for the second part of Lemma \ref{Lem:ConormalTripleProducts} arise naturally when combining Lemmas \ref{Lem:NonzeroProducts} and \ref{Lem:EverythingFromMinimalDegree}. 

%It is helpful to think in terms of a triple with vertices $L_{\wp_0}, L_{\wp_1}, L_{\wp_2}$, barycentrically subdivided with $L_{\wp_2'}$ at the interior central vertex, compare to Figure \ref{Fig:4Triples}. 
%The four sub-triples correspond to the three interior triples and the original outer triple, and the Lemma asserts that if three of the four are labelled by a $+$ then the final triple must also be labelled by a $+$.

\subsection{Description of the basis\label{Sec:Induction}}  We construct a basis of $\scrH_n^{symp}$ by the following procedure. Denote by $V_i$ both the Lagrangian 2-sphere in $A_{2n-1}$ joining the points $\{i, i+1\} \subset \{1,2, \ldots,2n\}$, and also the Lagrangian 2-sphere fibre of the elementary correspondence associated to the Morse-Bott degeneration which brings $\{i,i+1\}$ together.  Note that all these spheres are oriented.

We say a matching \emph{contains an odd cup} if it contains the arc joining $2j+1, 2j+2$ for some $0 \leq j \leq n-1$, and \emph{contains an even cup} if it contains the arc joining $2j, 2j+1$ for some $1 \leq j\leq n-1$.  Every matching contains at least one cup. The matching $\wpb$ is singled out by containing all $n$ odd cups, and $\mix$ is singled out by containing all $(n-1)$ even cups.

The basis is constructed inductively, so we assume that we already have bases for the algebra $\scrH_{n-1}^{symp}$, with the properties that (i) they are compatible with all cup-functors $\cup_j: \scrH_{n-2}^{symp} \rightarrow \scrH_{n-1}^{symp}$, where compatible means that any basis element is taken either to zero or to a basis element by any given $\cup_j$, and (ii) the algebra is isomorphic in the given bases to the arc algebra $H_{n-1}$, by an isomorphism entwining $\cup_j$ and $\cup_j^{comb}$.  The induction is based by Proposition \ref{Prop:n=2}, so we may assume throughout that $n\geq 3$. 

\begin{enumerate}
\item Fix bases of all groups $H^*(L_{\wp})$ as in Corollary \ref{Cor:Same}.
\item Suppose $\wp$ contains an odd cup $\cup_{2j+1}$. Then both $\wp$ and $\wpb$ lie in the image of some $\cup_{odd}$, say $\wp = \cup_{2j+1}(\wp^r)$, $\wpb = \cup_{2j+1} (\wpb^r)$ ($r$ for reduced), and hence the K\"unneth theorem gives  canonical isomorphisms
\begin{align*}
HF^*(L_{\wpb}, L_{\wp}) & =  HF^*(L_{\wpb^r}, L_{\wp^r}) \otimes H^*(V_{2j+1}) \\
HF^*(L_{\wp},L_{\wpb}) & =  HF^*(L_{\wp^r}, L_{\wpb^r}) \otimes H^*(V_{2j+1}).
\end{align*}
By induction we have a basis for the first factor on the right, by the orientation convention we have a basis for the second, and we take the induced basis.
\item Suppose $\wp$ contains an even cup $\cup_{2j}$. Write $\wp = \cup_{2j}(\wp^r)$, $\mix = \cup_{2j} (\mix^r)$,  and use K\"unneth and induction to take the bases induced from 
  \begin{align*}
    HF^*(L_{\mix}, L_{\wp}) & = HF^*(L_{\mix^r}, L_{\wp^r}) \otimes H^*(V_{2j}) \\
HF^*(L_{\wp}, L_{\mix}) & = HF^*( L_{\wp^r}, L_{\mix^r}) \otimes H^*(V_{2j})
  \end{align*}
\item We fix a basis of $HF^*(L_{\wpb}, L_{\mix})$ as follows.  Since $n \geq 3$ there is at least one matching $\wp$ which contains both an odd and an even cup.  Pick some such; we then have bases for the pairs $(\wpb,\wp)$ and $(\mix,\wp)$ (by the previous steps), and the triangle $\{\wp,\wpb,\mix\}$ admits a non-trivial product by Lemma \ref{Lem:PlaitMixNonzeroProducts}.  We pick the generators $\alpha_{\wpb,\mix}$ and $\alpha_{\mix,\wpb}$ so as to make this a positive triangle. 
\item If $\wp$ contains \emph{no} odd cup it necessarily contains an even cup so bases for the pair $(\mix,\wp)$ are fixed in step (3). The triangle $(L_{\wp}, L_{\wpb}, L_{\mix})$ again admits a non-trivial product, and we pick the unique minimal generators for $(L_{\wp}, L_{\wpb})$ making this a positive triangle.
\item For any $\wp, \wp'$ we now have bases for the pairs $(L_{\wp}, L_{\wpb})$ and $(L_{\wp'}, L_{\wpb})$. Since any triangle involving $L_{\wpb}$ has non-zero products, we now choose the basis for $(L_{\wp}, L_{\wp'})$ to make the triangle $(L_{\wp}, L_{\wp'}, L_{\wpb})$ a positive triangle.
\end{enumerate}

At this stage, all groups in $\scrH_n^{symp}$ have bases.  The remaining task is threefold:  to show that the bases described above are well-defined (independent of the choices made along the way); to show that these bases are preserved by cup functors $\scrH_{n-1} \rightarrow \scrH_n$; and to show that with respect to these bases, products of positive generators are linear combinations of positive generators in a way that matches the product in $H_n$.

\subsection{Well-definition of the basis} We show independence of choices.

\begin{Lemma}
In Steps (2) or (3) above, the choice of odd respectively even cup in $\wp$ does not affect the resulting minimal degree generator.
\end{Lemma}

\begin{proof} Immediate from the K\"unneth theorem, and the fact that the basis for $\scrH_{n-1}^{symp}$ is compatible with all the cup functors from $\scrH_{n-2}^{symp}$.
\end{proof}

In step (4), we make an arbitrary choice of a matching which contains both an odd and an even cup.   When $n\geq 4$ one can interpolate between different choices, as follows. We are defining the basis for the pair $(L_{\wpb},L_{\mix})$ by choosing a matching $\wp = \cup_{2i+1}\cup_{2j}\wp'$ which contains both an odd and an even cup.  Note that necessarily $i$ and $j$ are not adjacent, i.e. the four end-points of the cups $\cup_{2j}$ and $\cup_{2i+1}$ are distinct, since both belong to some matching. Let $\wp_{2j}$ denote the matching which differs minimally from $\wpb$ whilst containing the cup joining the points $\{2j, 2j+1\}$, so it does contain one even cup.  One considers a configuration
\begin{equation}
  \label{eq:mix-plait-well-defined}
  \xymatrix{
&L_{\mix} \ar@{-}[ldd] \ar@{.}[d] \ar@{-}[rdd] & \\
& L_{\wpb} & \\
L_{\wp} \ar@{-}[ur] \ar@{-}[rr] & & L_{\wp_{2j}} \ar@{-}[ul]
}
\end{equation}
The matchings in the bottom triple $(\wpb, \wp_{2j}, \cup_{2i+1}\cup_{2j}\wp')$ all contain $\cup_{2j}$, so this is a positive triple by the K\"unneth theorem and the choices fixed at Step (2). Similarly the three matchings in the outer triple $(\mix, \wp_{2j}, \cup_{2i+1}\cup_{2j}\wp')$ all contain the odd cup $\cup_{2i+1}$, hence this is a positive triple by the K\"unneth theorem.  There is a non-trivial product involving all four matchings, since $(\wpb, \wp_{2j})$ form a codimension one pair by construction of $\wp_{2j}$, and there is a non-trivial product involving the three matchings in the top left triple by Lemma \ref{Lem:PlaitMixNonzeroProducts}.  Therefore, the two possible basis elements for $(\wpb, \mix)$ on the dotted arrow defined by either choosing $\wp$ or $\wp_{2j}$ co-incide by applying Lemma \ref{Lem:ConormalTripleProducts}.  Iterating, one can compare any two choices of $\wp$ in Step (4) consistently. 

We record an important consequence of the above construction:
\begin{lem}  \label{lem:mix-plait-positive}
The triple $ (\mix, \wp_{2j},\wpb) $ is positive. In particular, so is the pair $(\mix, \wpb) $.
\end{lem}
\begin{proof}
The triangle $ (\mix, \wp_{2j},\wpb) $ is positive by the previous discussion, whilst the pairs $(\mix, \wp_{2j}) $ and $ (\wpb, \wp_{2j}) $ are positive by   the K\"unneth theorem. The result then follows from Corollary \ref{cor:positivity_triple_codim_1}.
\end{proof}
%There is a similar but simpler consistency check at Step (5), since one could write $\wp = \cup_{2j} \wp^r$ for different $j$. The compatibility of the construction with $\cup_{2i}$-functors at the previous stage of the induction shows that no ambiguity is introduced here.   

There is no further choice at Steps (5) and (6), so the bases are consistently determined.

\subsection{The case $n=3$}

We now prove independence of choices when $n=3$.  There are exactly two matchings containing an odd and an even cup:  the plait matching joins pairs $\{(12),(34),(56)\}$, the mixed matching joins pairs $\{(16),(23),(45)\}$ and the two possible matchings containing both an odd and even cup are
\[
\wp_{25} = \{(14),(23),(56)\} \qquad \textrm{and} \qquad \wp_{14} = \{(12),(36),(45)\}
\]
(the subscripts indicate which cups the matchings contain). Consider the triangle with vertices $\mix, \wp_{25}, \wp_{14}$ and with an interior vertex labelled $\wpb$, cf. Figure \ref{Fig:4Triangles}. The previous steps of the inductive construction have fixed bases for the pairs including $\mix$ or $\wpb$ and one of the other two vertices, at Steps (2) and (3).  Moreover, the definition of the basis for the pair labelled $(\wp_{14}, \wp_{25})$ at step (6) ensures that the triangle $(\wpb,\wp_{14},\wp_{25})$ is positive (this is independent of the choice in Step (4)). The basis for the pair $(\wpb,\mix)$ can be fixed by choosing either of the other two vertices, and making the corresponding triangle positive.  That leaves two further triangles: the outer triple, and the remaining internal triangle. These have the same sign by Lemma \ref{Lem:ConormalTripleProducts} (since Lemmas \ref{Lem:EverythingFromMinimalDegree} and \ref{Lem:PlaitMixNonzeroProducts} together imply that there is a non-trivial product involving all four Lagrangians) but that sign is not determined by consistency with choices already made, so we need to compute it directly.

\begin{figure}[ht]
\[
\xymatrix{
&L_{\mix} \ar[ldd]_{\gamma} \ar@{.>}[d] \ar[rdd]^{\alpha} & \\
& L_{\wpb} & \\
L_{\wp_{25}} \ar[ur]_{\delta} \ar@{-}[rr] & & L_{\wp_{14}} \ar[ul]^{\beta}
}
\]
\caption{Checking positivity when $n=3$ \label{Fig:4Triangles}}
\end{figure}
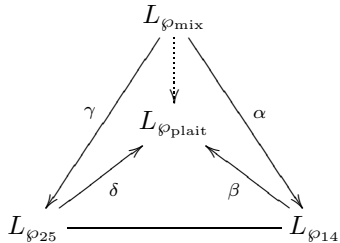

\begin{Lemma}
Suppose $n=3$. The two minimal degree generators for the pair $(\wpb,\mix)$ obtained from Step (4) of the inductive strategy respectively using $\wp_{14}$ or $\wp_{25}$ agree.
\end{Lemma}

\begin{figure}[ht]
\includegraphics[scale=0.25]{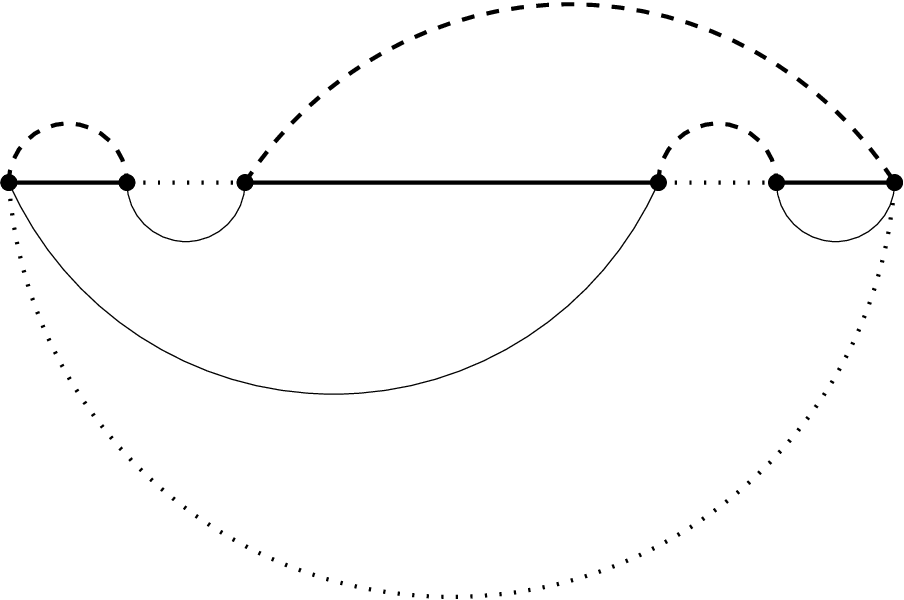}
\caption{Four matchings when $n=3$\label{Fig:4matchings}}
\end{figure}

\begin{proof}
We draw the four matchings in question in the diagram of Figure \ref{Fig:4matchings}, where we have used the Markov I move to slide certain arcs into the lower half-plane to remove excess intersections. The Markov I move preserves the orientation of each arc, hence preserves the bases of $H_2(L_{\wp})$ and the associated monomial bases of $H^*(L_{\wp})$, so this Hamiltonian isotopy introduces no signs into the computation of Floer products.  

We have drawn the six critical points grouped into two triples, and consider the degeneration which simultaneously collapses these triples, i.e. we work in an open subset of $\scrY_3$ which is an $(A_2\times A_2)$-fibration over $T^*S^2$. In this local model, the four matchings are all fibred over the zero-section, with fibres being products of the two basic real arcs in $A_2$, one in each factor of $A_2 \times A_2$.  The key claim is then that the two products corresponding to $\beta\circ\alpha$ and $\delta\circ\gamma$ in Figure \ref{Fig:4Triangles} agree.  We see this by explicit computation in the $(A_2\times A_2)$-fibred plumbing model. Schematically, the arrows $\alpha,\beta,\gamma,\delta$ are as follows: \newline
\[
\xymatrix{
\alpha & \bullet \ar@{.}[r] & \bullet \ar@{.}[r] & \bullet \ar@/_1pc/[ll] 
&  & \bullet \ar@{-}[r] & \bullet & \bullet \\
 \beta & \bullet \ar@{-}[r] & \bullet & \bullet & & \bullet \ar@{.}[r] \ar@/^1pc/[rr]  & \bullet \ar@{.}[r] & \bullet \\
  \gamma & \bullet & \bullet  \ar@{-}[r] & \bullet & & \bullet \ar@{.}[r] \ar@/^1pc/[rr]  & \bullet \ar@{.}[r] & \bullet \\
   \delta & \bullet \ar@{.}[r] & \bullet \ar@{.}[r] & \bullet \ar@/^1pc/[ll] & & \bullet  & \bullet \ar@{-}[r] & \bullet 
}
\newline
\]

The six critical points are grouped in the triples which define the local $(A_2\times A_2$)-fibration. A  solid line indicates that the morphism is between two Lagrangians which share that arc in Figure \ref{Fig:4matchings}, and represents the fundamental class of the corresponding $S^2$-factor of the $A_2$-fibre; an arrow between a pair of dotted lines denotes the morphism given by the transverse intersection point of the corresponding matching spheres in the $A_2$-fibre.  Since the fundamental class is a cohomological unit,  $\beta\circ\alpha$ and $\delta\circ\gamma$ define the tensor product of the curved arrows in respectively the first two and last two rows, hence represent the same Floer cycle.
\end{proof}

This shows that there is no ambiguity in the construction of the bases when $n=3$. By construction, all triangles among the quadruple $(\wpb,  \wp_{14}, \wp_{25}, \mix)  $ are positive. In each case, two pairs are also positive by the K\"unneth theorem, so Corollary \ref{cor:positivity_triple_codim_1} implies that all four triples are positive.

There is only one other matching when $n=3$, namely the ``horseshoe matching'' $\wp_{\circ}$ (joining pairs $\{(16),(25),(34)\}$) which enters into the definition of $\Kh_{symp}$ . Consider the diagram
\[
\xymatrix{
&L_{\mix} \ar@{.}[ldd] \ar@{-}[d] \ar@{-}[rdd] & \\
& L_{\wpb} & \\
L_{\wp_{\circ}} \ar@{-}[ur] \ar@{.}[rr] & & L_{\wp_{14}} \ar@{-}[ul]
}
\]
where the solid lines denote pairs for which Floer groups were chosen either in Step (2) or Step (4) of the inductive construction, and the dotted lines are the pairs for which a choice is made only thereafter.  The choice of bases up to Step (4) ensures that  all solid lines connect positive pairs, and  the top right (solid) triple is positive.  The two dotted arrows are then chosen to make the remaining two internal triangles positive.  Note that the bottom triangle is in fact a positive triple because $(\wpb, \wp_{14}) $ is a codimension $1$ pair. There is again a non-zero product  involving all four Lagrangians since two meet in codimension one, so Lemma \ref{Lem:ConormalTripleProducts} implies that all triangles are positive.  Noting that the top left triangle contains a codimension $1$ pair, we conclude that the top left triple is also positive by Corollary \ref{cor:positivity_triple_codim_1}. We conclude that all triples are positive.

Finally, by symmetry between $L_{\wp_{14}}$ and $L_{\wp_{25}}$ one easily checks that all products are positive when $n=3$.  An examination of the  argument furthermore shows that the bases for $\scrH_3^{symp}$ are indeed compatible with the cup functors $\cup_i: \scrH_2^{symp} \rightarrow \scrH_3^{symp}$; compare to the formally identical Lemma \ref{Lem:CupsOK} proved below.

%\subsection{Positivity for pairs}
%\label{sec:positivity-pairs}

\subsection{Cup-compatibility and positivity}
We now return to the task of proving positivity and compatibility with the cup functors for $n \geq 4$.

In steps (2)-(3) of the construction of bases, new generators are constructed from old ones using the K\"unneth theorem, and the positivity of the relevant pairs is inherited. The morphism constructed in step (4) was shown to be positive in Lemma \ref{lem:mix-plait-positive}.

\begin{lem} \label{lem:other_basis_also_works}
For the choice of basis in step (5), $(L_{\wpb}, L_{\wp})$ is a positive pair. 
\end{lem}
\begin{proof}
Fix an integer $j$ so that $L_{\wp}$ lies in the image of $\cup_{2j}$. Given the choices of bases on the pairs $ (L_{\wpb}, L_{\wp_{2j}}) $ and $ (L_{\wp}, L_{\wp_{2j}}) $,  which are induced by the K\"unneth theorem, and noting that the pair $ (L_{\wpb}, L_{\wp_{2j}})  $ meet in codimension $1$,  Corollary \ref{cor:positivity_triple_codim_1} implies that there is a unique choice of bases for the pair $(L_{\wpb}, L_{\wp})$ so that the triple $(L_{\wp_{2j}}, L_{\wpb}, L_{\wp})$ is positive. To show that the choice of basis fixed in step (5) agrees with this new one, it suffices to show that the triangle $ (L_{\mix}, L_{\wpb}, L_{\wp}) $ is positive for the new choice of basis for the pair  $(L_{\wpb}, L_{\wp})$. This is immediate from Lemma \ref{Lem:ConormalTripleProducts}, applied to the quadruple $(L_{\wp_{2j}}, L_{\wpb}, L_{\wp}, L_{\mix})$:
\begin{equation}
  \xymatrix{
&L_{\mix} \ar@{-}[ldd] \ar@{-}[d] \ar@{-}[rdd] & \\
& L_{\wpb} & \\
L_{\wp} \ar@{.}[ur] \ar@{-}[rr] & & L_{\wp_{2j}} \ar@{-}[ul]
}
\end{equation}
Lemma \ref{lem:mix-plait-positive} implies that the triple $(L_{\wp_{2j}}, L_{\wpb}, L_{\mix})$ is positive. On the other hand, the triple $(L_{\wp_{2j}}, L_{\wpb}, L_{\wp})  $ is positive by construction, while  $(L_{\wp_{2j}}, L_{\wp}, L_{\mix})$ is positive by K\"unneth.
\end{proof}

% The first observation is that all triples including $L_{\wpb}$ and a codimension $1$ pair are positive; this follows from positivity of the corresponding triangle and Corollary \ref{cor:positivity_triple_codim_1}.
\begin{lem} \label{lem:triple_plait_positive}
All triples containing $\wpb$ are positive. In particular, all pairs are positive.
\end{lem}
\begin{proof}
Recall that matchings $\wp$ and $\wp''$ meet in codimension $n-c(\wp,\wp'')$, where $c(\wp,\wp'')$ is the number of components of the planar unlink $\wp\cup \overline{\wp''}$. Assume, by decreasing induction on $c(\wp,\wpb) $ , that the triple $(\wp, \wp', \wpb) $ is positive whenever $c(\wp,\wpb) \geq  c(\wp',\wpb) $. The base case $\wp = \wpb$ is a reformulation of the positivity of pairs containing $\wpb$.

In  the inductive step, we consider a Lagrangian $L_{\wp''}$ so that $ c(\wp, \wp'') = n-1 $, and $c(\wp, \wpb) = c(\wp'', \wpb) -1$. By the inductive hypothesis, in the quadruple $(L_{\wpb}, L_{\wp}, L_{\wp'}, L_{\wp''})$, the triples containing $(L_{\wpb}, L_{\wp''})$ are both positive. The codimension $1$ condition for $\wp$ and $\wp''$ implies, via Lemma \ref{lem:OneArcChanged} and Lemma \ref{Lem:NonzeroProducts}, that either
\begin{equation}
  \alpha_{\wp, \wp'} \cdot \alpha_{\wp'', \wp} = \pm \alpha_{\wp'', \wp'} \textrm{ or }    \alpha_{\wp'', \wp} \cdot \alpha_{\wp', \wp''} = \pm \alpha_{\wp', \wp}.
\end{equation}
The two cases are similar; we consider the first. Reversing the roles of $\wp$ and $\wp''$ in Lemma \ref{Lem:NonzeroProducts}, we see that multiplication by $\alpha_{\wp,\wp''}$ is also necessarily non-trivial, yielding
\begin{equation}
    \alpha_{\wp, \wp''} \cdot \alpha_{\wp', \wp} = \pm (v_1+v_2) \alpha_{\wp', \wp''}.
\end{equation}
We therefore obtain a non-zero product among the Lagrangians $  (L_{\wpb}, L_{\wp'}, L_{\wp}, L_{\wp''})$ in forward and backward ordering. Applying Lemma   \ref{Lem:ConormalTripleProducts}, we conclude that the triangle  $   (L_{\wp'}, L_{\wp}, L_{\wp''}) $ is positive. The inductive hypothesis yields that the pairs containing $L_{\wp''} $ are positive. Since $  ( L_{\wp}, L_{\wp''})  $ is a codimension $1$ pair, we conclude that the triple $   (L_{\wp'}, L_{\wp}, L_{\wp''}) $ is positive, hence the pair $ (L_{\wp'}, L_{\wp}) $ is positive, which completes the proof of the inductive step. 
% $ (L_{\wp'}, L_{\wp}, L_{\wp''}) $ is positive with respect to either cyclic ordering. The first part of Lemma \ref{Lem:ConormalTripleProducts} then implies that it is a positive triangle, and in particular that the pair $ (L_{\wp}, L_{\wp'}) $ is positive, which completes the inductive step.
\end{proof}

\begin{Lemma} \label{Lem:ProductsOK}
For the bases constructed by the induction of Section \ref{Sec:Induction}, all Floer products are positive in the sense of Definition \ref{Defn:PositiveBasis}.
\end{Lemma}
\begin{proof}
The statement is equivalent to all triples being positive, and Lemma \ref{Lem:EverythingFromMinimalDegree} implies it suffices to know positivity for triples in which one pair is of codimension $1$. Given a codimension one pair $(\wp,\wp')$ and  another matching $\wp_{\dagger}$, consider:
\[
\xymatrix{
&L_{\wp} \ar@{.}[ldd] \ar@{-}[d] \ar@{-}[rdd] & \\
& L_{\wp'} & \\
L_{\wp_{\dagger}} \ar@{.}[ur] \ar@{-}[rr] & & L_{\wpb} \ar@{-}[ul]
}
\]
By Lemma \ref{lem:triple_plait_positive}, the triples involving $\wpb$ are positive. Lemma \ref{Lem:ConormalTripleProducts} then says that the remaining triple is positive.
\end{proof}

\begin{Lemma}  \label{Lem:CupsOK}
The bases constructed by the inductive procedure of Section \ref{Sec:Induction} are compatible with all cup functors $\cup_i: \scrH_{n-1}^{symp} \rightarrow \scrH_n^{symp}$.
\end{Lemma}

\begin{proof}
Compatibility with the odd cup functors is immediate from Step (2) of the construction, and the fact that at Step (6) arbitrary bases are defined by comparison with $\wpb$ which contains all the odd cups.  (Compatibility with the even cup functors requires a consistency check, since we have broken symmetry between $\wpb$ and $\mix$ in Step (6) of the induction.)  So suppose we have two matchings which are both in the image of an even cup functor, $\cup_{2i}\wp$ and $\cup_{2i}\wp'$. The argument then follows the same strategy as in Lemma \ref{lem:other_basis_also_works}. We consider the configuration
\[
\xymatrix{
& L_{\cup_{2i}\wp} \ar@{.}[d] \\
L_{\mix} \ar[ur] \ar[r] & L_{\cup_{2i}\wp'}
}
\]
which admits a non-trivial product involving all 3 Lagrangians, by Lemma \ref{Lem:PlaitMixNonzeroProducts}. 
The K\"unneth theorem in Floer cohomology, and the fact that all products are positive by Lemma \ref{Lem:ProductsOK}, implies that \emph{if we use K\"unneth bases} for the three groups in the triple, all of which lie in the image of $\cup_{2i}$, then the triple is positive. Lemma \ref{Lem:ProductsOK} implies that the triple is also positive with respect to the basis we constructed in the induction, hence the two bases agree.
% But our \emph{definition} of the basis for the dotted arrow in the figure, in Step (3) of the algorithm, \emph{is} the unique basis which makes this triple positive.  In other words, by construction the algorithm basis is the K\"unneth basis, which is equivalent to compatibility with $\cup_{2i}$.
\end{proof}

The combinatorial arc algebra underlying Khovanov homology is based on a 2d TQFT with underlying Frobenius algebra $V=\bZ\langle 1,x\rangle$ with the property that the co-product 
\[
V \rightarrow V\otimes V, \qquad 1 \mapsto 1\otimes x + x\otimes 1
\]
gives a positive linear combination of elements of the tensor product basis: the two terms in the final expression have the same sign (rather than $1\otimes x - x \otimes 1$). The corresponding positivity in our setting was contained in Lemma \ref{Lem:NonzeroProducts}.

\begin{Corollary} \label{Cor:WeWin} The positive basis for $\scrH_n^{symp}$ constructed previously defines an isomorphism $\scrH_n^{symp} \rightarrow H_n$ which entwines the $\cup_i$ and $\cup_i^{comb}$.
\end{Corollary}

\begin{proof} A helpful model of the combinatorial arc algebra for our purposes is that given by Stroppel and Webster \cite{SW}: they set 
\[
H_n = \oplus_{\wp,\wp'} H^*(\hat{L}_{\wp} \cap \hat{L}_{\wp'})
\]
and define a modified convolution product on these groups, which (they prove) co-incides with Khovanov's original TQFT product.   There are natural maps from $H^*(\scrY_n)$ to the centre of both $H_n$ and $\scrH_n^{symp}$. The clean intersection models for pairs of Lagrangians give identifications of $HF^*(L_{\wp},L_{\wp'})$ and $H^*(\hat{L}_{\wp} \cap \hat{L}_{\wp'})$ as modules over $H^*(\scrY_n)$. To see that these identifications yield an algebra isomorphism, it is sufficient to know that one can choose minimal degree generators for all groups such that all products between minimal degree generators agree.  Furthermore, we know from Lemma \ref{Lem:EverythingFromMinimalDegree} and the Stroppel-Webster algorithm that it suffices to check products of a minimal degree generator with a degree one generator arising from a pair $(\wp,\wp')$ which meet in codimension one.  

The algebra isomorphism now follows from comparing Lemma \ref{Lem:NonzeroProducts} with \cite[Section 4.2]{SW}. In particular, one should compare the ``push-forward" case on p. 503 of that paper, in which a saddle cobordism leads to multiplication by an element $\pm (z_{j,i+1} + z_{\sigma(j),i+1})$ which is a sum of two basis elements in $H^2(\scrY_n)$, with the corresponding situation in Lemma \ref{Lem:NonzeroProducts}.  The essential point is that the two $H^2$-classes which appear have the \emph{same} sign, rather than opposite signs, and in both cases represent the Poincar\'e dual in cohomology to the appropriate triple intersection submanifold.  (\cite{SW} gives an explicit sign recipe for ensuring that all the terms $\pm (z_{j,i+1} + z_{\sigma(j),i+1})$ appear with a sign $+$; our construction of a positive basis gives an implicit proof that on the symplectic side one can also ensure that only $+$ signs appear.) Compatibility of the isomorphism with cup functors follows from Lemma \ref{Lem:CupsOK}.
\end{proof}

\begin{Remark}\label{rem:modules_agree}
The isomorphism of Corollary \ref{Cor:WeWin} identifies the projective module over $\scrH_n^{symp}$ defined by the Lagrangian submanifolds $L_{\wp}$ with the projective module over $H_n$ defined by the idempotent  associated to $\wp$. In particular, the Lagrangian $L_{\wp_{\circ}}$ which enters into the definition of symplectic Khovanov cohomology is identified with the corresponding module $P_{\wp_{\circ}}$ over $H_n$. 
\end{Remark}

%%%%%%%%%%%%%%
%%%%%%%%%%%%%%
%%%%%%%%%%%%%%%
%%%%%%%%%%%%%%%
%%%%%%%%%%%%%%%

%\begin{cor} \label{cor:unit_is_equivariant}
%If $\scrF $ is pure then   $_\scrF\scrB_\scrF  $ admits a pure equivariant structure such that the natural map (unit)
%\begin{equation}
%\scrA \to  _\scrF\scrB_\scrF
%\end{equation}
%of $\scrA-\scrA$-bimodules is equivariant.
%\end{cor}
%\begin{proof}
%Lemma \ref{lem:purity_implies_formal_functor} implies that $\scrF$ is formal, hence $_\scrF\scrB_\scrF  $ is formal as an $\scrA-\scrA$-bimodule so it is pure and equivariant. Since the map $H \scrA \to  _{H \scrF} H\scrB_{H\scrF}  $ is equivariant, the result follows.
%\end{proof}

\section{The isomorphism}
We now combine the previous results with the exact sequence formulated and proved in Appendix \ref{Sec:LES_of_twist} to establish the isomorphism between Khovanov and symplectic Khovanov cohomology. 

\subsection{Parallel transport} \label{sec:parallel-transport}

We begin with a technical discussion of the choice of symplectic structure on the Slodowy slice, to ensure that the results of Appendix \ref{Sec:LES_of_twist} can be applied to our setting. 

We start in the following more general situation.  Suppose we have a Stein manifold $\scrX$ with a holomorphic map $\chi: \scrX \to \scrB$ to a Stein manifold $\scrB$, which is non-singular over the complement of a discriminant divisor $\Delta \subset \scrB$.  Suppose moreover we have 
 a smooth function $\psi: \scrX \to \bR$ with the properties that 
\begin{itemize}
\item the form $\omega = -dd^c\psi$ defines a K\"ahler form on $\scrX$;
\item  the set of critical values $\mathrm{Crit}(\chi)$ is proper over $\scrB\backslash \Delta$; 
\item outside a compact set, $\| \nabla \psi \|^2 \leq C \psi$ for some constant $C>0$.
\end{itemize}
The final equality is with respect to the K\"ahler metric associated to $\omega$. 
Denote by $\scrX_t$ the fibre of $\chi$ over a point $t\in \scrB$. The form $\omega$ defines an exact K\"ahler form on  $\scrX_t$, provided $t \in \scrB\backslash \Delta$ is generic so the fibre is smooth. In that case, the Liouville vector field $Z_t$ of $\psi_t = \psi|_{\scrX_t}$ satisfies 
\[
Z_{t} \cdot \psi_{t} \leq C \psi_{t}
\]
and hence has globally defined flow, so $(\scrX_t, \omega)$ is convex and has an infinite contact conical end.  Fix a compact subset $B \subset \scrB\backslash \Delta$ containing the point $t$.   In this situation, \cite{SS} explained how to relate distinct fibres of $\chi$ by ``rescaled symplectic parallel transport maps" which were defined on arbitrarily large compact subsets of the fibres.  Here we expand on \cite[Remark 30]{SS}, defining symplectic parallel transport globally for paths which stay away from the discriminant locus;  analogous arguments appear in \cite[Section 6]{khovanov-seidel} and \cite[Section 9]{Keating:freetwists}, which we follow closely.

\begin{Proposition} \label{prop:good_Kaehler_general}
There is a K\"ahler form $\omega'$ on $\scrX_t$, and an extension of this to a vertically non-degenerate closed 2-form on $\scrX$, for which  parallel transport maps are globally defined along paths in $B$. 
\end{Proposition}

\begin{proof}
Because of the properness of the critical values of $\psi$, we can find $c>0$ for which the truncated fibres $\scrX_b \cap (\psi^{-1}[0,c])$ form a family of smooth Stein domains as in \cite[Lemma 47]{SS}, and Gray's theorem implies that the contact boundaries vary locally trivially. Following \cite[Section 6]{khovanov-seidel}, we trivialise a collar neighbourhood of the horizontal boundary $\partial_{hor} \scrX$ of the fibration $\scrX|_B \cap \psi^{-1}[0,c]$ by the flow of the fibrewise Liouville field $Z_b$, defining a diffeomorphism onto its image
\begin{equation} \label{Eqn:Collar}
\partial_{hor} \scrX \times [-\epsilon, 0] \longrightarrow W \subset \scrX|_B \cap \psi^{-1}[0,c], \qquad (v,t) \mapsto \phi_{Z_{\mathbf{\chi(v)}}}^{-t}(v).
\end{equation}
If $\alpha = -d^c\psi|_{\partial_{hor} \scrX}$ is the contact 1-form on $\partial_{hor} \scrX$, we obtain a 1-form $\alpha' = e^t(\phi_Z)^*\alpha \in \Omega^1(W)$ with the property that $d\alpha'$ is fibrewise symplectic on $W$, and its associated horizontal subspaces are $\phi_{Z_t}$-equivariant in a collar neighbourhood of $\partial_{hor} \scrX$. 

Fix a cut-off function $\eta: [-\epsilon,0] \rightarrow \bR$ which equals $0$ on $[-\epsilon, -3\epsilon/4]$ and equals $1$ on $[-\epsilon/4,0]$.  Let $\widehat{\eta}: W \rightarrow \bR$ denote the function obtained from $\eta$ via \eqref{Eqn:Collar}. The form $\omega' = \omega + d(\widehat{\eta}\cdot (\alpha'+d^c\psi))$ is a closed vertically non-degenerate 2-form with the following properties:
\begin{itemize}
\item in $\scrX|_B \cap \psi^{-1}[0,c]$ it agrees with $\omega$ on the complement of the neighbourhood $W$ of the horizontal boundary;
\item its restriction to any fibre $\scrX_b \cap \psi^{-1}[0,c]$  agrees with the restriction of $\omega$, for $b \in B$;
\item the $\omega'$-parallel transport maps over $B$ commute with the Liouville flow in $W$, so are cones on contactomorphisms. 
\end{itemize}
The final condition implies that the $\omega'$-parallel transport maps have globally defined flows on the symplectic completions of the fibres (which was not obvious for the $\omega$-parallel transport maps). 
\end{proof}

We now return to the particular situation of interest in this paper.
Recall the slice $\Slice$ from \eqref{eq:y-matrix} and the adjoint quotient map \eqref{Eqn:AdjointQuotient}, which defines a holomorphic fibration $\Slice \rightarrow \Sym_{2n}(\bC)$. Let $\mathbf{w} = \{1,2,\ldots,2n\}$ and $\scrY_n = \scrY_n^{\mathbf{w}}$ denote the corresponding fibre of $\chi|_{\Slice}$. Fix a compact set $B\subset \Sym_{2n}(\bC) \backslash \Delta = \Conf_{2n}(\bC)$ which lies away from the discriminant locus of non-simple configurations (corresponding to repeated eigenvalues).

 In \cite{SS} a K\"ahler form on $\Slice$, and hence by restriction on its general fibre $\scrY_n$, is defined as follows.  For each co-ordinate $z$ of the matrix $A_i$, we take the function $\xi_i(z) = |z|^{4n/i}$.  The sum of these define a proper $C^1$-smooth function $\xi: \Slice \rightarrow \bR$, which can be smoothed near the co-ordinate hyperplanes by replacing $\xi_i \mapsto \xi_i+\eta_i$ for suitable compactly supported functions $\eta_i$.   Let $\psi$ denote the function obtained by adding up all the $\xi_i+\eta_i$ over all entries of $A\in\Slice$. It was proven in \cite[Section 5.1]{SS} that $\psi$ satisfies the conditions listed before Proposition \ref{prop:good_Kaehler_general}.
 
 For each $1 \leq i \leq 2n-1$ we consider a holomorphic map $w_i : \bC \rightarrow \Conf_{2n}(\bC)$ with $w_i(0) = \{1,2,\ldots, i-1, i+1/2, i+1/2, i+2, \ldots, 2n\}$ and with image im$(w_i) \pitchfork \Delta$ transverse at $0$ to the discriminant in the symmetric product $\Sym_{2n}(\bC)$ and otherwise disjoint from the discriminant.  Let $\mathbf{w_i} = \{1,2,\ldots, i-1, i+1,\ldots, 2n\}$.  Choosing $B$ to be sufficiently large, we pull back $\chi|_{\Slice}$ by $w_i$ to obtain a holomorphic map 
\[
W_i: E^{B}_i \rightarrow D^{2}
\]
 which is Morse-Bott-Lefschetz at $0$, with zero-fibre having singular locus canonically isomorphic to the space $\scrY_{n-1}^{\mathbf{w_i}}$, and with generic fibre $\scrY_n$. One can apply Proposition 
 \ref{prop:good_Kaehler_general} in this setting, since the required hypotheses on the K\"ahler form are preserved under pullback by $w_i$.  After a further isotopy of K\"ahler forms, it follows that parallel transport is well-defined for $W_i$, in particular the horizontal boundary $W_i^{-1}(\partial D^2)$  is exhibited as the mapping torus of a symplectomorphism.  There is another K\"ahler form in a neighbourhood of $W_i^{-1}(\partial D^2)$  which is a flat symplectic fibration in a neighbourhood of the boundary of the disc, obtained by parallel transport in radial directions from the given K\"ahler form on the boundary. The two symplectic forms agree on the boundary, and can be interpolated through forms with well-defined parallel transport. We write
 \begin{equation}
 W_i: E_i \rightarrow \bC
 \end{equation}
for an extension of the previous map $W_i$ to a fibration over $\bC$, which, outside the unit disc, is modelled after the symplectic mapping torus of the monodromy. By construction $W_i$ is an \emph{exact LG model} in the sense of the Appendix, whose restriction to a neighbourhood of the critical point agrees with the map constructed  in \cite{SS}.

\begin{Proposition} \label{prop:good_Kaehler}
 There is a K\"ahler form $\omega'$ on $\scrY_n$, and an extension of this to a vertically non-degenerate closed 2-form on $\Slice$, for which  parallel transport maps are globally defined along paths in $B$.  Writing $\Symp_\infty(\scrY_n)$ for the group of symplectomorphisms of $(\scrY_n, \omega')$ which are modelled on contactomorphisms at infinity, there is a global monodromy representation
\[
Br_{2n} \longrightarrow \pi_0\Symp_{\infty}(\scrY_n).
\]
In particular, the monodromy
$\tau_i: \scrY_n \longrightarrow \scrY_n $ of $W_i$ is well-defined in $\pi_0\Symp_{\infty}(\scrY_n)$.
\end{Proposition}

\begin{proof} The first statement is the conclusion of Proposition \ref{prop:good_Kaehler_general}. For the final statement, the global monodromy representation
\[
Br_{2n} \longrightarrow \pi_0\Symp_{\infty}(\scrY_n)
\]
 is obtained from the previous construction by taking $B$ to be the complement of a relatively compact open neighbourhood of the discriminant, so $B$ is a deformation retract of  configuration space.   Similarly, parallel transport maps exist over compact subsets disjoint from the critical values of the Morse-Bott Lefschetz fibrations obtained by restricting $\chi|_{\Slice}$ to a disc transverse to the discriminant in $\Sym_{2n}(\bC)$.  It follows that parallel transport maps are globally defined for $W_i$ along paths which do not go into the origin, which completes the proof. \end{proof}

 \begin{Remark}
 There are at least three natural symplectic structures on $\scrY_n$: 
 \begin{enumerate}
 \item the one just discussed, from \cite{SS}, which has good parallel transport properties helpful for establishing the long exact triangle;  
 \item the one discussed in Section \ref{Sec:Flags}, which extends smoothly to the compactification $(S^2)^{2n}$ and is convenient for comparing the Lagrangians $L_{\wp}$ to the anti-diagonals $\hat{L}_{\wp}$ in the topological model and computing the ring structure in the symplectic arc algebra; 
 \item and the forms from Section \ref{Sec:Hilb}  induced from $\scrY_n \hookrightarrow \Hilb^{[n]}(A_{2n-1})$ which are product-like away from the diagonal, with respect to which the Lagrangians $L_{\wp}$ are products of fibred Lagrangians in Milnor fibres, and which we used in constructing the equivariant structure on the Fukaya category in \cite{AbSm} and on the cup bimodules in this paper.  
 \end{enumerate}
All three of these symplectic structures are in fact Stein structures for a fixed complex structure $J=J_n^t$ on $\scrY_n=\scrY_n^t$ (the complex isomorphism type depends on the eigenvalues, i.e. the point $t$ in configuration space).  They are therefore related by linear interpolation through exact K\"ahler forms. Non-compactness of $\scrY_n$  means that those deformations need not integrate to global symplectomorphisms, but they do yield symplectic embeddings of compact subdomains. In particular,   if $M \subset \scrY_n$ is a compact subdomain with $J$-convex boundary (a notion depending only on the complex structure $J$), the Liouville completion of $M$ is independent up to exact symplectomorphism of  which of the above Stein structures one chooses, cf. \cite{Cieliebak-Eliashberg}.  It follows that Floer theory for closed Lagrangians in $M$ is also independent of that choice.
 
To compare the Lagrangians themselves, note that Manolescu showed \cite{Manolescu} that one can embed a large compact subset of the Hilbert scheme $\Hilb^{[n]}(A_{2n-1})$ symplectically into $(\scrY_n, -dd^c\psi)$ in such a way that the Lagrangians obtained as products of matching spheres are Hamiltonian isotopic to the iterated vanishing cycles. One can therefore identify Floer theory for vanishing cycles in $\scrY_n$, equipped with the symplectic structure from \cite{SS}, with Floer theory for products of matching spheres in the Hilbert scheme, as used to study formality in \cite{AbSm}.

The K\"ahler form from section \ref{Sec:Flags} was only used to simplify the study of cohomological properties of the restriction $H^*(\scrY_n) \rightarrow H^*(L_{\wp})$, by reduction to the smooth models $\hat{L}_{\wp}$; those cohomological properties then carry over to any other model, even without identifying the corresponding Lagrangians up to Hamiltonian isotopy.   We may therefore appeal to the formality results for cup and cap functors, the long exact triangle, and the vanishing cycle description of Lagrangians (and the corresponding clean intersection models arising from Lemma \ref{Lem:CleanPairs}) simultaneously.  We do so without further comment.
 \end{Remark}

\subsection{The cup and cap functors revisited}
\label{sec:cup-functor}

Let $\Tw \scrH_{n-1}$ denote the category of twisted complexes over $\scrH_{n-1}$, and let $\Delta_{\tau^{-1}_i} $ denote the $\Tw \scrH_n$ bimodule associated to the inverse monodromy of  $W_i$ (i.e. the monodromy around a \emph{clockwise} oriented path encircling $0\in \bC$). Recall that we previously constructed a functor $\cup_i \co  \scrH_{n-1} \to \scrH_{n} $,  which we proved to be formal in Section \ref{sec:cup-bimodules-are}.

\begin{prop} \label{prop:inverse_twist_LES}
 There is a functor $ \cap_i \co \scrH_n \to \Tw \scrH_{n-1}$ which is left adjoint to $\cup_i$. Moreover, the cone of the unit
\begin{equation}
  \Delta^{\Tw \scrH_n} \to \Delta^{\Tw \scrH_n}_{\cup_i \circ \cap_i  }
\end{equation}
is quasi-isomorphic to the graph bimodule  $\Delta_{\tau^{-1}_i} $.
\end{prop}
\begin{proof}
Keeping in mind that $\tau^{-1}$ corresponds to a clockwise path around the critical point,  this is a direct application of Proposition \ref{prop:what_we_need_of_the_twist}, whose assumptions it  therefore suffices to verify.  In the notation of \emph{op. cit.} we have $\scrA = \scrH_{n-1}$, with $M=\scrY_{n-1}$, and $\scrA' = \scrH_n$ with $W = \scrY_n$;  we set $E=\bC^3 \rightarrow \bC$ to be the standard Lefschetz fibration, with fibre $N = T^*S^2$; and we obtain the inclusion $M \times E \rightarrow E' = E_i$, compatible with the map $W_i$, from the existence of $A_1$ degenerations as pairs of eigenvalues come together, as in \cite[Lemma 27]{SS}.  (Strictly speaking, $E$ and $N$ should be defined as subdomains of $\bC^3$ respectively $T^*S^2$ with contact boundary, and similarly for the fibration $E_{L'}$ introduced below, but we will keep to the simpler notation and hope no confusion will arise.)  To fulfil the hypotheses of Proposition \ref{prop:what_we_need_of_the_twist} we require:
\begin{enumerate}
\item For each object $L'$ of $\scrA'$, there exists a Lefschetz fibration $E_{L'} \to \bC$, with fibre $N_{L'}$,  including $E$ as a subfibration containing all critical points, and an inclusion $M \times E_{L'} \subset E'$ compatible with the maps to $\bC$ such that $L'$ is contained in  $M \times N_{L'}$;
\item  For each object $L'$  of $\scrA'$, there is an object $L$ of $\scrA$ such that either  (i) $L'$ is quasi-equivalent to the product of $L \times K$ or (ii) $L'$ is quasi-equivalent to a Lagrangian which meets $L \times K$ cleanly along a section of the projection to $K$. 
\end{enumerate}
Recall from Lemma \ref{Lem:PlaitMixNonzeroProducts} that if one places the eigenvalues $\mathbf{w}_{cyc} \in \Conf_{2n}(\bC)$ defining the fibre $\scrY_n^{\mathbf{w}_{cyc}}$ at the roots of unity, there is a cyclic symmetry which acts transitively on consecutive pairs, and hence relates the different Morse-Bott fibrations $W_i$.  Without loss of generality we therefore suppose $i=1$.

If $L'$ is in the image of $\cup_i$, say $L' = \cup_i L$, we simply take $E_{L'} = E$.  This corresponds to case (i) of Condition (2) above.

If $L'$ is not in the image of $\cup_i$, say $L' = L_{\wp}$, then since $i=1$, we can suppose that $\wp$ contains the arcs $(2,p)$ and $(1,m)$ for some $p < m$.  The hypothesis that $i=1$ additionally implies that the arc $(1,m) \subset \wp$ is necessarily outermost.  Embed $W_i$ into the larger two-parameter degeneration in which the triple of eigenvalues $\{1,2,p\}$ come together along the arcs $[1,2]\subset \bR$ and $(2,p) \subset \wp \subset \bC$, as in \cite[Lemma 29]{SS}, which is reproduced as Lemma \ref{lem:local_A2_degeneration} above.  Let $E_{L'}$  be $ T^*S^2\times\bC$, equipped with the Lefschetz fibration $E_{L'} \rightarrow \bC$ which has generic fibre $N_{L'}$ the $A_2$-Milnor fibre, and a unique critical value. One can view this as a subset, containing only one of the two critical values, of the stabilisation of the fibration $E\rightarrow \bC$, which gives a fibre-preserving inclusion $E \subset E_{L'}$.  Lemma \ref{lem:local_A2_degeneration} implies that there is an open subset of $E'$ which is an $A_2$-fibration over $\scrY_{n-1} \cong \mathrm{Crit}(W_i)$, and this fulfils the first set of required conditions.  

 From the construction of the $A_2$-degeneration,  $L_{\wp}$ is fibred over the matching given by removing the arc $(2,p)\subset \wp$, replacing the critical points $\{1,2,p\}$ by a single point $\ast$ viewed as lying at position $\{1\}$, and considering the matching $\wp_r$ comprising $\wp \backslash (2,p)$, which contains the arc $(\ast,m)$ (which is uniquely defined to isotopy, and still an upper half-plane matching, since $(\ast,m)$ is still outermost).  We set $L = L_{\wp_r} \in \scrA$.  Then $\cup_i(L_{\wp_r})$ and $L_{\wp}$ meet in codimension one, and are given locally by fibrations over a common base with fibre the two arcs of the compact core of the $A_2$-fibre $N_{L'}$.  This fulfils the requirements of case (ii) of Condition (2),  completing the proof.
 \end{proof}

%STILL NEED TO SAY MORE.

%By hypothesis, $L_{\wp}$ and hence its monodromy image $\tau_i(L_{\wp})$ are fibred in this model, since, after Hamiltonian isotopy, one can assume that the fibred Dehn twist monodromy acts fibrewise (with respect to the local $A_2$-fibration) on any fixed compact subset of $U$.  

%If we denote by $\gamma_{1,2}$ and $\gamma_{2,p}$ the core components of the $A_2$-fibre, then $L_{\wp}$ has fibres $\gamma_{2,p}$ and its monodromy image has fibres the curve $\gamma_{1,p}=\tau_{\gamma_{1,2}}(\gamma_{2,p})$.  The Morse-Bott surgery is given by smoothing the fibrewise transverse intersection, which yields a Lagrangian submanifold which is locally fibred by the vanishing cycle $\gamma_{1,2}$. From the construction of the $A_2$-degeneration,  $L_{\wp}$ is fibred over the matching given by removing the arc $(2,p)\subset \wp$, replacing the critical points $\{1,2,p\}$ by a single point $\ast$ viewed as lying at position $\{1\}$, and considering the (not necessarily half-plane) matching comprising $\wp \backslash (2,p)$, which contains the arc $(\ast,m)$ (which is uniquely defined to isotopy, still being outermost). It is now apparent that the base of the surgery is Hamiltonian isotopic in $\scrY_{n-1}$ to $L_{\cap_i \wp} $.  By Lemma \ref{lem:weakly_one_ended} the result follows.

If we consider  the monodromy twist $\tau_i$ instead of $\tau_i^{-1}$, functoriality of the construction of the graph bimodule implies that $ \Delta_{\tau_i }$  is the bimodule inverse (under tensor product) to $  \Delta_{\tau_i^{-1} } $. Since inverse bimodules are uniquely determined up to quasi-isomorphism, we conclude that this bimodule agrees with the diagramatic inverse twist:
\begin{cor} \label{cor:triangle_2}
The bimodule $\Delta_{\tau_i} $  is quasi-isomorphic to the cone of the counit:
\begin{equation}
\Delta_{\tau_i}\,  \simeq\, \mathrm{Cone}\left(\Delta^{\Tw \scrH_{n}} _{\cup_i \circ \cap_i} \to \Delta^{\Tw \scrH_n}\right).
\end{equation} \qed
\end{cor}
We therefore reach the main theorem:

\begin{thm} \label{thm:main}
 For any link $\kappa$, and characteristic zero field $\bk$, we have an isomorphism 
 \begin{equation}
 \Kh_{symp}(\kappa; \bk) \cong \Kh(\kappa; \bk).
 \end{equation}
\end{thm}
\begin{proof}
Khovanov homology is completely determined by the arc algebra,  the cup functors, and the distinguished module $P_{\wp_{\circ}}$ associated to the $H_n$-idempotent defined by $\wp_{\circ}$.  Indeed, the cap functors are adjoints to cups, the twist functors are cones on adjunctions, and the link invariants are obtained from the resulting braid group action as $\Ext$-groups 
\[
\Kh(\kappa_{\beta}) = \Ext_{D(\mathrm{mod}-H_n)}(P_{\wp_{\circ}}, (\beta \times \id)(P_{\wp_{\circ}}))
\] 
for $\beta \in \Br_n$ and $\beta\times\id \in \Br_{2n}$ having closure $\kappa_{\beta}$. 
From our description of the long exact triangle for the braid group half-twists, the symplectic link invariant $\Kh_{symp}$ is obtained by exactly the same procedure, starting from the symplectic arc algebra and its cup functors, and the Lagrangian $L_{\wp_{\circ}}$ which corresponds to $P_{\wp_{\circ}}$, cf. Remark \ref{rem:modules_agree}. Therefore symplectic and combinatorial Khovanov cohomologies co-incide over $\bk$. 
\end{proof}

Up to this point, we have not discussed gradings.  One can use Theorem \ref{thm:main}, or the underlying fact that we have compatible braid group actions on the derived category of graded modules over the arc algebra on both the combinatorial and symplectic sides, to equip $\Kh_{symp}$ with a relative $\bZ\times\bZ$-grading. However, this is of limited interest ,in the sense that the resulting bigrading is not defined or Markov invariant for intrinsically symplectic reasons.  Instead, we show that the isomorphism of Theorem \ref{thm:main} is compatible with the symplectically defined absolute grading mentioned in the Introduction. 

Let $\beta \in \Br_{n}$ be a braid.  The Lagrangian submanifold $L_{\wp_{\circ}}$ admits a grading; by parallel transport, the image $(\beta \times \id)(L_{\wp_{\circ}})$ is also graded, which yields an absolute $\bZ$-grading on the group $HF^*(L_{\wp_{\circ}}, (\beta \times \id)(L_{\wp_{\circ}}))$, which is independent of choices.   Then
\begin{equation} \label{eq:grading_shift}
\Kh^{\ast}_{symp}(\kappa_{\beta}) = HF^{\ast + n + w}(L_{\wp_{\circ}}, (\beta\times \id)(L_{\wp_{\circ}}))
\end{equation}
defines a Markov-invariant absolute $\bZ$-grading on symplectic Khovanov cohomology, where $n$ is the number of strands in the braid $\beta$, and $w$ is the writhe, and the braid closure $\kappa_{\beta}$ is canonically oriented as a link by orienting each of the strands of $\beta$ in the same direction.  We also recall the unoriented skein relation\footnote{According to long-standing convention, a positive (fibred Dehn) twist in the braid group corresponds to the negative crossing in a link diagram, and vice-versa.} for Khovanov homology, which reads (for a positive respectively negative crossing)
%\begin{multline}
% \label{eq:plustriangle}
 %\qquad \cdots \longrightarrow \Kh^{i,j}(\northpos) \longrightarrow
 %\Kh^{i,j-1}(\smoothingnorth) \longrightarrow
 %\Kh^{i-v,j-3v-2}(\sourcesink) \qquad\qquad \\ \longrightarrow
 %\Kh^{i+1,j}(\northpos) \longrightarrow \cdots \qquad \qquad \qquad
 %\qquad \quad \qquad \qquad \qquad
%\end{multline}
\begin{multline}
\label{eq:plustriangle}
 \qquad \cdots \longrightarrow \Kh^{i,j}(\overcrossing) \longrightarrow
 \Kh^{i,j-1}(\asmoothing) \longrightarrow
 \Kh^{i-v,j-3v-2}(\hsmoothing) \qquad\qquad \\ \longrightarrow
 \Kh^{i+1,j}(\overcrossing) \longrightarrow \cdots \qquad \qquad \qquad
 \qquad \quad \qquad \qquad \qquad
\end{multline}
and
\begin{multline}
 \label{eq:minustriangle}
 \quad \cdots \longrightarrow  \Kh^{i,j}(\undercrossing) \longrightarrow
 \Kh^{i-v+1,j-3v+2}(\hsmoothing) \longrightarrow
 \Kh^{i+1,j+1}(\asmoothing) \quad\quad \\ \longrightarrow
  \Kh^{i+1,j}(\undercrossing) \longrightarrow \cdots \qquad \qquad \qquad
  \qquad \quad \qquad \qquad \qquad
\end{multline}
Here, in the complement of the crossing under consideration,
 one takes the arc which ends at the top left corner of the crossing, and sets $v$ to be the signed number of crossings between this arc and the other
connected components of the complement (this compensates for
the non-local change of orientation that occurs in the given 
way of resolving the crossing).  To extend the definition of \eqref{eq:grading_shift} to a diagram of a link which is not a braid closure, such as the unoriented crossing resolution, one should interpret the shift $n+w$ as the sum of the \emph{number of cups} $n$ in the diagram with the writhe $w$ as before, compare to \cite[Section 5]{CK} or \cite[Equation 40]{Reza2}.  Given this, considering for instance the first of the two skein triangles (for the positive crossing and hence negative twist monodromy), one finds that in the $k=i-j$ grading the terms which occur in the first line have degrees $k, k+1, k+2v+2$ respectively.  Suppose the given link is presented as a braid closure for $\beta \in \Br_n$ with writhe $w$; then the corresponding absolute Floer gradings are for a sequence (in an obvious schematic notation)
\begin{multline}
 \label{eq:plustriangle2}
 \qquad \cdots \longrightarrow HF^{k+n+w}(\overcrossing) \longrightarrow
 HF^{(k+1)+n+(w-1)}(\asmoothing) \longrightarrow
 HF^{(k+2v+2)+(n+1)+(w-2v-1)}(\hsmoothing) \qquad\qquad \\ \longrightarrow
 HF^{(k+1)+n+w}(\overcrossing) \longrightarrow \cdots \qquad \qquad \qquad
 \qquad \quad \qquad \qquad \qquad
\end{multline}
Here we note that the writhe decreases by one in $\overcrossing \to \asmoothing$ whilst the number of cups increases by one in  $\asmoothing \to \hsmoothing$.  Furthermore, if the original $\beta$ has $a$ positive crossings and $b$ negative crossings, then the unoriented resolution has $a-v-1$ positive and $b+v$ negative crossings, so the writhe of the third term differs from $w$ by $2v+1$. In short, the degrees in the exact triangle in $i-j$ grading are precisely those associated with the mapping cone construction of the twist in Corollary \ref{cor:triangle_2}, with the boundary map of the mapping cone having cohomological degree $+1$ and the other arrows degree $0$ respectively $2$ (the last being the dimension of the vanishing sphere, cf. the degrees in \cite{Seidel:LES}).  The same discussion applies to the skein triangle for the negative crossing, versus the mapping cone of Proposition \ref{prop:inverse_twist_LES}.   Note that if $\bigcirc$ is the unknot, the absolute grading on $\Kh_{symp}(\bigcirc)$ concentrates that group in symmetric degrees $\{\pm 1\}$. Given  that the cohomological grading in the symplectic arc algebra agrees with the $i-j$-grading collapse in the combinatorial arc algebra (which is concentrated in a single homological degree, and is given by tensor products of $H^*(S^2)[1]$ in the $j$-grading),  one infers that, in the absolute grading, one has an isomorphism
\[
\Kh_{symp}^k = \bigoplus_{i-j=k} \Kh^{i,j}
\]
which completes the proof of \cite[Conjecture 2]{SS}.

\appendix
  
   %%%%%%%%%%%%%%

\section{The long exact sequence of a twist} \label{Sec:LES_of_twist}

\subsection{The result}
Let $(E,\omega = d \theta,J)$ be an exact symplectic manifold equipped with a compatible almost complex structure $J$ such that $E$ is geometrically bounded in the sense of Sikorav \cite{Sikorav}. Recall that this means that for some auxiliary Riemannian metric $g$ on $E$, we have that
\begin{enumerate}
\item $J$ is uniformly tamed, so there are positive constants $\alpha$ and $\beta$ so that $\omega(X,JX) \geq \alpha \|X\|_g^2$, $\omega(X,Y) \leq \beta \|X\|_g \cdot \|Y\|_g$; 
\item the Riemannian manifold $(E,g)$ has sectional curvature uniformly bounded above and injectivity radius uniformly bounded below.
\end{enumerate} 
The conditions  imply completeness of $g$, as well as the monotonicity Lemma which bounds the diameter of holomorphic curves in terms of their area,  which yields the following a priori estimate:
\begin{prop} \label{prop:Sikorav}
Let $K \subset E$ be a compact subset. For each positive real number $A$, there is a compact subset $K(A)$ so that any map from a compact Riemann surface with boundary to $E$, which is $J$-holomorphic outside of $K$ and has geometric energy bounded by $A$, has image contained in $K(A)$. \qed
\end{prop}

\begin{defn}
An \emph{exact Landau-Ginzburg model} is a smooth map $W \co E \to \bC$ such that, outside a compact set $K \subset \bC$, $W$ is a $J$-holomorphic submersion, and the symplectic connection defined by $\omega$ over $\bC\backslash K$ is flat.
\end{defn}
For convenience, we specify that the compact set is the  disc of radius $1/4$ centered at $(-1/2,0)$, and assume that the triple $(E,\theta,J )$ is trivial over the right half-plane. For the purpose of defining $\bZ$-graded Fukaya categories, we will furthermore assume that $c_1(E)=0$, and fix a trivalisation of the top exterior power of $TE$ as a complex vector bundle.

We will find it useful to study  inclusions of Landau-Ginzburg models, in the following set-up.  Recall that a Lefschetz fibration $W \co E \to \bC$ consists of the data of an exact symplectic manifold $E$ with contact boundary, equipped with a map to $\bC$, whose singularities are disjoint from the boundary and are locally modelled after the quadratic map
\begin{equation}
  \sum z_i^2 \co \bC^{n} \to \bC.
\end{equation}
In addition, we require that $W$ be a symplectic fibration away from the critical points, i.e. that the fibres are symplectic submanifolds of $E$.
%to a convex domain which is proper over $\bC$.
%taking $(z_1, \ldots, z_n)$ to $-1/2 + \sum z_i^2$; the smooth fibres are affine quadrics, which are symplectomorphic to the cotangent bundle of $S^{n-1}$.   As in  Seidel \cite{??}, we introduce a codimension $0$ submanifold with contact boundary $E_{st} \subset \bC^{n}$ whose fibre is symplectomorphic to the unit disc bundle; we denote the restriction of $-1/2 +\sum z_i^2  $ to $E_{st}$ the \emph{standard Lefschetz fibration}:
%\begin{equation}
%  W_{st} \co E_{st} \to \bC.
%\end{equation}

Consider the following specific geometric situation: (i) $M$ is a Liouville domain, (ii) $W \co E \to \bC$ is a Lefschetz fibration with a unique critical point, whose smooth fibre we denote $N$, and (iii)  $W' \co E' \to \bC $ is an exact Landau-Ginzburg model with fibre $M'$, equipped with an inclusion
\begin{equation}
 M \times E \subset E'
\end{equation}
which is compatible with the map to $\bC$, and so that the difference between the chosen Liouville forms is exact. In particular, we have a symplectic embedding $M \times N \subset M'$.

Consider subcategories  $\scrA \subset \scrF_{M} $ and $\scrA' \subset \scrF_{M'}$, and assume that the following technical condition holds:
\begin{equation}
  \label{eq:enough_subfibrations}
  \parbox{38em}{For each object $L'$ of $\scrA'$, there exists a Lefschetz fibration $E_{L'} \to \bC$, with fibre $N_{L'}$,  including $E$ as a subfibration containing all critical points, and an inclusion $M \times E_{L'}  \subset E'$ compatible with the maps to $\bC$ such that $L'$ is contained in  $M \times N_{L'} \subset M'$. }
\end{equation}

\begin{rem}
 Note that $M'$ is exact and geometrically bounded, but no Liouville structure is specified, so that it does not directly fall within the purview of Seidel's construction \cite{FCPLT}. Nonetheless, the construction of the category of \emph{compact} exact Lagrangians in $M'$ is standard, as the only issue is to ensure compactness of moduli spaces of (pseudo)-holomorphic curves, which follows under the assumption that all perturbations are compactly supported by Proposition \ref{prop:Sikorav}. 
\end{rem}

We denote by $\phi \co M' \to M'$ the inverse monodromy of the fibration, i.e. the symplectomorphism obtained by parallel transport along a loop going clockwise once around a circle lying in the region where $W'$ is locally flat. We write $\Delta^{\Tw \scrA'}_{\phi}$ for the graph bimodule over the category of twisted complexes on $\scrA'$ associated to the endo-functor of the Fukaya category induced by $\phi$.

Fix a closed exact Lagrangian $K \subset N$. As in Section \ref{sec:cup-bimodules-are}, let $\cup \co \scrA \to \scrF_{M'}$ denote the functor which assigns to a Lagrangian  $L \in \Ob \scrA $ its product with $K$ (we constructed this indirectly as a representing functor for the bimodule $\scrK$). The following result is proved in  Section \ref{sec:fibred-twist}, at the very end of this Appendix, as a consequence of Theorem \ref{thm:les}; the latter in turn is a special case, sufficient for our purposes, of a more general result which will appear in forthcoming work of the first author and Sheel Ganatra \cite{Abouzaid-Ganatra}.

\begin{prop} \label{prop:what_we_need_of_the_twist}
Assume that for every object $L'$  of $\scrA'$, there is an object $L$ of $\scrA$ such that either  (i) $L'$ is quasi-equivalent to the product of $L \times K$ or (ii) $L'$ is quasi-equivalent to a Lagrangian which meets $L \times K$ cleanly along a section of the projection to $K$. There is a functor $ \cap \co \Tw \scrA' \to \Tw \scrA$ which is left adjoint to $\cup$. Moreover, the graph bimodule $ \Delta^{\Tw \scrA'}_{\phi}$ is quasi-isomorphic to the cone of the unit:
\begin{equation}
 \Delta^{\Tw \scrA'}_{\phi} \ \simeq \ \mathrm{Cone}\left(
  \Delta^{\Tw \scrA'} \to \Delta^{\Tw \scrA'}_{\cup \circ \cap  } \right).
\end{equation}
\end{prop}

The reader may compare with Proposition \ref{prop:inverse_twist_LES} to see how this result is applied. Its assumptions are neither optimal nor natural. Assumption \eqref{eq:enough_subfibrations} would more naturally be replaced by the property of being a fibration with Morse-Bott critical locus and \emph{globally integrable symplectic connection} (yielding global parallel transport maps), and the hypothesis on objects of $\scrA'$ should be dropped. The global properties of parallel transport are unfortunately not well understood in the desired application, because the K\"ahler form on the slice $\Slice_n$ is only known to have well-defined parallel transport maps out of bounded subsets of the critical locus of a generic fibre in the discriminant locus of the adjoint quotient $\chi$.  In a different direction, omitting the conditions (i) or  (ii) on objects of $\scrA'$ in the statement of the Proposition would require additional algebraic machinery, which whilst natural in a general development is not needed for our particular application (this however is undertaken in the forthcoming work of Abouzaid-Ganatra).

%We briefly set this general discussion in the context  of the paper: $W' \co E' \to \bC  $ is a Morse-Bott Lefschetz fibration with critical locus $M = \scrY_{n-1} \subset W^{-1}(0)$ and fibre $M'=\scrY_n$. The categories $\scrA$ and $\scrA'$ are respectively $\scrH_{n} $ and $\scrH_{n-1}$.  The monodromy $\phi$ is taken clockwise, and corresponds to a negative half-twist generator of the braid group acting on $\scrY_n$. The existence of the Lefschetz fibration in Assumption \eqref{eq:enough_subfibrations} follows from the existence of fibred $A_2$ degenerations as in Lemma \ref{lem:local_A2_degeneration}.
% If parallel transport out of the critical locus is defined on a compact set containing a Lagrangian $K$, then one obtains a global thimble $V(K) \subset E$, and its intersection $V(K) \cap M$ with a generic fibre. Proposition \ref{prop:U=cup} implies that the restriction bimodule $\scrR$ and the cup functor associated to the Morse-Bott degeneration of $\scrY_n$  both act on objects of $\scrF(\scrY_{n-1})$ via $K \mapsto V(K)$.  Informally, $\scrR = \cup$ and $\scrD = \cap$.

\subsection{The setting}
The proof of Proposition \ref{prop:what_we_need_of_the_twist} relies on a more abstract adjunction involving the Fukaya category $\scrF_W$ of a Landau Ginzburg potential $W: E \rightarrow \bC$, which we construct in section \ref{sec:fukaya-category-w}. The idea that there should be such a Fukaya category is due to Kontsevich. There are many implementations, some of which are not yet in the literature (see, e.g. \cite{FCPLT,Abouzaid:toric,Seidel:Lefschetz-I,Seidel:Lefschetz-II,BC:Lef,AS:Lef-Wrap}). We shall adopt a viewpoint which is a mixture of Seidel's approach in \cite{Seidel:Lefschetz-I} and Abouzaid-Seidel's approach in \cite{AS:Lef-Wrap}, and where the Lagrangians we consider are similar to those studied by Biran and Cornea in \cite{BC:Cob,BC:Lef}. In particular, the objects of this category satisfy the following geometric condition:

\begin{defn} \label{def:h-admissible}
An exact Lagrangian in $E$ is \emph{horizontally admissible} if it is proper over $\bC$, and its image under $W$ agrees, outside a compact set, with a finite union of half-lines parallel to the positive real axis.
\end{defn}
Note that \cite[Lemmas 16.2 \& 16.3]{FCPLT} show that, over any half-line $\delta: [0,\infty) \rightarrow \bC$ near infinity with $W(E) \supset \im(\delta)$, such an admissible Lagrangian $L$ is smoothly fibred over $\delta$, and the component $L \cap W^{-1}(\im(\delta))$ is obtained by parallel transport along $\delta$ of a smooth Lagrangian submanifold of $W^{-1}(\delta(0))$. 

The \emph{height} $h(L)$ of an admissible Lagrangian $L$ is the collection of real numbers which appear as $y$-coordinates of the corresponding half-lines near infinity. A brane structure on such a Lagrangian consists of a choice of $\Spin$ structure and a real lift of the $S^1$ valued phase.  If $W(L)$ comprises a single half-line near infinity, we say that $L$ has \emph{one end}.

Let $M$ denote a fibre of $W$ over the right half-plane. We denote by $\phi \co M \to M$ the monodromy symplectomorphism obtained by parallel transport along a loop going \emph{clockwise} once around a circle which is sufficiently large that it lies entirely in the region where $W$ is locally flat.  In Section \ref{sec:fukaya-category-m} we construct a particular model for the Fukaya category $\scrF_M$ of this fibre. We write $\Delta_{\phi}$ for the graph bimodule of the endo-functor of the Fukaya category induced by $\phi$.

In section \ref{sec:orlov-functor}, we  construct a functor $\scrD \co \scrF_M \to \scrF_W$ which the first author first heard described by Dima Orlov.  Let $ _{\scrD} \Delta^{ \scrF_W } _{\scrD}  $ denote the $2$-sided pullback of the diagonal bimodule $\Delta^{\scrF_W}$ of $\scrF_W $ by $\scrD  $. Applying $\scrD$ induces a natural map
\begin{equation}
\Delta^{ \scrF_M }  \rightarrow  {_{\scrD} \Delta^{ \scrF_W } _{\scrD}} 
\end{equation}
of bimodules which we call the unit.  
\begin{thm}[Abouzaid-Ganatra]\label{thm:les}
The cone over the unit  $ \Delta^{\scrF_M} \to  {_{\scrD} \Delta^{ \scrF_W } _{\scrD}}   $ is quasi-isomorphic to the graph bimodule $\Delta_{\phi}$ of the clockwise monodromy.
\end{thm}

Very schematically, the functor $\scrD$ involves taking the parallel transport of a Lagrangian $L \subset M$ along an arc which has two ends and encircles the origin, cf. Figure \ref{fig:curves_for_cap_functor}.  There is a canonical two-step filtration of Floer complexes for objects in the image of $\scrD$ with objects in the fibre,  coming from the two-ended structure, and the exact triangle arises from that filtration.

\subsection{Floer cochains and operations}

If $h(L)$ is disjoint from $h(L')$, we define the Floer complex $ CF^*(L,L') $
using the methods of \cite{FCPLT}; i.e. we pick a compactly supported pair $(H,K)$ consisting of a Hamiltonian on $E$ and a perturbation of the complex structure $J$, so that regularity is achieved for all moduli spaces of finite energy solutions to Floer's equation with boundary on $L$ and $L'$. The Floer complex is generated by Hamiltonian chords starting at $L$ and ending on $L'$, and the Floer differential counts solutions to Floer's equation with such ends. The moduli spaces of Floer trajectories are precompact because (i) projection to the base is holomorphic outside a compact set, so that the maximum principle applies, and (ii) Proposition \ref{prop:Sikorav} implies that moduli spaces of pseudo-holomorphic curves of bounded energy whose boundaries are contained in a compact set cannot escape to infinity (exactness provides the necessary energy bound). 

By choosing a proper Morse function $f_L \co L \to [0,\infty)$, and defining
\begin{equation} \label{eq:Floer=Morse}
  CF^*(L,L) = CM^*(L; f_L),
\end{equation}
we extend this definition to the case $L = L'$.

There is a special case in which this complex may be readily computed: we say that $L_{-\epsilon}$ is a  \emph{small negative perturbation}  of $L$ if is obtained by a small Hamiltonian isotopy which  decreases all heights by $-\epsilon < 0$. For an appropriate almost complex structure, the Floer complex $CF^*(L,L_{-\epsilon})$ is isomorphic to the Morse complex of a function on $L$ whose gradient flow points outwards at infinity (see \cite{FukayaOh, Abouzaid:toric}). In particular:
\begin{lem} \label{lem:degree_0_gen_small}
There is a canonical degree $0$ generator of $ CF^0(L,L_{-\epsilon}) $ for $-\epsilon$ sufficiently small. \qed
\end{lem}

We shall need a technical result,  a version of which forms the basis to all approaches to the Fukaya category of $W$.  Consider a half-plane  $H_{a} = [a,\infty) \times \bR $, and two collections of functions $\{ f_i \}_{i \in A}$ and $\{ f'_j \}_{j \in A'} $ on $[a,+\infty)$, which are locally constant on a neighbourhood of $a$ and $\infty$. Let $\Gamma$ and $\Gamma'$ denote the unions of the graphs of these functions; assume that the curves in each of these sets are pairwise disjoint.
\begin{defn} \label{def:monotone_pair}
A monotone pair $(\Gamma, \Gamma')$ is a pair which satisfies:
\begin{equation}
  \frac{df_i}{dx}(x) \geq \frac{df'_j}{dx}(x) \textrm{ for all } (i,j) \in A \times A', \textrm{ with strict inequality if } f_i(x) = f'_j(x).
\end{equation}
\end{defn}
\begin{figure}
\centering
\begin{tikzpicture}
% scope environment restricts the region that is shown
\begin{scope}               
\node[label=right:{$f_1$}] at (2,0) {};
\draw[line width=\lw, blue] (0,0)--(2,0);
\node[label=right:{$f_2$}] at (2,-2) {};
\draw[line width=\lw, blue] (0,-2)--(2,-2);
\draw[line width=\lw] (0,1)  .. controls (1, 1) and  (1, -3) .. (2, -3);
\node[label=right:{$f'_2$}] at (2,-3) {};
\draw[line width=\lw] (0,2)  .. controls (1, 2) and  (1, -1) .. (2, -1);
\node[label=right:{$f'_1$}] at (2,-1) {};
\end{scope}
\end{tikzpicture}
\caption{}
\label{fig:monotone_pair}
\end{figure}

Consider Lagrangians $L$ and $L'$ which respectively project to the graphs of $\Gamma$ and $\Gamma'$, and assume that their intersections over $H_{a}$ are transverse.  Assuming that the heights of $L$ and $L'$ are disjoint, let $  CF^*(L,L') $  be defined with respect to a Hamiltonian supported away from $W^{-1}(H_{a})$, and a perturbation of almost complex  structures for which $W$ is holomorphic over $H_a$. Consider $CF^*_{in}(L,L')  \subset CF^*(L,L') $ to be the submodule generated by intersections projecting away from $H_{a}$.  We claim this is a subcomplex:
\begin{lem} \label{lem:inner_subcomplex}
The differential preserves $CF^*_{in}(L,L')$, and all holomorphic curves which contribute to the differential on this subcomplex have image which project to $\bC\backslash H_a$.
\end{lem}
\begin{proof}
Let $u$ be a holomorphic strip whose input is a generator of $ CF^*_{in}(L,L') $. After possibly deforming $a$, each component $\Sigma$ of  $(W \circ u)^{-1}(H_a)$ is a surface with corners; we shall show by contradiction that all such components are constant and contained in the vertical line $x=a$. Letting $v$ denote the restriction of $W \circ u $ to such a component we first consider the situation in which this component does not include the outgoing end. In this case, every boundary segment of $\Sigma$ which maps under $v$ to a curve in $\Gamma$ or $\Gamma'$ must have both endpoints at the same point of the vertical line $x=a$. In particular, the intersection number of the boundary of $v (\Sigma)$ with a horizontal half-ray starting at $x=a$ must vanish. This implies invariance of the signed count of pre-images of any point in the plane, which then vanishes because of vanishing at infinity. By the open mapping theorem, we conclude that the image of this component is constant, and contained in the boundary.

In order to extend this argument to the case that $\Sigma$ includes the outgoing end, let $v$ denote the extension of $W \circ u $  to the  compactification $\overline{\Sigma}$, and let $z$ denote the image of the point at infinity. The two segments of $\partial \overline{\Sigma}$ meeting at the point at infinity map to curves $\gamma$ and $\gamma'$ in $\Gamma$ and $\Gamma'$ which meet at $z$.  The situation is summarised in Figure \ref{fig:inner_subcomplex}, and the ordering convention for the output of a Floer differential is such that the degree of $v$ over the unique embedded triangle in that figure must be negative in order for the degree to vanish for large values of the $x$-coordinate. By positivity of degree (and the open mapping theorem) we conclude that such a component cannot exist.
\end{proof}
\begin{figure}
\centering
\begin{tikzpicture}
% scope environment restricts the region that is shown
\begin{scope}               
\node[label=right:{$\gamma$}] at (2,0) {};
\draw[line width=\lw, blue] (0,0)--(2,0);
\draw[line width=\lw] (0,1)  .. controls (1, 1) and  (1, -1) .. (2, -1);
\node[label=right:{$\gamma'$}] at (2,-1) {};
\draw[line width=2*\lw, dashed] (0,-1) -- (0, 1.5);
\end{scope}
\end{tikzpicture}
\caption{}
\label{fig:inner_subcomplex}
\end{figure}

The same methods yield a higher product
\begin{equation} \label{eq:A_oo_operations}
\mu_{d} \co  CF^*(L_d, L_{d-1}) \otimes \cdots \otimes CF^*(L_0, L_1) \to CF^{*}(L_0, L_d)[2-d]
\end{equation}
whenever all Floer complexes are defined. The case in which some of the Lagrangians are equal is handled by counting configurations of holomorphic discs and (perturbed) gradient flow lines, as in \cite{BC:Pearl,Seidel:genus,Sheridan}. Chosen inductively, these products satisfy the $A_{\infty}$ relation. The same argument as in Lemma \ref{lem:inner_subcomplex} shows:
\begin{lem} \label{lem:inner_subA_oo}
If $L_0, \cdots, L_d$ project to collections of curves $\Gamma_0, \cdots, \Gamma_d$ such that each pair $(\Gamma_i, \Gamma_j)$ is monotone if $i < j$, then $\mu_d$ preserves  interior Floer cochains. \qed
\end{lem}

\subsection{The Fukaya category of $W$} \label{sec:fukaya-category-w}

The collection of horizontally admissible branes forms a partially ordered set, with $L > K$ if and only if $h(L) > h(K)$ as subsets of $\bR$. We define a category $\scrO_{W}$ with objects such branes and morphisms
\begin{equation}
\scrO_{W}(L,K) =   \begin{cases} 
  CF^*(L,K) &  \textrm{ if } L \geq K \\
0 & \textrm{otherwise.}     
  \end{cases}
\end{equation}
The $A_{\infty}$ operations are given by Equation \eqref{eq:A_oo_operations}.

We shall define the Fukaya category as a localisation of $\scrO_W$, following \cite{AS:Lef-Wrap}. To this end, consider a pair $(L,L')$ such that there is a real number $2 < a$, and a small negative perturbation $L_{-\epsilon}$ of $L$ which projects to straight lines on $H_{a}$ such that (i) $L'$ agrees with $L_{-\epsilon}$ away from $W^{-1}(H_a)$, and (ii) the pair $(L,L')$ projects to a monotone pair of arcs outside a compact set. In that case, Lemma  \ref{lem:inner_subcomplex} implies that we have a subcomplex
\begin{equation}
 CF^*(L, L_{-\epsilon})  = CF^*_{in}(L,L')  \subset CF^*(L,L'),
\end{equation}
hence Lemma \ref{lem:degree_0_gen_small} yields a class which we call a \emph{quasi-unit}
\begin{equation}
  \kappa \in H\scrO_{W}(L,L').
\end{equation}

\begin{Definition}[Abouzaid-Seidel]
  The Fukaya category $\scrF_W$ is the localisation of $\scrO_W $ with respect to quasi-units.
\end{Definition}

  The definition of localisation relies on the quotient construction of $A_{\infty}$ categories. Following Drinfeld \cite{Drinfeld} in the differential graded case, a convenient model for such quotients was provided by \cite{LO}. By the universal property of localisations, there is a functor $\scrO_W \to \scrF_W$. The main justification of our definition is the following result:
\begin{prop}
  If $L > K$, the localisation map induces an isomorphism on cohomology
  \begin{equation}
    H \scrO_{W}(L,K) \cong  H\scrF_{W}(L,K). 
  \end{equation}
\end{prop}
\begin{proof}[Sketch]
The proof  is completely formal once one shows that multiplication with respect to  quasi-units induces an isomorphism on homology among directed objects, i.e. if $L > K $, and $ \kappa_+ \co L_+ \to L $ and $\kappa_- \co K \to K_-$ are quasi-units, then the maps
\begin{align}
\mu_{2}(\_, \kappa_+) \co CF^*(L,K) & \to CF^*(L_+,K) \\
 \mu_{2}(\kappa_-, \_) \co CF^*(L,K) & \to CF^*(L,K_-) 
\end{align}
induce isomorphisms on homology. The essential point is that, in the isotopy between $L$ and $L_+$, no intersections with $K$ at infinity are created or destroyed (and similarly for swapping the roles of $K$ and $L$).  The proof then follows from  invariance of Floer cohomology under continuation maps \cite{Seidel:Lefschetz-I,BC:Cob};  compare to \cite[Lemma 10.7]{Seidel:CatDynamics} for a related localisation construction. 
\end{proof}

\begin{Remark}
The category $\scrF_W$ is a ``partially wrapped" category, in the terminology of e.g. \cite{Auroux}.  In constructing any version of the Fukaya category of a Liouville manifold which involves non-compact Lagrangian submanifolds $L$ and perturbations by a Reeb-type flow $\phi_H$  at infinity defined by a Hamiltonian function $H$, one must always contend with the fact that $A_{\infty}$-operations in Hamiltonian Floer cohomology are defined by maps 
\begin{equation} \label{eq:basicproblem}
CF^*(L,\phi_{H_1}(L)) \otimes \cdots \otimes CF^*(L, \phi_{H_k}(L)) \longrightarrow CF^*(L, \phi_{H_1+\cdots+H_k}(L))
\end{equation}
which involve multiples $H_i = \lambda_i H$ of a given Hamiltonian function, and the fact that the Floer complex on the right of \eqref{eq:basicproblem} is not isomorphic at chain level to the factors on the left.  The ``telescope construction" of \cite{Abouzaid-Seidel} circumvents this for ``fully wrapped" categories, but there are additional complications for compactness of spaces of holomorphic curves, and well-definition of the required continuation maps for direct systems of partially wrapped Floer groups,  when the Hamiltonian flow is degenerate on a subset of the contact boundary (roughly stemming from non-properness of $H$ on the completion).  The localisation construction of \cite{AS:Lef-Wrap}, borrowed above, is designed to circumvent these issues.
\end{Remark}

One-ended Lagrangians, which by definition project to a single arc outside a compact set, play a special role in the theory. In our intended applications, we shall need a more flexible notion: to this end, we say that a Lagrangian $L$ is \emph{weakly one-ended} if there exists a positive real number $x_0 \in \bR$ such that $W(L)$ agrees near the vertical line $x=x_0$ with a horizontal line (say $y=y_0$). Let $\Lambda$ denote the fibre of $L$ over $(x_0,y_0)$, and let $L_{in} \subset L$ be the submanifold of $L$ (with boundary $\Lambda$) defined by the inequality that the real part of $W$ is bounded by $x_0$.  We define $T_L$ be the Lagrangian which agrees with $L_{in}$ to the left of $x=x_0 $, and with the parallel transport of $\Lambda$ along $y=y_0$ to the right of this line. 
\begin{lem} \label{lem:weakly_one_ended}
The Lagrangians $L$ and $T_L$ are quasi-isomorphic in $\scrF_W$.  
\end{lem}
\begin{proof}[Sketch of proof:]
To show that $L$ is a summand of $T_L$ in  $\scrF_W$,  it suffices to construct Lagrangians $T^{+}_L$ and $T^-_L$ which are quasi-equivalent to $T_L$ in $\scrF_W$, together with maps $f \in HF^*(T^{+}_L,L)$ and $g \in   HF^*(L, T^{-}_L)$ such that the product
\begin{equation} \label{eq:product_classes_quasi-unit}
  \mu^2(g,f) \in HF^*(T^{+}_L,  T^{-}_L) )
\end{equation}
is a quasi-unit. 

Pick a $C^2$-small Morse function $h$ on $L$ whose restriction to a neighbourhood of $\Lambda$ agrees with the sum of a Morse function on $\Lambda$ with a small multiple of $(x-x_0)^2 + (y-y_0)^2) $. We can then pick $T^{+}_L$  (respectively $T^-_L$) to agree with Hamiltonian pushoff of $L_{in}$ by $-h$ (respectively $h$) in the region $x\leq x_0$, and with the parallel transport of $\Lambda$ by a line above $y=y_0$ (respectively below $y=y_0$) which does not  intersect $W(L)$. Since all relevant intersection points project to the left of the line $x=x_0$, we have identities of sets
\[
L \cap T_L^- = T_L^+ \cap L = T_L^+ \cap T_L- = \mathrm{crit}(h|_{L_{in}});
\]
at the level of Floer cochains, these yield isomorphisms of Floer groups
%This situation is illustrated in Figure ??, and we have:
\begin{equation}
HF^*(L, T^{-}_L) \cong HF^*(T^{+}_L,L)  \cong HF^*(T^{+}_L,  T^{-}_L) ) \cong H^*(L_{in}),
\end{equation}
since all of the complexes are identified with the Morse complex of $h|_{L_{in}}$. These isomorphisms are compatible with multiplication, hence the quasi-unit in $HF^*(T^{+}_L,  T^{-}_L)$ can be written as a product as in Equation \eqref{eq:product_classes_quasi-unit}.  This completes the proof that $L$ is a summand of $T_L$; the proof that they are quasi-isomorphic follows by reversing their r\^oles in the above argument.
\end{proof}

\subsection{The Fukaya category of $M$} \label{sec:fukaya-category-m}

Let $M$ be the fibre of $W$ at a point in the upper half-plane; by our assumptions on $W$, parallel transport in the right half-plane yields an identification of any pair of such fibres which preserves the primitive $\theta|M$ and the complex structure $J|M$.

To ease comparison with $\scrF_W$, we define a version of the Fukaya category of $M$ by localisation: fix a finite collection of exact Lagrangian branes $\scrL_M$ in $M$. The objects of $\scrO_M$ are pairs $(L,i)$, with $L \in \scrL_M$, and $i$ a negative integer. Choose a sequence of Hamiltonian perturbations $\{ L^i \}$ of $L$ which are uniformly $C^2$-small, so that
\begin{equation}
  \parbox{30em}{$L^i$ is transverse to $ K^j  $ whenever $i \neq j$.} 
\end{equation}
We obtain a directed category with morphisms
\begin{equation}
 \scrO_{M}((L,i),(K,j)) =  \begin{cases}   CF^*(L^i,K^j) &  \textrm{ if } (L,i) \geq (K,j) \\
0 & \textrm{otherwise,}     
  \end{cases} 
\end{equation}
where we again choose an auxiliary Morse function on each Lagrangian to define self-Floer cohomology, and the partial order is $(L,i) > (K,j)$ if and only if $i>j$.  The definition of the Floer cochains uses perturbed almost complex structures, but we require that the inhomogeneous terms vanish. Given that the Hamiltonian perturbations were assumed to be $C^2$-small, there is a canonical element
\begin{equation}
  \kappa_{i,j} \in HF^*(L^i, L^j) 
\end{equation}
under the identification of Floer cohomology with Morse cohomology. If $i \geq j$, this represents a morphism in $H\scrO_{M}((L,i),(L,j)) $.
\begin{defn}
The Fukaya category $\scrF_M$ of $M$ is the localisation of $\scrO_{M}$ at the continuation elements $\kappa_{i,j}$. 
\end{defn}
The objects $ (L,i)$ and $(L,j) $ are quasi-isomorphic in $\scrF_M $, but not identical.  If $\scrF_{M}$ denotes the ``usual" Fukaya category, as constructed in \cite{FCPLT} and used in the body of this paper, one can construct an $A_{\infty}$-functor $\scrO_M \rightarrow \scrF_{M}$ which sends $\kappa_{i,j}$ to a quasi-isomorphism for every $i,j$. From that point, a formal argument based on the properties of localisation again shows that $\scrF_M$ and $\scrF_{M}$ are quasi-equivalent; see \cite{AS:Lef-Wrap} and \cite[Remark 10.8]{Seidel:CatDynamics}.

\subsection{Vertical Lagrangians}

 Pick a monotonically increasing function $g$ on $\bR$ which agrees with $-1$ for $y \ll 0$, with $1$ for $0 \ll y$, and which vanishes with non-zero derivative at the origin. Let $g_i = (1+\frac{1}{i}) g(y)$, as shown on Figure \ref{fig:functions_rescaling}.  For each negative integer $i$, Lagrangian $L \in \scrL_{M}$, and point $p$ in the right half plane,  let $\scrV_p(L,i) $  denote the parallel transport of $L^i$ along the curve  $ p+ (g_n(y),y)$.
\begin{figure}
\centering
\begin{tikzpicture}
% scope environment restricts the region that is shown
\begin{scope}               
  \node[label=right:{$g_{-1}$}] at (2,0) {};
\draw[line width=\lw] (-2,0)--(2,0);
\draw[line width=\lw] (-2,-.5)  .. controls (0, -.5) and  (0, .5) .. (2, .5);
\node[label=right:{$g_{-2}$}] at (2,.5) {};
\draw[line width=\lw] (-2,-.667)  .. controls (0, -.667) and  (0, .667) .. (2, .667);
\draw[line width=\lw] (-2,-.75)  .. controls (0, -.75) and  (0, .75) .. (2, .75);
\draw[line width=\lw] (-2,-1)  .. controls (0, -1) and  (0,1) .. (2,1);
\node[label=above:{$g$}] at (2,1) {};
\end{scope}
\end{tikzpicture}
\caption{}
\label{fig:functions_rescaling}
\end{figure}
The Lagrangians $\scrV_p(L,i)$ are not horizontally admissible in the sense of Definition \ref{def:h-admissible}, but they are vertically admissible. Extending the choices of almost complex structure on $M$ used to define $CF^*(L^{i_0}_0,L^{i_1}_1)$ to $E$, we obtain a canonical isomorphism
\begin{equation}
CF^*(L^{i_0}_0,L^{i_1}_1) \equiv  CF^*(\scrV_p(L_0,i_0), \scrV_p(L_1,i_1))
\end{equation}
of free abelian groups given by the inclusion of intersection points. Moreover, given a sequence $(L^{i_0}_0, \cdots , L^{i_d}_d)$, the maximum principle implies that all holomorphic discs in $E$ with boundary on  the sequence obtained by applying $\scrV_p$ are contained in the fibre over $0$, i.e. we have an identification between moduli spaces of holomorphic discs in $E$ and $M$. The following result was proved by Seidel in \cite{FCPLT}:
\begin{prop} \label{prop:regularity-good-order}
 If $i_d < \cdots < i_0$, regularity for holomorphic discs with boundary conditions $ (L^{i_0}_0, \cdots , L^{i_d}_d) $  is equivalent to regularity of the corresponding disc with boundary conditions $(\scrV_p(L_0,i_0), \cdots , \scrV_p(L_d,i_d))  $.  \qed
\end{prop}
We can extend this discussion to Lagrangians which are equal by choosing the Morse function on $\scrV_{p}(L,i)$ to be the sum of the Morse function on $L^i$ with $(y-p)^2$. With these choices, the $A_{\infty}$-structure on $\scrO_{M}$ can be equivalently defined as a subcategory of a category of vertically admissible Lagrangians in $E$.

\subsection{The restriction bimodule} \label{sec:restr-bimod}
Given a horizontally admissible Lagrangian $T$ and a vertically admissible Lagrangian $V$, we define $
  CF^*(T,V) $ by choosing a compactly supported Hamiltonian on $E$ which maps $T$ to a Lagrangian transverse to $V$. More generally, given sequences $(T_0, \ldots, T_s)$ and $(V_r, \ldots, V_0)$ of horizontally and vertically admissible Lagrangians projecting to arcs which are disjoint outside a compact set, the count of discs with $r+s+2$ boundary marked points defines an operation
\begin{multline} \label{eq:Floer_bimodule_maps}
CF^*(V_{r-1}, V_r) \otimes \cdots \otimes CF^*(V_0, V_1) \otimes CF^*(T_s,V_0) \\
 \otimes  CF^*(T_{s-1},T_s) \otimes \cdots \otimes CF^*(T_{0}, T_{1}) \longrightarrow CF^*(T_{0},  V_r),
\end{multline}
which satisfies the equation for an $A_{\infty}$-bimodule.

We now define an $\scrO_{M}$-$\scrO_W$ bimodule $\scrR$, called the restriction bimodule, as follows:  given objects $(L,i)$ of $\scrO_M$ and $T$ of $\scrO_W$, we set
\begin{equation}
  \scrR(T,(L,i)) = \begin{cases} CF^*(T,\scrV(L,i))  & \textrm{ if }  2i+1  < h(T)  \\
0 & \textrm{ otherwise. } 
 \end{cases}
\end{equation}
The structure maps are obtained from Equation \eqref{eq:Floer_bimodule_maps} and the identifications of morphism spaces in $\scrO_M $ with Floer groups among vertical Lagrangians.  The key point is that, given sequences $(T_0, \ldots, T_s)$ and $((L_r,i_r) , \ldots,  (L_0,i_0) )$, the condition $2i_0+1  < h(T_s)  $ implies $2i_r+1  < h(T_0)  $ whenever $i_{k+1} \leq i_{k} $ for all $ 0 \leq k  \leq s-1$ and $h(T_k) \leq h(T_{k+1})  $   for all $0 \leq k \leq r-1$. In particular, given sequences such that $T_k \leq T_{k+1} $ and $(L_{k+1}, i_{k+1}) \leq (L_k, i_k)$, the  Floer complexes between horizontal and vertical Lagrangian appearing in Equation \eqref{eq:Floer_bimodule_maps} are respectively isomorphic to $\scrR(T_s,(L_0,i_0))$ and  $\scrR(T_0,(L_r,i_r))$, and all other complexes are morphism groups in $\scrO_M$ or $\scrO_W$. The $A_{\infty}$ equation for the bimodule $\scrR$ therefore follows from the $A_\infty$ equation satisfied by Equation \eqref{eq:Floer_bimodule_maps}.

%To define the structure maps, recall that, given pairs $(T_0,T_1)$ and $((L_0,i_0) , (L_1,i_1) )$, the conditions $h(T_0) > h(T_1)$ and $i_0 > i_1$  imply that $ \scrO_W( T_0,T_1)  $ and $ \scrO_M( (L_0,i_0) , (L_1,i_1)  ) $  are defined as Floer complexes. Assuming in addition that $  2 i_0 +1 < h(T_1)$, we find that all complexes appearing in the structure maps below are given by Floer complexes:
%\begin{align}
%  \scrR(T_1,(L_0,i_0))   \otimes \scrO_W( T_0,T_1)  & \to   \scrR(T_0,(L_0,i_0))  \\
%  \scrO_M( (L_0,i_0) , (L_1,i_1)  ) \otimes  \scrR(T_1,(L_0,i_0))  & \to  \scrR(T_1,(L_1,i_1)) .
%\end{align}

\subsection{The Orlov functor} \label{sec:orlov-functor}
Fix a point $q \in  \bC$, lying in the region where $W$ is a locally flat symplectic fibration. For concreteness, we also assume that $q$ lies to the left of the line $x=2$, though the entire construction can done without this assumption by changing constants below.  In this section, we build a functor from $\scrO_{M_q}$ to $\scrO_W$; the key input is a careful construction of a sequence of arcs in the plane.
\begin{figure}
\centering
\begin{tikzpicture}[scale =1]
% scope environment restricts the region that is shown
\begin{scope}
\node[label=left:{$q=(-2,0)$}] at (-2,0) {};
\draw[dashed] (3,-5)--(3,2);
\node[label=right:{$x=3$}] at (3,1) {};
\draw[dashed] (2,-5)--(2,2);
\node[label=left:{$x=2$}] at (2,-4) {};

\fill[black!40!white] (-.5,0) circle (.125);

\node[label=right:{$\gamma_{-1}$}] at (4,0) {};
\draw[line width=4*\lw] (3,0)--(4,0);
\draw[line width=4*\lw] (3,-1)--(4,-1);
\draw[line width=4*\lw] (2,2)  .. controls (2.5, 2) and  (2.5, 0) .. (3,0);
\draw[line width=4*\lw] (2,-1)  .. controls (2.5, -1) and  (2.5, -1) .. (3,-1);  
\draw[line width=4*\lw] (-2,0) .. controls (-2, 2) .. (2,2);
\draw[line width=4*\lw] (-2,0) .. controls (-2, -1) .. (2,-1);

\node[label=right:{$\gamma_{-2}$}] at (4,-2) {};
\draw[line width=2*\lw] (3,-2)--(4,-2);
\draw[line width=2*\lw] (3,-3)--(4,-3);
\draw[line width=2*\lw] (2,1.5)  .. controls (2.5, 1.5) and  (2.5, -2) .. (3,-2);
\draw[line width=2*\lw] (2,-1.5)  .. controls (2.5, -1.5) and  (2.5, -3) .. (3,-3);  
\draw[line width=2*\lw] (-2,0) .. controls (-2, 1.5) .. (2,1.5);
\draw[line width=2*\lw] (-2,0) .. controls (-2, -1.5) .. (2,-1.5);

\node[label=right:{$\gamma_{-3}$}] at (4,-4) {};
\draw[line width=\lw, densely dotted] (3,-4)--(4,-4);
\draw[line width=\lw, densely dotted] (3,-5)--(4,-5);
\draw[line width=\lw, densely dotted] (2,1.333)  .. controls (2.5, 1.333) and  (2.5, -4) .. (3,-4);
\draw[line width=\lw, densely dotted] (2,-1.667)  .. controls (2.5, -1.667) and  (2.5, -5) .. (3,-5);  
\draw[line width=\lw, densely dotted] (-2,0) .. controls (-2, 1.333) .. (2,1.333);
\draw[line width=\lw, densely dotted] (-2,0) .. controls (-2, -1.667) .. (2,-1.667);
\end{scope}
\end{tikzpicture}
%\label{fig:curves_for_cap_functor}
\caption{\label{fig:curves_for_cap_functor}
}

\end{figure}
Let $\gamma_i$ be a sequence of arcs as in Figure \ref{fig:curves_for_cap_functor}, indexed by negative integers $i$. More precisely, we first list the properties which only involve one curve at a time:
\begin{enumerate}
\item $\gamma_i$ agrees in $[3,+\infty) \times \bR$ with horizontal lines at heights $2i+1$ and $2i+2$.  
\item $\gamma_i$ intersects $[2,3] \times \bR$ in two components, which are graphs of monotonically decreasing functions with values $1-1/i$ and $-2-1/i$ at $2$.
\item $\gamma_i$ is disjoint from $ (-1,2) \times ( -1, 1 ) $ .
\end{enumerate}
The next conditions are required for all pairs $i > j$:
\begin{enumerate}
\item $\gamma_i$ is transverse to $\gamma_j$.
\item The intersection of the pair $(\gamma_i, \gamma_j)$ with $[2,+\infty) \times \bR$ is monotone (c.f. Definition \ref{def:monotone_pair}.)
\item $\gamma_i \cap \gamma_j  \cap  (-\infty,2] \times \bR = \{ q \}$.
\end{enumerate}

To define a functor, we set $\scrD(L,i)$ to be the parallel transport of $L^i$ along the curve $\gamma_i$, parametrised monotonically so that $t=0$ maps to the point $q$. If $f_{L^i}$ is the Morse function used to define the self-Floer cochains of $L^i$, the function $  f_{L^{i}} + t^2$ is Morse on $  $ 
$ \scrD(L,i)$, so we obtain an identification
\begin{equation}
  CF^*(L^i, L^i) = CF^*(\scrD(L,i), \scrD(L,i)),
\end{equation}
of self-Floer cochains. For pairs, we note that $\scrD(L,i)$ and $\scrD(K,j)$ are transverse if $i \neq j$, so we can define all $A_{\infty}$ operations among such Lagrangians without using inhomogeneous terms. We have an inclusion
\begin{equation}
CF^*(L^i, K^j) \subset CF^*(\scrD(L,i), \scrD(K,j)),
\end{equation}
corresponding to the intersection points lying over $q \in \bC$. Lemma \ref{lem:inner_subcomplex} implies that this is an inclusion of subcomplexes. Since $j<i$ implies that $ h(\scrD(K,j)) < h(\scrD(L,i))  $, we obtain an inclusion of morphisms in the directed categories:
\begin{equation}
  \scrO_{M_q}((L,i), (K,j) ) \subset  \scrO_{W}((L,i), (K,j).
\end{equation}

This inclusion yields a functor
\begin{equation}
\scrD \co \scrO_{M_q} \to \scrO_{W} 
\end{equation}
with trivial higher order terms by the generalisation of  Lemma \ref{lem:inner_subcomplex} to multiple Lagrangian boundary conditions. As quasi-units are defined in the same way on $\scrO_{M_q} $ and $\scrO_{W}  $ we obtain, by the universal property of localisation, a functor
 \begin{equation}
\scrD \co \scrF_{M_q} \to \scrF_{W}.
\end{equation}

\subsection{An equivalence of bimodules}
\label{sec:adjunction}
The Orlov functor and the restriction bimodule can be compared whenever $p=q$ lies in the right half plane. For concreteness, we set $ q = (1,1)$. In this case, note that $ \scrD(L,i) $ and $\scrV(L,i)$ meet cleanly along a copy of $L^i$ over $q$, as shown in Figure \ref{fig:curves_for_adjunction}. Given any horizontally admissible Lagrangian $T$, the count of holomorphic polygons with corners mapping to this clean intersection defines a map
\begin{multline}
CF^*(\scrD(L_{r-1}, i_{r-1}), \scrD(L_r,i_r))  \otimes \cdots \otimes CF^*(\scrD(L_0, i_0), \scrD(L_1,i_1)) \\ \otimes CF^*(T_{s}, \scrD(L_0,i_0)) 
\otimes CF^*(T_{s-1},T_s) \otimes \cdots \otimes CF^*(T_{0}, T_{1})   \longrightarrow CF^*(T_{0}, \scrV(L,i)).
\end{multline}
Using the inclusion of morphism spaces in $\scrO_M$ as subcomplexes of morphism spaces among the images of these Lagrangians under $\scrD$, we obtain a map of $\scrO_M$-$\scrO_W$ bimodules
\begin{equation} \label{eq:adjunction_map_of_bimodules}
  _{\scrD} \Delta^{\scrO_W} \to \scrR.
\end{equation}

\begin{figure}
\centering
\begin{tikzpicture}[scale =1.5]
  % scope environment restricts the region that is shown
\begin{scope}
\clip (-3,-2.5) rectangle (6,3);
%\node[label=left:{$q=(1,1)$}] at (1,1) {};
\draw[dashed] (3,-5)--(3,2);
%\node[label=right:{$x=3$}] at (3,1) {};
\draw[dashed] (2,-5)--(2,2);
\node[label=left:{$x=2$}] at (2,-4) {};

\fill[black!40!white] (-.5,0) circle (.25);
\node[label=right:{$T$}] at (4,0.5) {};
\draw[line width=4*\lw,black!40!white] (-.5,0)   .. controls (.5, .25)  ..(1,0.25) -- (4,.25);
\draw[line width=4*\lw,black!40!white] (-.5,0)   .. controls (.5, .5)  ..(1,0.5) -- (4,.5);
\draw[line width=4*\lw,black!40!white] (-.5,0)   .. controls (.5, .7)  ..(1,0.7) -- (4,.7);

\node[label=right:{$\scrD(L,-1)$}] at (4,0) {};
\draw[line width=4*\lw,blue] (3,0)--(4,0);
\draw[line width=4*\lw,blue] (3,-1)--(4,-1);
\draw[line width=4*\lw,blue] (2,2)  .. controls (2.5, 2) and  (2.5, 0) .. (3,0);
\draw[line width=4*\lw,blue] (2,-1)  .. controls (2.5, -1) and  (2.5, -1) .. (3,-1);  
\draw[line width=4*\lw,blue] (2,2)  .. controls (1.75, 2) and  (1.75, 1) .. (1,1);
\draw[line width=4*\lw,blue] (1,1)  .. controls (-1, 1) and  (-2, 1) .. (-2,0);  
\draw[line width=4*\lw,blue] (-2,0) .. controls (-2, -1) .. (2,-1);

\node[label=right:{$\scrD(L,-2)$}] at (4,-2) {};
\draw[line width=2*\lw,blue] (3,-2)--(4,-2);
\draw[line width=2*\lw,blue] (3,-3)--(4,-3);
\draw[line width=2*\lw,blue] (2,1.5)  .. controls (2.5, 1.5) and  (2.5, -2) .. (3,-2);
\draw[line width=2*\lw,blue] (2,-1.5)  .. controls (2.5, -1.5) and  (2.5, -3) .. (3,-3);  

\draw[line width=2*\lw,blue] (2,1.5) .. controls (1.5, 1.5) and (1.5,.5) .. (1,1);
\draw[line width=2*\lw,blue] (-2.5,0) .. controls (-2.5, 1.5) and (.5,1.5) .. (1,1);
\draw[line width=2*\lw,blue] (-2.5,0) .. controls (-2.5, -1.5) .. (2,-1.5);

\node[label=right:{$\scrD(L,-3)$}] at (4,-4) {};
\draw[line width=\lw,blue] (3,-4)--(4,-4);
\draw[line width=\lw,blue] (3,-5)--(4,-5);
\draw[line width=\lw,blue] (2,1.333)  .. controls (2.5, 1.333) and  (2.5, -4) .. (3,-4);
\draw[line width=\lw,blue] (2,-1.667)  .. controls (2.5, -1.667) and  (2.5, -5) .. (3,-5);  
\draw[line width=\lw,blue] (2,1.333)  .. controls (1.5, 1.333) and  (1.5, .333) .. (1,1);
\draw[line width=\lw,blue] (-2.667,0) .. controls (-2.667, 1.667) and (.5,1.667) .. (1,1);
\draw[line width=\lw,blue] (-2.667,0) .. controls (-2.667, -1.667) .. (2,-1.667);

  \node[label=left:{$\scrV(L,-1)$}] at (1,2.5) {};
\draw[line width=4*\lw, red] (1,-2)--(1,2.5);
\draw[line width=2*\lw, red] (.5,-2)  .. controls (.5, .5) .. (1, 1) .. controls (1.5,1.5) .. (1.5,2.5);
\node[label=above:{$\scrV(L,-2)$}] at (1.5,2.5) {};
\draw[line width=\lw,red ] (.333,-2) .. controls ( .333,.5) .. (1,1) .. controls (1.667,1.5) .. (1.667,2.5);

\node[label=110:{$q$}] at (1,1) {};

\end{scope}
\end{tikzpicture}
\caption{}
\label{fig:curves_for_adjunction}
\end{figure}

\begin{lem} \label{lem:adjunction_map_equivalence}
  Whenever $2i+2< h(T) $, the induced map
\begin{equation} 
\scrO_{W}(T, \scrD(L,i) )   \to \scrR(T,(L,i))
\end{equation}
is a chain equivalence.
\end{lem}
\begin{proof}
Up to equivalence, the map is invariant under isotopies of $\scrD(L,i) $ among horizontally admissible Lagrangians of fixed height, and all isotopies of $\scrV(L,i)$ among vertically admissible Lagrangians. By moving the intersection point $q$ so that its $x$ and $y$ coordinate are both much larger than $0$, we may assume that all intersections points of $T$ with $\scrD(L,i)$ and $\scrV(L,i)  $ occur along the ends of $T$, and all these ends have height smaller than the $y$-coordinate of $q$ (see Figure \ref{fig:curves_for_adjunction}). Each end of $T$ projects to an arc which intersects the arcs defining $\scrV(L,i)$ and $\scrD(L,i)$ once, in the second case along the path going from $q$ to the horizontal line of height $2i+2 $ by assumption. The result follows from a straightforward count of holomorphic triangles which are constant in the fibre. 
\end{proof}
\begin{cor} \label{cor:adjunction_exists}
$\scrD$ represents the bimodule $\scrR$, i.e. there is a natural quasi-equivalence of $\scrF_M$-$\scrF_W$-bimodules
\begin{equation}
_{\scrD} \Delta^{\scrF_W} \cong \scrR.
\end{equation} \qed
\end{cor}

%\begin{cor} \label{cor:adjunction_exists}
%Equation \eqref{eq:adjunction_map_of_bimodules} defines an adjunction between  the functor $\scrD$ and the bimodule  $\scrR$. \qed
%\end{cor}

\subsection{A two-step filtration}
\label{sec:two-step-filtration}

Consider the right-pullback $\scrW_{q}$ of $\scrR$ by the functor $\scrD$: this is a $\scrO_{M_p}$-$\scrO_{M_q}$-bimodule which assigns to $(L,i) \in \Ob\, \scrO_{M_p}$ and $(K,j) \in \Ob \,\scrO_{M_q}$ the complex
\begin{equation}
    \scrW_{q}( (K,j) ,  (L,i)  )  = \begin{cases}  CF^*( \scrD(K,j), \scrV(L,i ) ) & \textrm{ if }  i< j \\
0 & \textrm{ otherwise.} 
\end{cases}
\end{equation}
The bimodule structure maps arise from the identification of morphisms in $\scrO_{M_p}$ with the Floer complexes among the Lagrangians $ \scrV(L,i )  $, and the inclusion of morphisms in $\scrO_{M_q} $ as the subcomplex of  the Floer complexes among the Lagrangians $ \scrD(K,j) $ corresponding to the intersection points lying over $q$. 
\begin{figure}
\centering
\begin{tikzpicture}[scale =1.5]
% scope environment restricts the region that is shown
\begin{scope}
\node[label=left:{$q$}] at (-2,0) {};
\node[label=right:{$0$}] at (0,0) {};

\fill[black!40!white] (-.5,0) circle (.125);

\node[label=right:{$\scrD(K,-1)$}] at (2,2) {};
\draw[line width=4*\lw, blue] (-2,0) .. controls (-2, 2) .. (2,2);
\draw[line width=4*\lw, blue] (-2,0) .. controls (-2, -1) .. (2,-1);

\node[label=right:{$\scrD(K,-2)$}] at (2,1.5) {};
\draw[line width=2*\lw, blue] (-2,0) .. controls (-2, 1.5) .. (2,1.5);
\draw[line width=2*\lw, blue] (-2,0) .. controls (-2, -1.5) .. (2,-1.5);

%\node[label=right:{$\gamma_{-3}$}] at (2,1.333) {};
\draw[line width=\lw, blue] (-2,0) .. controls (-2, 1.333) .. (2,1.333);
\draw[line width=\lw, blue] (-2,0) .. controls (-2, -1.667) .. (2,-1.667);

  \node[label=left:{$\scrV(L,-1)$}] at (0,2.5) {};
\draw[line width=4*\lw] (0,-2)--(0,2.5);
\draw[line width=2*\lw] (-.5,-2)  .. controls (-.5, 0) and  (.5, 0) .. (.5, 2.5);
\node[label=above:{$\scrV(L,-2)$}] at (.5,2.5) {};
\draw[line width=\lw] (-.667,-2)  .. controls ( -.667,0) and  ( .667,0) .. ( .667,2.5);

\draw[line width = \lw, dashed] (2,0) .. controls (0,-.25) .. (-2,-2);
\end{scope}
\end{tikzpicture}
\caption{}
\label{fig:pullback_bimodule}
\end{figure}

The definition of $\scrW_q$ is summarised in Figure \ref{fig:pullback_bimodule}: since the generators of all these complexes, as well as the intersection points among the Lagrangians corresponding to $\scrD$ and $\scrV$, take place in the region $x < 2$, we have only illustrated this subset of the base.

The intersection points between  $\scrD(K,j)$  and $ \scrV(L,i )$ which occur over points in the lower half-plane generate a submodule which we denote $\scrW^{+}_q ( (K,j) ,  (L,i)  ) $.
\begin{lem}
  The differential preserves $\scrW_q^{+} ( (K,j) ,  (L,i)  )   $.
\end{lem}

%\marginpar{Don't follow this. Why not same proof as Lemma 7.7 for $CF_{in}$ ?}

\begin{proof}
We apply the same argument as in Lemma \ref{lem:inner_subcomplex} to the dashed line in Figure \ref{fig:pullback_bimodule}, which separates the intersection points in the lower half plane from the region where the fibration is not flat and which contains the intersection points in the upper half plane. 
%Assume by contradiction that the differential applied to $\scrW_q^{+} ( (K,j) ,  (L,i)  )  $   has a component not in this submodule. A pseudo-holomorphic map from $[0,1] \times \bR $ to $E$  contributing to the differential on $ \scrW_q( (K,j) ,  (L,i)  )   $ has the boundary arc $  \{ 0 \} \times \bR $ labelled by $\scrD(K,i) $, and converges at $+\infty$ to the intersection point with $ \scrV(L,j ) $ lying in the upper half plane, and at $-\infty$ to the intersection point in the lower half plane. The signed count of inverse images of a point in the plane therefore decreases by one at $q$ when going radially towards infinity. Since the degree at infinity vanishes, this implies that there is an interior point near $q$ at which the degree is negative, contradicting holomorphicity near $q$.
\end{proof}

Generalising the above argument to the case of holomorphic polygons, we find a sub-bimodule $ \scrW_q^{+}  \subset \scrW_q$, given on pairs by the subcomplex $ \scrW_q^{+} ( (K,j) ,  (L,i)  ) $.    We introduce the quotient complex 
\begin{equation}
\scrW^{+}_q \to \scrW_q \to \scrW_q^{-},
\end{equation}
whose value on a pair is the Floer complex $\scrW_q^{-} ((K,j) ,  (L,i))$ generated by all intersection points occurring in the lower half plane.

\subsection{Parallel transport and the exact triangle}\label{Sec:Orlov}
Let us return to the situation considered in the proof of Lemma \ref{lem:adjunction_map_equivalence}, where $p=q$ is a point in the upper right quadrant of the plane. In this case, all intersection points  defining the bimodule $\scrW_q^{+}$  lie over this point and hence so do all holomorphic discs contributing to the bimodule structure maps. Since the Lagrangians $\scrD(K,j)$  and $ \scrV(L,i )$ intersect this fibre along $K^{j} $ and $L^{i}$, the bimodule $\scrW_q^{+}$ therefore represents the diagonal bimodule of $\scrO_{M_q}$. 
\begin{lem}\label{lem:inclusion_is_unit}
The inclusion $ \scrW_q^{+} \to \scrW_q $ agrees with the unit $ \scrO_{M_q} \to {}_{\scrD}\Delta^{\scrO_{W}}_{\scrD}  $, i.e. the following diagram commutes:
\begin{equation}
  \xymatrix{  \scrO_{M_q} \ar[d] \ar[r] &  \scrW_q^{+} \ar[d] \\ 
_{\scrD}\Delta^{\scrO_{W}}_{\scrD} \ar[r] & \scrW_q \\
}
\end{equation}
\end{lem}
\begin{proof}
The bottom arrow is induced (by pullback under $\scrD$) by the equivalence between $  {}_{\scrD} \Delta^{\scrO_W} \cong \scrR $. Commutativity is implied by the fact that all holomorphic curve counts take place over $q$, as in Lemma \ref{lem:adjunction_map_equivalence}.
\end{proof}

In order to compute $\scrW_q^{-}$, we move $q$ along a path from $(1,1)$ to $(1,-1)$ that moves counterclockwise (e.g. going through $(-2,0)$). Having fixed the isomorphism between the categories $\scrO_{M_{(1,1)}}  $ and  $\scrO_{M_{(1,-1)}}  $ arising from parallel transport in the right half-plane, the above path induces the monodromy symplectomorphism between these fibres; and hence acts correspondingly on Fukaya categories. On the other hand, setting $p=q=(1,-1)$ we have that $  \scrW_{(1,-1)}^{-}$ is isomorphic to diagonal bimodule. Pulling back again to a point in the upper right quadrant, we conclude 
\begin{lem}
The bimodule $\scrW_q^{-}$ is quasi-isomorphic to the graph bimodule of the clockwise monodromy. \qed
\end{lem}
Combining the triangle $\scrW_q^{+} \to \scrW_q \to \scrW_q^{-}$ with Lemma \ref{lem:inclusion_is_unit}, we obtain Theorem \ref{thm:les}, namely:
\begin{cor} \label{cor:twist_cone_adjunction-bimod}
The graph bimodule $\Delta_{\phi}$ is quasi-isomorphic to the cone  $  \Delta^{\scrF_{M}} \to {}_{\scrD}\Delta^{\scrF_{W}}_{\scrD}  $ of the unit. \qed
\end{cor}

\subsection{Lefschetz fibrations}
\label{sec:lefschetz-fibrations}

We now implement some of the above ideas in the setting of a Lefschetz fibration $W \co E \to \bC$ on a total space of dimension $2n$, with a unique critical point (by convention, we can set the value to be $-1/2$ to remain consistent with the previous section). Fix the thimble $T$ which projects to the subset $[-1/2,+\infty)$ of the real axis.  As an object of $\scrF_W$, a thimble $T$ has self-Floer cochains generated by  the critical points of any proper Morse function on $T \cong \bR^{n}$. We fix such a Morse function $f_{\bR^{n}}$ with a unique minimum, which near the unit sphere is the sum of a Morse function on the sphere with $(\|W \|-1)^2 $. Denoting by $\bk$ the category with one object whose endomorphism group is $\bk$, we obtain a functor
\begin{equation}
 \scrT \co  \bk  \to   \scrF_W,
\end{equation}
which is a fully faithful embedding since the class of the minimum maps to the identity, which generates the self-Floer cohomology of a thimble. We abuse notation and write $\scrT$ either for the functor, or for the corresponding object of $\scrF_W$.

Let $K \subset M = \pi^{-1}(1)$ denote the vanishing cycle. Consider the Yoneda module $\scrK$ of this object of $\scrF_M$; note this $\scrF_M$ module can be equivalently thought of as a $\bk-\scrF_M$ bimodule.
\begin{Lemma} \label{lem:pullback_restriction}
The pullback of $\scrR$ by $\scrT$ is quasi-isomorphic to $\scrK$. 
% is quasi-equivalent to the functor $K$. \qed
\end{Lemma}
\begin{proof}
The Yoneda module assigns to any Lagrangian $L$ the Floer complex with $K$, computed in the fibre $M$, whereas the module ${}_{\scrT} \scrR$ is obtained by taking the Floer complex of the vertically admissible Lagrangian associated to $L$ with the thimble $\scrT$. For $L \neq K$, these agree by the maximum principle, whereas for $L = K$ they agree because our choice of Morse function on $\scrT$ restricts to a Morse function on $K$, which we use to define the self-Floer complex of $K$.
\end{proof}

Lemma \ref{lem:pullback_restriction} explains the nomenclature \emph{restriction bimodule}: there is a basic link between $\scrR$ and the geometric process of ``restricting" a Lefschetz thimble to a generic fibre. We next consider the functor $\scrD$.
%Let $T_{\gamma}$ the object of $\scrF_W$ obtained by transport of $K$ by a curve 
\begin{Lemma} \label{lem:splitting_transport_vanishing}
There is a quasi-isomorphism $\scrD K \cong \scrT \oplus \scrT[n-1]$.
\end{Lemma}
\begin{proof}
Let $ \scrD ( K)$ be a representative of the image of the Orlov functor with ends at $\pm 1$, and  $\scrT_{\pm} $ be representatives of $\scrT $ in $\scrO_{W}$ with ends at $\pm 2$ (see Figure \ref{fig:direct_sum_orlov}). Note that the product
\begin{equation} \label{eq:product_to_K}
  HF^*(\scrD ( K ), \scrT_{-} )  \otimes HF^*(\scrT_{+},\scrD ( K) ) \longrightarrow HF^*(\scrT_{+}  , \scrT_{-}  ) \cong H^*(K)
\end{equation}
is in the correct order for computing morphisms in $\scrF_{W}$. The existence of an embedding  of two shifted copies of $\scrT $ as summands in  $  \scrD ( K)$ is equivalent to the existence of classes $p_0, p_1  \in   HF^*(\scrD (K) , \scrT_{-} )  $  and $ \iota_0, \iota_1  \in HF^*(\scrT_{+},\scrD ( K))   $, whose products satisfy
\begin{equation}
  \mu_{2} ( p_j, \iota_i) = \delta_{i,j} ,
\end{equation}
since, up to shifts, this yields (mutually orthogonal) projections and inclusions of $ \scrT_{\pm}  \to \scrD (K) $.
\begin{figure}
\centering
\begin{tikzpicture}[scale =1]
% scope environment restricts the region that is shown
\begin{scope}

\node[label=left:{$\scrT_+ $}] at (-3,1) {};
\draw[line width=\lw, red] (0,0) .. controls (-1,0) and (-3,0) .. (-3,1) ..  controls (-3,2) and (0,2) .. (4,2);
\node[label=left:{$\scrT_- $}] at (-3,-1) {};
\draw[line width=\lw, blue] (0,0) .. controls (-1,0) and (-3,0) .. (-3,-1) ..  controls (-3,-2) and (0,-2) .. (4,-2);
\draw[radius = 8*\lw, fill] (0,0) circle ;

\node[label=below:{$ \scrD \left( K \right)$}] at (2,1) {};
\draw[line width=4*\lw] (-1,0) .. controls (-1, 1) .. (4,1);
\draw[line width=4*\lw] (-1,0) .. controls (-1, -1) .. (4,-1);

\end{scope}
\end{tikzpicture}
\caption{}
\label{fig:direct_sum_orlov}
\end{figure}
By deforming the Lefschetz fibration if necessary, it suffices to prove the result for the standard Lefschetz fibration $\pi: \bC^n \rightarrow \bC$, $\pi: (z_1, \ldots , z_n) \mapsto \sum z_j^2$.  The Lagrangians $ \scrT_\pm$ project to the arcs shown in Figure \ref{fig:direct_sum_orlov}, hence in particular agree with small perturbations of $i \bR^{n}$ in a neighbourhood of the origin in $\bC^{n}$.  We can therefore perform the Floer-theoretic computation in $T^* \bR^{n}$, in which case we have a natural isomorphism $  HF^*(\scrT_+, \scrT_-) \cong H^*_{c}(\bR^{n}) $. Using the fact that $\scrD S^{n-1}$ meets $i \bR^{n}$ cleanly along $S^{n-1}$, we have a commutative diagram
\begin{equation}
 \xymatrix{ HF^*(\scrD S^{n-1}, \scrT_-)  \otimes HF^*(\scrT_+, \scrD S^{n-1} ) \ar[r]  \ar[d] &  HF^*(\scrT_+, \scrT_-) \ar[d] \\
H^*(S^{n-1}) \otimes H^*(\bR^n, \bR^n \setminus S^{n-1}) \ar[r] &  H^*_{c}(\bR^{n}).},
\end{equation}
(If $U$ is an open neighbourhood of $S^{n-1} \subset i\bR^n$, the plumbing model for clean intersections is usually stated with morphism groups $C^*(U)$ and $C^*(U,\partial U)$, cf. Proposition  \ref{Prop:plumb}; we have used excision to identify $H^*(U,\partial U) \cong H^*(\bR^n, \bR^n\backslash S^{n-1})$.)  The classes $p_{0}$ and $p_1$ can now be chosen to be generators of the two non-zero graded components of $HF^*(\scrD S^{n-1},\scrT_- ) \cong H^*(S^{n-1})  $, with $\iota_0 $ and $\iota_1$ their Alexander-Lefschetz duals.  
\end{proof}

Consider now a closed exact Lagrangian  $L \subset M$:
\begin{Lemma}  \label{lem:lag_transverse_vanishing_one}
If $L$ meets $K$ at a single point, there is a quasi-isomorphic $\scrD L \cong \scrT$.  
\end{Lemma}
\begin{proof}
Polterovich proved that, if Lagrangian submanifolds $ K, L \subset M$ meet transversely at a single point, there is a Lagrangian cobordism $\Gamma \subset M\times \bC$ between $K \amalg L$ and the Lagrange surgery $K\# L$, fibred over a tripod (figure $Y$, thickened at the vertex) with the surgery lying over a small neighbourhood of the trivalent vertex and the Lagrangians $K,L, K\# L$ lying over the three ends (see \cite{pol,BC:Cob}).   Supposing further that $K \cong S^n \subset M$ is a Lagrangian sphere, this cobordism can be extended inside the Lefschetz fibration $E$ with fibre $M$ and vanishing cycle $K$ by continuing the edge labelled by $K$ into the critical point (this was also used by Biran and Cornea \cite{BC:Lef}). This yields a Lagrangian $T_L \subset E$, which after Hamiltonian isotopy is horizontally admissible with ends fibred by copies of $L$ and the monodromy image $L\# K = \tau_K(L)$. The Lemma will follow from the claim that $L_\gamma$ is Hamiltonian isotopic to $T_L$, together with Lemma \ref{lem:weakly_one_ended}, which asserts the equivalence of $T_L$ with the thimble $\scrT$.

\begin{figure}
\centering
\begin{tikzpicture}[scale =1]
% scope environment restricts the region that is shown
\begin{scope}
\clip (2,-.25) rectangle (2.5,.25);
\draw[line width=4*\lw, red, fill]  (2,0) .. controls (3,0) and (3,1) .. (4,1) --  (4,-1) ..  controls (3,-1) and (3,0) .. (2,0) ;
\end{scope}
\begin{scope}
\draw[line width=4*\lw, red] (0,0) -- (2,0) .. controls (3,0) and (3,1) .. (4,1);
\draw[line width=4*\lw, red] (2,0) .. controls (3,0) and (3,-1) .. (4,-1);
\node[label=left:{$T_L$}] at (0,0) {};
\draw[radius = 8*\lw, fill] (0,0) circle ;

\node[label=left:{$L_{\gamma}$}] at (-2,0) {};
\draw[line width=4*\lw, blue] (-2,0) .. controls (-2, 1) .. (4,1);
\draw[line width=4*\lw, blue] (-2,0) .. controls (-2, -1) .. (4,-1);

\end{scope}
\end{tikzpicture}
\caption{
\label{fig:lambda-lemma}}
\end{figure}

The Hamiltonian isotopy is a consequence of the $\lambda$-Lemma.  Given a hyperbolic critical point $(x,y)=(0,0) \in \bR^k\times \bR^m$ of a flow $(\phi_t)$, with local unstable manifold $W^u = \{0\}\times \bR^m$ and stable manifold $W^s = \bR^k \times \{0\}$, the $\lambda$-Lemma (see e.g. \cite{Deng})  %\cite[Theorem 25]{Banyaga} for the Morse-Bott case) 
asserts that if $\Delta$ is an $m$-disc transverse to $W^s$, then its image under $\phi_R$, for sufficiently large $R\gg 0$, is the graph of a function $\phi: W^u \rightarrow \bR$ with values and derivatives bound by $e^{\lambda T}$.  (Schematically, although the given flow takes exponential time into the origin, one can consider the flow associated to a system whose critical point is shifted slightly, and then reparametrise the flow-lines to have uniformly bounded time.)  In the Lefschetz fibration, apply the Hamiltonian flow of $\mathrm{Im}(\pi)$, which is the gradient flow of $\mathrm{Re}(\pi)$. This flows the complement of an open  neighbourhood of $-1 \in \gamma$ into the right half-plane, and the $\lambda$-lemma implies that the resulting fibred Lagrangian with boundary can be completed to a piecewise-smooth Lagrangian submanifold which contains a compact subset of the Lefschetz thimble, and which after smoothing is Hamiltonian isotopic to $L_{\gamma}$. That smoothing  yields the surgery $T_L $ presented in fibred position, cf. Figure \ref{fig:lambda-lemma}.
\end{proof}

\subsection{The fibred twist}
\label{sec:fibred-twist}

It would be natural, next, to consider the generalisation of the discussion of the previous section to the Morse-Bott case; for exact Morse-Bott fibrations with  globally defined parallel transport that is relatively straightforward. However, to obtain a theory  applicable to the geometric setting relevant to symplectic Khovanov cohomology with minimal further technical development, we consider the following setting introduced at the beginning of Appendix \ref{Sec:LES_of_twist}.  Namely, let $M$ be a Liouville domain, $W \co E \to \bC $ an exact Lefschetz fibration with a unique critical point and fibre $N$, and $W' \co E' \to \bC $ an exact Landau-Ginzburg model with fibre $M'$. Assume there is an exact inclusion
\begin{equation}
 M \times E \subset E'
\end{equation}
compatible with the maps to $\bC$. We now generalise the construction of the previous section to this setting. First, associated to each Lagrangian $L \subset M$, we obtain a thimble $\scrT L$ which is an object of $\scrF_{W'}$, by taking the product with the thimble $\scrT$ of $E$ which projects to the real axis.
\begin{lem} \label{lem:thimble_embed}
The assignment $L \to \scrT L$ extends to a fully faithful embedding
\begin{equation}
\scrT \co  \scrF_M  \to   \scrF_{W'}.
\end{equation}
\end{lem}
\begin{proof}
We fix the Morse $ f_{\bR^n} $ on $\bR^{n}$ used in the previous section, with a unique minimum on the unit sphere. This induces an inclusion
\begin{equation}
  CF^*(L_0, L_1) \subset  CF^*(L_0 \times K, L_1 \times K),
\end{equation}
which defines the $A_{\infty}$-homomorphism $\scrT$. As in Proposition \ref{prop:regularity-good-order}, holomorphic curves in $M$ which are regular are regular as curves in $M \times E$, hence in $E'$.
%The Lagrangian $\scrT L$ is  diffeomorphic to the product $L \times \bR^n$. We choose a Morse function on this space of the form $f_L + f_{\bR^n} $, where .    The morphism group $\scrO_{W_i}(T_L, T_L) = CM^*(f_{T_L}) = CM^*(f_L)$. All critical points lie over $0\in\bC$, and by maximum principle in $M \times E$ and exactness of the embedding in $E'$, all holomorphic polygons contributing to the $A_{\infty}$-operations amongst any finite collection of submanifolds  in the image of $\scrT$ are also concentrated in the $0$-fibre.    
\end{proof}
Let $K \subset N = W^{-1}(1)$ denote the vanishing cycle as before. As in Section \ref{Sec:BimodulesWithoutQuilts}, we assign to $K$ an $\scrF_{M'}$-$\scrF_{M}$-bimodule $\scrK$ by considering Floer theory in $M'$. The proof of the next result is a straightforward generalisation of the proofs of Lemmas  \ref{lem:pullback_restriction} and \ref{lem:splitting_transport_vanishing}, using the fact that products of regular holomorphic curves are regular. 

\begin{Lemma} \label{lem:basic_quasi-isomorphisms-MB} In the setting and notation of Lemma \ref{lem:thimble_embed}:
  \begin{enumerate}
  \item The pullback of the bimodule $\scrR$ by $\scrT$ is quasi-isomorphic to $\scrK$.
\item For each Lagrangian $L \subset M$, there is a quasi-isomorphism
\[
  \scrD \left( L \times K \right)  \cong \scrT  L \oplus \scrT L [n-1].
\]
  \end{enumerate} \qed
\end{Lemma}

Consider now a closed exact Lagrangian  $ L' \subset M'$, and a Liouville domain $N_{L'}$ equipped with a Liouville inclusion $N_{L'} \times M \subset M'$, containing the image of $L'$. Moreover, we assume that $N \subset N_{L'}$, and that the inclusion $E \times M \subset E'$ extends to an inclusion
\begin{equation}
  E_{L'} \times M \subset E'
\end{equation}
where $E_{L'} \to \bC$ is a Lefschetz fibration with fibre $N_{L'}$ all of whose critical points are contained in $E$.

\begin{lem} \label{lem:Morse-bott-lambda}
If $L'$ meets $K \times L$ cleanly along a section of the projection to $K$, there is a quasi-isomorphic $\scrD L' \cong \scrT L$. 
\end{lem}
\begin{proof}
This is the Morse-Bott version of Lemma \ref{lem:lag_transverse_vanishing_one}; the construction of the sugery and the cobordism takes place in the product of the Lefschetz fibration $E_{L'}$ with $M$, in which the gradient flow of the real part of $W$ is integrable. The result follows by the  Morse-Bott case of the $\lambda$-Lemma (see e.g. the proof of \cite[Theorem 25]{Banyaga}).
\end{proof}

To conclude, we prove the version of the exact sequence of a twist that we use in the main part of the paper:
\begin{proof}[Proof of Proposition \ref{prop:what_we_need_of_the_twist}]
By assumption, every object of $\scrA'$ either satisfies the hypothesis of Lemma \ref{lem:lag_transverse_vanishing_one} or is a product of a Lagrangian in $M$ with $K$. We conclude that the functor $\scrD \co \scrA' \to \scrF_W$  (composed with the inclusion $ \scrF_W \to  \Tw \scrF_W$) is equivalent to a functor which factors through the image of $\scrT$. Since $\scrT$ is a fully faithful embedding, we may therefore fix a functor $\cap \co \Tw \scrA' \to \Tw \scrA$ so that we have a diagram which commutes up to equivalence:
\begin{equation}
  \xymatrix{\Tw  \scrA' \ar[r]^{\cap} \ar[dr]^{\scrD} & \Tw \scrA   \ar[d]^{\scrT} \\
& \Tw \scrF_{W}.}
\end{equation}
Theorem \ref{thm:les} therefore implies that the graph bimodule $\Delta^{\Tw A}_{\phi}  $ is quasi-isomorphic to the cone of the unit $ \Delta^{\Tw \scrA} \to  {_{\cap} \Delta^{\Tw  \scrA } _{\cap}} $. It remains to show that we have an equivalence
\begin{equation} \label{eq:statement_adjunction_cup_cap}
  {_{\cap} \Delta^{\Tw  \scrA } _{\cap}} \cong  \Delta^{\Tw \scrA'}_{\cup \circ \cap  },
\end{equation}
i.e. that $\cap$ is adjoint to $\cup$, where $\cup$ is the  K\"unneth-type functor representing $\scrK$. Observe that the two-sided pullback of $ \Delta^{\scrF_W} $ by $\scrT$ is equivalent to $\Delta^{\scrA}$ because $\scrT$ is a fully faithful embedding, so   $ _{\cap} \Delta^{\scrA} $ is equivalent to $ _{\scrD} \Delta^{\scrF_W}_{\scrT}$. Corollary \ref{cor:adjunction_exists} implies that $\scrR_{\scrT} $ is represented by $ _{\scrD} \Delta^{\scrF_W}_{\scrT}$, but Lemma \ref{lem:basic_quasi-isomorphisms-MB} implies that $\cup$ is represented by $\scrR_{\scrT}$, which proves the existence of the adjunction, hence establishes the equivalence in Equation \eqref{eq:statement_adjunction_cup_cap}.
\end{proof}

%%% Local Variables: 
%%% mode: latex
%%% TeX-master: "Kh=Floer-Dec-2017"
%%% End: 

%%%%%%%%%%%%%%%%%%%%%%%%%%%%%%%%%%%%%%%
%%%%%%%%%%%%%%%%%%%%%%%%%%%%%%%%%%%%%%%%
%%%%%%%%%%%%%%%%%%%%%%%%%%%%%%%%%%%%%%%

\end{document}